\documentclass[11pt, reqno]{amsart}
\usepackage{amssymb}
\usepackage{amsmath}
\usepackage{amsthm}
\usepackage{braket}
\usepackage{pdfpages}
\usepackage{graphicx}
\usepackage{caption}
\usepackage{subcaption}
\usepackage{mathrsfs} 
\usepackage{tikz-cd} 
\usepackage{bm}
\usepackage{enumitem}
\usepackage{mathtools}
\usepackage{pgfplots}
\usepackage{import}
\usepackage{xifthen}
\usepackage{transparent}

\usepackage{hyperref}
\numberwithin{equation}{section}
\theoremstyle{plain}
\newtheorem{theorem}{Theorem}[section]

\theoremstyle{theorem}
\newtheorem{prop}[theorem]{Proposition}
\newtheorem{lem}[theorem]{Lemma}
\newtheorem{cor}[theorem]{Corollary}

\newtheorem*{question*}{Question}

\theoremstyle{definition}
\newtheorem{defn}[theorem]{Definition}

\newtheorem*{remark}{Remark}

\newcommand{\C}{\mathbb{C}}

\newcommand{\Z}{\mathbb{Z}}

\newcommand{\N}{\mathbb{N}}

\DeclareMathOperator{\Rat}{Rat}

\DeclareMathOperator{\PSL}{PSL}

\numberwithin{figure}{section}

\title{On hyperbolic rational maps with finitely connected Fatou sets}
\author{Yusheng Luo}
\address{Dept. of Mathematics \& University of Michigan, Ann Arbor, MI 48109 USA}
\email{yusheng.s.luo@gmail.com}
\date{\today}
\begin{document}

\begin{abstract}
In this paper, we study hyperbolic rational maps with finitely connected Fatou sets.
We construct models of {\em post-critically finite hyperbolic tree mapping schemes} for such maps, generalizing post-critically finite rational maps in the case of connected Julia set.
We show they are general limits of rational maps as we quasiconformally stretch the dynamics.
Conversely, we use quasiconformal surgery to show any post-critically finite hyperbolic tree mapping scheme arises as such a limit.

We construct abundant examples thanks to the flexibilities of the models, and use them to construct a sequence of rational maps of a fixed degree with infinitely many non-monomial rescaling limits.
\end{abstract}

\maketitle

\setcounter{tocdepth}{1}
\tableofcontents

\section{Introduction}
Let $f: \hat\C \longrightarrow \hat\C$ be a rational map of degree $d\geq 2$. It is said to be {\em hyperbolic} if all the critical points converge to attracting periodic cycles under iteration.
Hyperbolic maps form an open and conjecturally dense subset in the moduli space of all rational maps of degree $d$, and a connected component $\mathcal{H}$ is called a {\em hyperbolic component}.

When the Julia set is connected, hyperbolic components are classified by {\em post-critically finite} hyperbolic rational maps.
That is, given such a hyperbolic component $\mathcal{H}$, there exists a unique rational map $f_0 \in \mathcal{H}$, regarded as the {\em center} of $\mathcal{H}$, with finite critical orbits.
The map $f_0$ gives a model for the dynamics of hyperbolic rational maps $f\in \mathcal{H}$ in the sense that the dynamics of $f$ on its Julia set is quasiconformally conjugate to that of $f_0$, and it plays a crucial role in understanding the organization and topology of the corresponding hyperbolic component (\cite{McM88} and \cite{Milnor12}).

The situation is different for hyperbolic rational maps with disconnected Julia sets as such a post-critically finite model does not exist.
The aim of the paper is to propose an appropriate model for the class of hyperbolic rational maps such that every periodic Fatou component is a disk (equivalently, every Fatou component is finitely connected by Riemann-Hurwitz formula).

The object we consider in this general setting is a {\em tree mapping scheme} $(F, R)$, which combines discrete data of a {\em tree map} $F$ with analytic information of a {\em rational mapping scheme} $R$.
Roughly speaking, a tree mapping scheme is a mapping to a finite tree of Riemann spheres, from a sub-forest of such a tree, where the action on each sphere is a rational map.

Given a hyperbolic rational map $f \in \mathcal{H}$ with finitely connected Fatou set, we construct a {\em quasiconformal stretch} $f_n \in \mathcal{H}$ of $f$ and show that $f_n$ converges to an irreducible {\em post-critically finite hyperbolic}  tree mapping scheme $(F, R)$ (Theorem \ref{thm:cvg}).
The main tool for the construction of the tree map $F$ is a construction due to Shishikura in the late 80’s often called the Shishikura tree \cite{Shishikura89}, and the rational mapping scheme $R$ on the union of Riemann spheres is obtained by taking appropriate {\em rescaling limits} of $f_n$.

Our approach realizes tree mapping schemes as general limits of rational maps.
This is illustrative of the idea that the post-critically finite hyperbolic tree mapping scheme is, in a certain sense, a degenerate post-critically finite rational map, and it is interpreted as the {\em center at infinity} for the hyperbolic component $\mathcal{H}$ when the Julia set is disconnected.

Conversely, using a standard quasiconformal surgery, we show that any irreducible post-critically finite hyperbolic tree mapping scheme arises as such a limit (Theorem \ref{thm:classification}).

As an application, we use the tree mapping schemes to construct sequences with infinitely many non-monomial periodic rescaling limits, answering in the negative a question posed in \cite{Kiwi15} (see Theorem \ref{thm:infiniterescalinglimit}).

We now turn to a detailed statement of results.

\subsection*{Tree mapping schemes}
A {\em tree of Riemann spheres} $(\mathcal{T}_1, \hat\C^{\mathcal{V}_1})$ is a finite tree $\mathcal{T}_1$ with the vertex set $\mathcal{V}_1$, a disjoint union of Riemann spheres $\hat \C^{\mathcal{V}_1} := \bigcup_{a\in {\mathcal{V}_1}}\hat \C_{a}$, together with markings $\xi_a: T_a\mathcal{T}_1 \xhookrightarrow{} \hat\C_a$ for $a\in \mathcal{V}_1$.

Let $\mathcal{T}_0 \subseteq \mathcal{T}_1$ be a finite union of subtrees of $\mathcal{T}_1$ with vertices $\mathcal{V}_0 \subseteq \mathcal{V}_1$.
A {\em tree mapping scheme} $(F, R)$ from $(\mathcal{T}_0, \hat\C^{\mathcal{V}_0})$ to $(\mathcal{T}_1, \hat\C^{\mathcal{V}_1})$ is a {\em tree map}
$$
F: (\mathcal{T}_0, \mathcal{V}_0) \longrightarrow (\mathcal{T}_1, \mathcal{V}_1)
$$
which is injective on each edge and a {\em rational mapping scheme}
$$
R:= \bigcup_{a\in \mathcal{V}_0} R_a: \hat\C^{\mathcal{V}_0} \longrightarrow \hat\C^{\mathcal{V}_1}
$$ with
\begin{itemize}
\item $R_a= R_{a\to F(a)}: \hat\C_a \longrightarrow \hat \C_{F(a)}$ is a non-constant rational map; and
\item $R_a \circ \xi_a = \xi_{F(a)} \circ DF_a$,
\end{itemize}
where $DF_a: T_a\mathcal{T}_0 \longrightarrow T_{F(a)}\mathcal{T}_1$ is the tangent map of $F$ at $a$.
The set of {\em marked points} is denoted by $\Xi_a := \xi_a(T_a\mathcal{T}_1) \subseteq \hat\C_a$.

A sequence $f_n$ of rational maps is said to {\em converge} to $(F, R)$ if there exist {\em rescalings} $M_{a,n} \in \PSL_2(\C)$ {\em representing} $a \in \mathcal{V}_1$ such that 
$$
M_{F(a),n}^{-1} \circ f_n \circ M_{a,n}(z) \to R_a(z)
$$ 
compactly on $\hat \C_a- \Xi_a$ for all $a\in \mathcal{V}_0$ (see Definition \ref{defn:ctm}, cf. \S 4 in \cite{Arfeux17}).
This means that the rational maps $R_a$ are {\em rescaling limits} of $f_n$.

Let $f \in \mathcal{H}$ be a hyperbolic rational map with finitely connected Fatou set, i.e., each Fatou component of $f$ has finitely many boundary components. 
We construct a sequence of rational maps $f_n \in \mathcal{H}$ by quasiconformally stretching all Fatou components simultaneously.
We call such a sequence $f_n$ a {\em quasiconformal stretch} or simply a {\em stretch} for short (see \S \ref{sec:tms} for details).

When the Julia set is connected, this sequence $f_n$ converges to the corresponding post-critically finite rational map $f_0 \in \mathcal{H}$, which can be regarded as the tree mapping scheme $(F: \{a\} \longrightarrow\{a\}, R = f_0)$.
When the Julia set is disconnected, $f_n$ diverges in $\mathcal{M}_d = \Rat_d/\PSL_2(\C)$.
We show
\begin{theorem}\label{thm:cvg}
	Let $f \in \mathcal{H}$ be a hyperbolic rational map with finitely connected Fatou set. Then the sequence of stretch $f_n$ converges to an irreducible post-critically finite hyperbolic tree mapping scheme $(F, R)$.
\end{theorem}

\begin{remark}
	One of the main contributions of this paper is to give a good and explicit definition of {\em post-critically finite hyperbolic} tree mapping schemes, which is discussed in detail in \S \ref{sec:treemappingschemes} (see Definition \ref{defn:hpcf}).	
\end{remark}

\begin{figure}[ht]
    \centering
    \includegraphics[width=.7\linewidth]{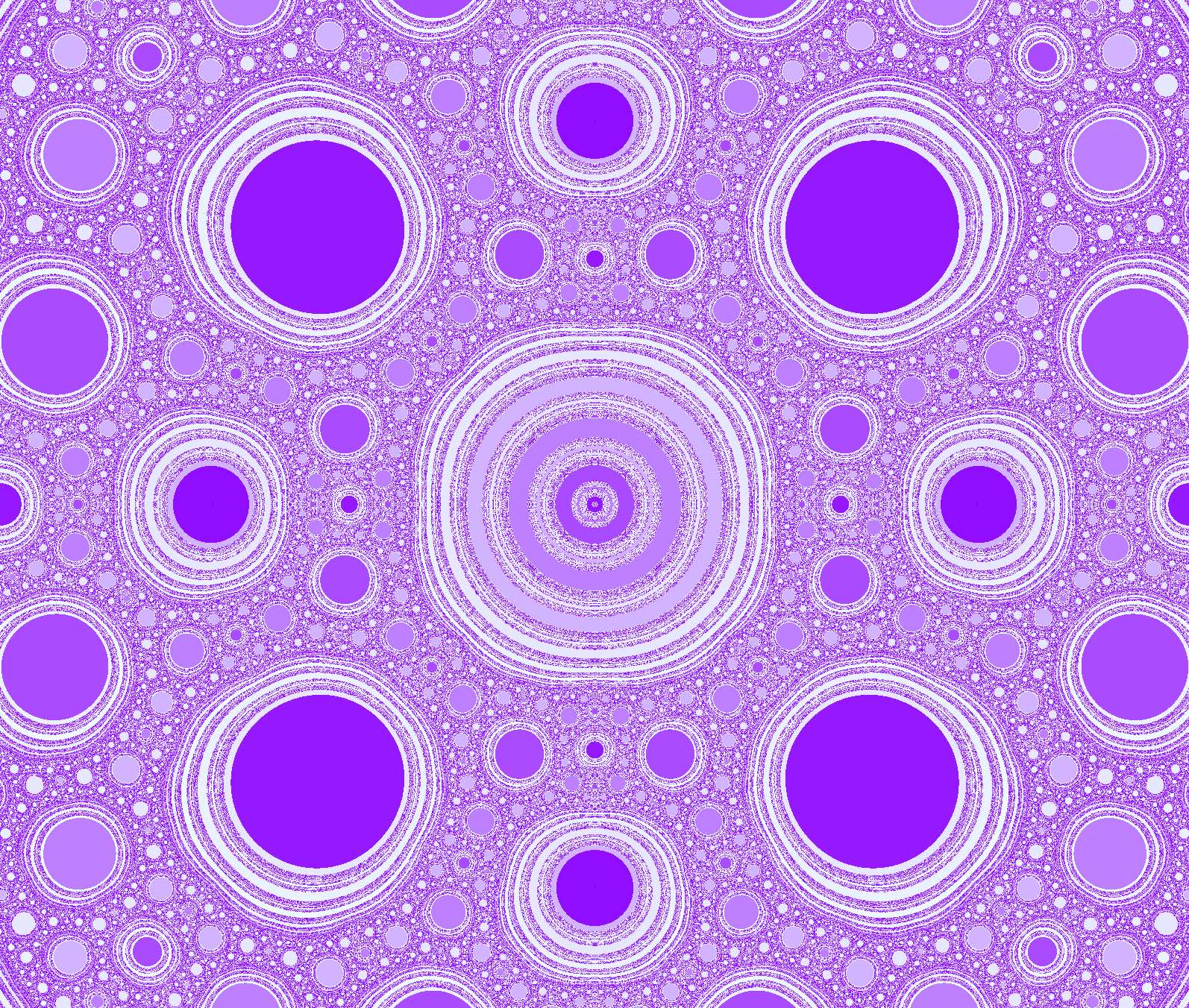}
    \caption{The Julia set of a hyperbolic rational map with finitely connected Fatou set. It contains a fixed Sierpinski carpet as a buried component. The algebraic formula is $f_n(z) = \frac{z^2}{1-\frac{1}{16}z^4}+\frac{1}{nz^4}$ for sufficiently large $n$. The tree mapping scheme is described in Figure \ref{fig:BuriedSierpinskiTreeMap}.}
    \label{fig:BuriedSierpinski}
\end{figure}

Using a standard quasiconformal surgery, we show every irreducible post-critically finite hyperbolic tree mapping scheme arises as such a limit:

\begin{theorem}\label{thm:classification}
Let $(F, R)$ be an irreducible post-critically finite hyperbolic tree mapping scheme. 
Then there exists a hyperbolic rational map $f$ whose quasiconformal stretch $f_n$ converges to $(F,R)$.

Moreover, if $g$ is another hyperbolic rational map corresponding to $(F, R)$, then $g$ is J-conjugate to $f$: there exists a quasiconformal map $\psi: \hat\C \longrightarrow\hat\C$ conjugating the dynamics of $f$ and $g$ on the Julia sets.
\end{theorem}

\begin{remark}
	When the Julia set is connected, $f$ and $g$ are J-conjugate if and only if $f$ and $g$ are in the same hyperbolic component.
	However, when the Julia set of $f$ is disconnected, the space of rational maps that are J-conjugate to $f$ can be disconnected (see Figure \ref{fig:JConjugate}).
	In some loose sense, the closure of such hyperbolic components `intersect' at infinity.
\end{remark}

\subsection*{Discussion on related work}
Our model is motivated by the McMullen family $f_n(z)=z^2+\frac{1}{nz^3}$.
In \cite{McM88}, McMullen showed that for sufficiently large $n$, the Julia set of $f_n$ is homeomorphic to a Cantor set of circles, and thus each Fatou component is either simply, or doubly connected.
It is already observed in \cite{McM88} that for large $n$, the dynamics on the Julia components are recorded by a `tent map' on an interval, which gives our tree map (see Figure \ref{fig:CantorCircleExample}).
Rescaling limits also appear explicitly in this example by choosing different scales at $0$ (see the discussion in \S \ref{sec:eg}).
This construction is generalized as {\em singular perturbation} or using {\em quasiconformal surgery}, and is studied broadly in the literature (see \cite{PT00, DM07, QYY15, WY17, WY20}).
They provide different perspectives of the theory: the singular perturbation is an algebraic construction that is related to rescaling limits; while quasiconformal surgery is analytic in nature. Our approach, in some sense, relates the two methods and gives an explicit characterization for all such maps.

The tree mapping scheme is an analogue of the Deligne-Mumford compactification in algebraic geometry \cite{DM69, HK14}.
The concept has appeared implicitly in Shishikura's work on Herman rings \cite{Shishikura87} and \cite{Shishikura89}.
Using quasiconformal surgery, Shishikura converts each Herman ring into two Siegel disks and uses it prove the famous Fatou-Shishikura inequality \cite{Shishikura87}.
In order for the construction to work, it is important to leave the realm of rational maps and consider maps on a finite union of Riemann spheres.
The configurations of these Riemann spheres give a tree like structure \cite{Shishikura89}.
This generalized map can also be constructed as a limit of some degenerating rational maps, similar as in this paper.
A similar notion of {\em dynamics on tree of spheres} is developed and studied by Arfeux in \cite{Arfeux16, Arfeux17}.

From a topological perspective,
the realization of tree mapping schemes is related to the Thurston's topological characterization of rational maps.
For a general {\em sub-hyperbolic semi-rational map} $f$, Cui and Tan \cite{CT11} constructed canonical decomposition and show that $f$ is equivalent to a rational map if and only if there are no Thurston's obstruction.
The dual of the canonical decomposition gives a natural tree map.
Recently, using this theory,  Cui and Peng \cite{CP19} classify the Thurston's obstruction for the tree map, and show a sub-hyperbolic semi-rational map is equivalent to a rational map if and only if there are no Thurston obstructions for the tree map and the local systems.
Thus, the first part of Theorem \ref{thm:classification} can be proved by applying the Thurston's theory in \cite{CP19}, and showing the stretch converges to the tree mapping scheme.
We choose to use quasiconformal surgery as it provides a concrete quasirational model and relates the tree mapping scheme with rescaling limits.
We remark that the realization result in \cite{CP19} works in a more general setting, but it is unclear how to find a natural model for a general hyperbolic component.
We also remark that the tree map there is different from the tree map considered in this paper.

The classification of the dynamics on a periodic Fatou component with infinitely many boundary components is related to the study of polynomials with disconnected Julia sets.
The classification problem for such polynomials has been studied extensively in a series of paper by DeMarco and Pilgrim \cite{DP11a, DP11b, DP17}.
An invariant for a general hyperbolic rational map might combine tree mapping schemes with these invariants.

The automorphism group of a post-critically finite rational map plays a crucial part in understanding the topology of the corresponding hyperbolic component (see \cite{McM88} and \cite{Milnor12}).
Recently, Wang and Yin studied the topology of the hyperbolic component for the McMullen family \cite{WY17}.
The post-critically finite hyperbolic tree mapping schemes generalize the roles of post-critically finite hyperbolic rational maps.
In a subsequence work \cite{L20},  we will use tree mapping schemes to understand the structure and organization of the corresponding hyperbolic components.

\begin{figure}[ht]
    \centering
    \resizebox{0.65\linewidth}{!}{
    \def\svgwidth{\columnwidth}
    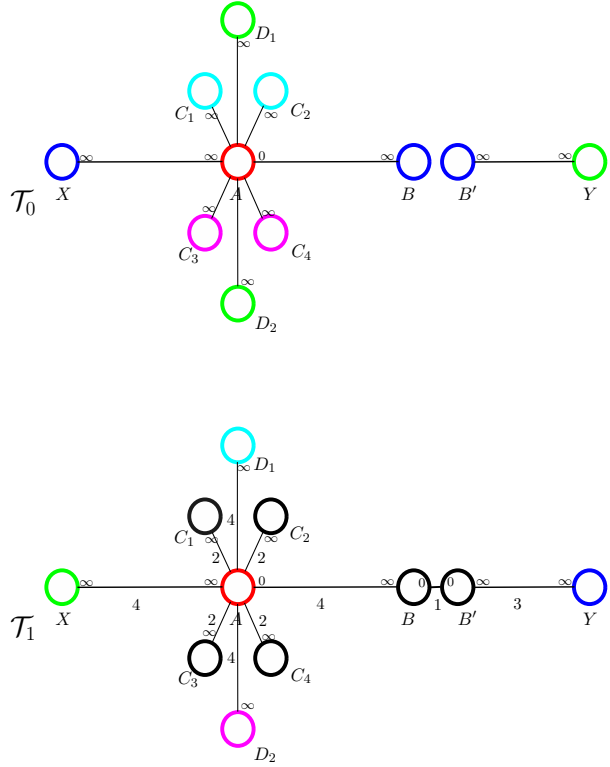

    }
    \caption{The tree mapping scheme associated to Figure \ref{fig:BuriedSierpinski}.
    The length of edges and the markings (except for the vertex $A$) are labeled in the figure. 
    The markings on $\hat\C_A$ correspond to $2\sqrt[4]{-1}$ and $2\sqrt{-1}$.
    The dynamics are recorded by colors and the rational mapping scheme is given as $R_{A\to A}(z) = \frac{z^2}{1-\frac{1}{16}z^4}$, $R_{X\to Y}(z) = R_{B\to Y}(z) = R_{C_i\to D_j}(z) = z^2$, $R_{D_j\to X}(z) = z$ and $R_{B'\to Y}(z) = R_{Y\to X}(z) = z^4$.}
    \label{fig:BuriedSierpinskiTreeMap}
\end{figure}

\subsection*{Application: cycles of rescaling limits.}
Let $f_n$ be a sequence of rational maps. 
A {\em period $p$ rescaling} for $f_n$ is a sequence $M_n \in \PSL_2(\C)$ so that $M_n^{-1} \circ f_n^p \circ M_n$ converges compactly away from a finite set to a non-constant rational map $g$.
The limiting map $g$ is called a {\em rescaling limit of period $p$}.

In \cite{Kiwi15}, Kiwi proved that for a sequence of rational maps $f_n$ of degree $d$, there can be at most $2d-2$ dynamically distinct cycles of rescaling limits that are not post-critically finite.
In the same paper, Kiwi asked if the number of cycles of non-monomial rescaling limits is finite.

Using tree mapping schemes, and a diagonal argument, we give the negative answer to the question:
\begin{theorem}\label{thm:infiniterescalinglimit}
For any $d\geq 3$, there is a sequence $f_n$ of rational maps of degree $d$ with infinitely many dynamically independent cycles of non-monomial rescaling limits.
\end{theorem}

We remark that Theorem \ref{thm:infiniterescalinglimit} improves a result in \cite{AC16} where it is shown that there can be an arbitrarily large number of dynamically independent cycles of non-monomial rescaling limits.
Our construction can be modified to produce a holomorphic family of rational maps of degree $d\geq 3$ with an arbitrarily large number of dynamically independent cycles of non-monomial rescaling limits.
However, whether a holomorphic family of rational maps can admit infinite non-monomial rescaling limits still remains open.

\subsection*{Structure of the paper}
The definitions of tree mapping schemes are given in \S \ref{sec:treemappingschemes}, followed by various examples in \S \ref{sec:eg}.
Theorem \ref{thm:cvg} is proved in \S \ref{sec:tms}.
The realization problem is discussed in \S \ref{sec:tmsr}, where Theorem \ref{thm:classification} is proved.
We study Julia components and cycles of rescaling limits in \S \ref{sec:JC} and \S \ref{sec:rl}.

We encourage the reader to first explore the examples in \S \ref{sec:eg} before going through the proofs.
The notion of post-critically finite hyperbolic tree mapping scheme should be intuitively clear after several examples.

\subsection*{Acknowledgment}
The author would like to thank Dzmitry Dudko, Curt McMullen for helpful discussions and useful advices; Laura DeMarco for bringing up the motivation question.
The author would like to thank the anonymous referees for useful comments and suggestions.

\section{Tree mapping schemes}\label{sec:treemappingschemes}
In this section, we give the definition of post-critically finite hyperbolic tree mapping schemes.

\subsection*{Tree mapping schemes}
A finite tree $\mathcal{T}$ is a non-empty, connected, finite $1$-dimensional simplicial complex without cycles.
The set of vertices of $\mathcal{T}$ will be denoted by $\mathcal{V}$, and the set of closed edges will be denoted by $\mathcal{E}$.
The edges adjacent to a given vertex $x\in \mathcal{V}$ form a finite set $T_x\mathcal{T}$ and will be called the {\em tangent space} of $\mathcal{T}$ at $x$,
whose cardinality $\nu(x)$ is the {\em valence} of $x$.
We will usually use $\vec{v}$ to represent an element in $T_x\mathcal{T}$ and $e_{\vec{v}}$ to denote the edge associated to $\vec{v}$.

A point $x\in \mathcal{V}$ is called an {\em end point} or a {\em branch point} if $\nu(x) = 1$ or $\nu(x)\geq 3$ respectively.
We will use $\partial \mathcal{T}\subseteq \mathcal{V}$ and $\mathcal{B}\subseteq \mathcal{V}$ to denote the set of end points and the set of branch points of $\mathcal{T}$.

In this paper, every finite tree $\mathcal{T}$ is equipped with a path metric $d$, i.e., a metric on $\mathcal{T}$ satisfying
$$
d(x,z) = d(x,y) + d(y,z)
$$
whenever $y$ lies on the unique arc connecting $x$ and $z$.
This metric is uniquely determined by the lengths $d(e)=d(a,b)$ it assigns to edges $e=[a,b]\in \mathcal{E}$.

A {\em tree of Riemann spheres} $(\mathcal{T}, \hat\C^\mathcal{V})$ is a finite tree $\mathcal{T}$ with the vertex set $\mathcal{V}$, a disjoint union of Riemann spheres $\hat \C^{\mathcal{V}} := \bigcup_{a\in {\mathcal{V}}}\hat \C_{a}$, together with markings $\xi_a: T_a\mathcal{T} \xhookrightarrow{} \hat\C_a$ for $a\in \mathcal{V}$.
The set of {\em marked points} is denoted by $\Xi_a:= \xi_a(T_a \mathcal{T}) \subseteq \hat\C_a$ and $\Xi := \bigcup_{a\in \mathcal{V}} \Xi_a$.
\begin{defn}
	Let $(\mathcal{T}_1, \hat\C^{\mathcal{V}_1})$ be a tree of Riemann spheres, and let $\mathcal{T}_0 \subseteq \mathcal{T}_1$ be a finite union of subtrees of $\mathcal{T}_1$ with vertices $\mathcal{V}_0 \subseteq \mathcal{V}_1$.
We define a {\em tree mapping scheme} $(F, R)$ from $(\mathcal{T}_0, \hat\C^{\mathcal{V}_0})$ to $(\mathcal{T}_1, \hat\C^{\mathcal{V}_1})$ as a map
$$
F: (\mathcal{T}_0, \mathcal{V}_0) \longrightarrow (\mathcal{T}_1, \mathcal{V}_1)
$$
that is injective on each edge and a union of maps 
$$
R:= \bigcup_{a\in \mathcal{V}_0} R_a : \hat\C^{\mathcal{V}_0} \longrightarrow \hat\C^{\mathcal{V}_1}
$$
with
\begin{itemize}
\item $R_a = R_{a\to F(a)}: \hat\C_a \longrightarrow \hat \C_{F(a)}$ is a non-constant rational map;
\item $R_a \circ \xi_a = \xi_{F(a)} \circ DF_a$,
\end{itemize}
where $DF_a: T_a\mathcal{T}_0 \longrightarrow T_{F(a)}\mathcal{T}_1$ is the tangent map.
We shall call the map $F$ the {\em tree map} and the rational map $R$ the {\em rational mapping scheme}.
\end{defn}
Note that a rational map $f$ can be regarded as the tree mapping scheme with tree map $F: \{a\} \longrightarrow \{a\}$ and rational mapping scheme $R = f$.

Whenever it is defined, we will denote
$$
R^k|_{\hat\C_{a_1}} = R_{a_{k-1}} \circ ... \circ R_{a_1}: \hat\C_{a_1} \longrightarrow \hat\C_{a_k}
$$
where $a_{i+1} = F(a_i)$.

Two tree mapping schemes $(F, R)$ and $(F', R')$ are {\em conjugate} if there exists a simplicial isomorphism
$$
\Phi: \mathcal{T}_{1} \longrightarrow \mathcal{T}'_{1},
$$
and a collection of homeomorphisms
$$
\Psi_a : \hat\C_{a} \longrightarrow \hat\C_{\Phi(a)}
$$
for $a\in \mathcal{V}_{1}$ that conjugate the dynamics:
$$
\Psi_{F(a)}\circ R_{a} = R'_{\Phi(a)} \circ \Psi_a
$$
for any $a\in \mathcal{V}_{0}$.

In the next few subsections, we introduce some natural conditions on the tree mapping schemes $(F, R)$, which will lead to the definition of post-critically finite hyperbolic tree mapping schemes.

\subsection*{Local degrees and piecewise linear tree map}
Let $(F, R)$ be a tree mapping scheme from $(\mathcal{T}_0, \hat\C^{\mathcal{V}_0})$ to $(\mathcal{T}_1, \hat\C^{\mathcal{V}_1})$.
We define the {\em local degree} of a vertex $a\in \mathcal{V}_0$ by 
$$
\deg(a) := \deg (R_a),
$$
as the degree of the corresponding rational map $R_a: \hat\C_a \longrightarrow\hat\C_{F(a)}$.

Let $e = [a,b]$ be an edge of $\mathcal{T}_0 \subseteq \mathcal{T}_1$.
Let $z_{a,e} \in \Xi_a \subseteq \hat\C_a$ and $z_{b,e} \in \Xi_b \subseteq \hat\C_b$ be the corresponding marked points.
We say $R$ has {\em compatible local degrees at $e$} if
$$
\deg_{z_{a,e}} R_a = \deg_{z_{b,e}} R_b,
$$
where $\deg_{z_{a,e}} R_a$ (or $\deg_{z_{b,e}} R_b$) is the local degree of the rational map $R_a$ (or $R_b$) at the point $z_{a,e}$ (or $z_{b,e}$).

The rational mapping scheme $R$ has {\em compatible local degrees} if it has compatible local degrees at every edge $e \in \mathcal{E}_0$.
This compatibility allows us to define the {\em local degree} at an edge $e = [a,b] \in \mathcal{E}_0$ by
$$
\deg(e) := \deg_{z_{a,e}} R_a = \deg_{z_{b,e}} R_b.
$$

Assume $R$ has compatible local degrees. 
We say the tree map $F: \mathcal{T}_0 \longrightarrow \mathcal{T}_1$ is {\em piecewise linear} compatible with the local degree if $F$ is linear with derivative $\deg(e)$ on each edge $e\in \mathcal{E}_0$. 

\subsection*{Exposed points and exposed critical points.}
A point $z\in \hat\C^{\mathcal{V}_1}$ is called {\em exposed} if $z\notin \Xi$.
In other words, $z$ does {\em not} correspond to any tangent direction $\vec{v}\in T_a\mathcal{T}_1$ for any $a\in \mathcal{V}_1$.
A point $z$ is said to have an exposed orbit if $z, R(z), R^2(z)...$, whenever defined, are all exposed.

A critical point $z\in \hat\C^{\mathcal{V}_0}$ of $R$ is called an {\em exposed critical point} if it is exposed.
The set of exposed critical points is denoted by $EC(R)$.

\subsection*{Extension on $\Omega = \mathcal{T}_1-\mathcal{T}_0$}
A finite tree $\mathcal{T}$ is called {\em star-shaped} if there is exactly $1$ vertex that is not an end point.
This unique vertex is called the {\em center}.
Depending on the number of end points, we will call $\mathcal{T}$ a {\em $k$-star} if there are $k$ end points of $\mathcal{T}$ (for $k\geq 2$).
Note that a $k$-star has exactly $k$ edges: $[p,a_j]$ for $j=1,..., k$, where $p$ is the center and $a_j\in \partial \mathcal{T}$. 

Let $(F, R)$ be a tree mapping scheme with compatible local degrees such that $F$ is piecewise linear compatible with the local degree.
Then $\mathcal{T}_1 - \mathcal{T}_0$ is a finite union of open subtrees $\mathcal{U}_k$ with $\partial \mathcal{U}_k \subseteq \mathcal{V}_0$.
We denote the union by
$$
\Omega = \bigcup_{k=1}^n \mathcal{U}_k.
$$
We say that $\Omega$ is {\em star-shaped} if each closure of the component $\overline{\mathcal{U}}_k$ is star-shaped.
In this case, $\mathcal{V}_1$ is the union of $\mathcal{V}_0$ and the centers of $\overline{\mathcal{U}}_k$.

Assume that $\Omega$ is star-shaped.
Let $\mathcal{U}$ be a component of $\Omega$ with center $p$.
We say $(F, R)$ extends to a {\em branched covering} on $\mathcal{U}$ if there exists a component $\mathcal{W}$ of $\Omega$ with center $p'$, a rational map $R_p: \hat\C_p \longrightarrow \hat\C_{p'}$ so that
\begin{itemize}
	\item $F(\partial \mathcal{U}) = \partial \mathcal{W}$;
	\item $F$ extends to a piecewise linear map $F: \overline{\mathcal{U}} \longrightarrow \overline{\mathcal{W}}$ that is compatible with the local degree;
	\item $\Xi_p = R_p^{-1}(\Xi_{p'})$.
\end{itemize}
Since the components are star-shaped, the extension of the tree map $F$ is uniquely determined.
We remark that the second condition implicitly implies that $R_p$ is compatible with the tangent map of $F$ at $p$, and has compatible local degrees with $R$. The last condition simply says that $\mathcal{U}$ is the `full preimage' of $\mathcal{W}$.

We say a component $\mathcal{U}$ of $\Omega$ {\em escapes} $\mathcal{T}_1$, or $\mathcal{U}$ is an {\em escaping component} if there exist $v\in \mathcal{V}_0$ and an exposed point $z_v\in \hat\C_v-\Xi_v$ such that 
\begin{itemize}
	\item $F(\partial \mathcal{U}) = v$;
	\item For any $a\in \partial \mathcal{U}$ and $z_a\in \Xi_a \subseteq \hat\C_a$ corresponding to the edge adjacent to $a$ in $\mathcal{U}$, $R(z_a) = z_v$.
\end{itemize}

\subsection*{Post-critically finite hyperbolic tree mapping schemes}
We are now able to define post-critically finite hyperbolic tree mapping schemes:
\begin{defn}\label{defn:hpcf}
Let $(F,R)$ be a tree mapping scheme from $(\mathcal{T}_0, \hat\C^{\mathcal{V}_0})$ to $(\mathcal{T}_1, \hat\C^{\mathcal{V}_1})$.
Let $\Omega = \mathcal{T}_1 - \mathcal{T}_0$.
It is said to be {\em post-critically finite hyperbolic} if
\begin{enumerate}[label=\roman*)]
\item $R$ has compatible local degrees;
\item $F$ is piecewise linear compatible with the local degree;
\item {\em (Extension on $\Omega$)} $\Omega$ is star-shaped, and for any component $\mathcal{U}$ of $\Omega$, either
\begin{enumerate}
\item $\mathcal{U}$ is an escaping component; or
\item $(F,R)$ extends to a branched covering on $\mathcal{U}$.
\end{enumerate}
Moreover, every component $\mathcal{U}$ of $\Omega$ is eventually mapped to an escaping component;
\item {\em (Hyperbolicity and post-critical finiteness for exposed critical points)} For any exposed critical point $z\in EC(R)$, $z$ has an exposed orbit and is eventually mapped to a periodic exposed critical cycle;
\item {\em (Hyperbolicity and post-critical finiteness for $\Omega$)} For any escaping component $\mathcal{U}$ of $\Omega$, $a\in \partial \mathcal{U}$ and $z_a\in\Xi_a \subseteq \hat\C_a$ associated to the edge adjacent to $a$ in $\mathcal{U}$, $R(z_a)$ has an exposed orbit and is eventually mapped to a periodic exposed critical cycle;
\item {\em (Hyperbolicity for the tree map)} The non-escaping set
$$
\mathcal{J}=\bigcap_{n=0}^\infty F^{-n}(\mathcal{T}_0)
$$
is totally disconnected.
\end{enumerate}
\end{defn}
We remark that the extension of $(F, R)$ to $\mathcal{U}$ is {\em not} canonical. See more on the discussion after Proposition \ref{prop:hpcf}.

Given a post-critically finite hyperbolic tree mapping scheme $(F, R)$, there is a natural way to define the degree of $(F,R)$. If $v\in \mathcal{V}_0$, $n_c(v)$ is the number of exposed critical points counted multiplicities at $v$. If $\mathcal{U}$ is an escaping component, $n_c(\mathcal{U}) = \sum \deg_{z_{x}} R_x$ where $x \in \partial \mathcal{U}$ and $z_{x} \in \hat\C_x$ corresponds to the adjacent edge in $\mathcal{U}$. Otherwise, $n_c(\mathcal{U})$ is the number of exposed critical points counted multiplicities at the center of $\mathcal{U}$ for any extension of the tree mapping scheme.
The degree of $(F,R)$ is
$$
\deg (F,R) = (\sum n_c(v) + \sum n_c(\mathcal{U})+2)/2.
$$
Note that it is not hard to verify that this degree is always a natural number. It also follows immediately from Theorem \ref{thm:classification} and the following discussion.

\subsection*{Remarks on the conditions}
Before studying properties for such tree mapping schemes, we will give intuitions behind each of the 6 conditions.

Let $f$ be a hyperbolic rational map with finitely connected Fatou set.
We shall prove later that each vertex $v\in \mathcal{V}_0$ corresponds to some pre-periodic Julia component $K_v$, and the length of the edge corresponds to the modulus between the two components. 
The exposed critical points correspond to the critical points in the disk Fatou components.
Each component of $\Omega$ corresponds to a critical or post-critical multiply connected Fatou component or a multiply connected Fatou component that separates the critical and post-critical points in three or more components.

The first two conditions in the definition come from the usual relations on moduli for coverings between annuli.
The condition iii) is a direct translation of the dynamics on the corresponding Fatou components. 
The conditions iv) and v) follow from the hyperbolicity and the fact that as we stretch the rational map $f$, the post-critical set is buried deeper and deeper in the Fatou set.
The last hyperbolicity condition for the tree map comes from the Gr\"otzsch's inequality for the moduli of annuli.
It is related to the necessary condition in Thurston's theory on rational maps (see Appendix B in \cite{McM94}). See \S \ref{sec:eg} for more discussions.

The number $n_c(v)$ is the number of critical points in the disk Fatou components whose boundaries are contained in the Julia component $K_v$, and $n_c(\mathcal{U})$ is the number of critical points in the corresponding multiply connected Fatou component $U$.
Thus, $\deg (F,R)$ equals the degree of $f$.

\subsection*{Properties of hyperbolicity and post-critical finiteness}
The following proposition follows directly from the definition and is worth mentioning.
\begin{prop}
Let $(F,R)$ be a post-critically finite hyperbolic tree mapping scheme, then
\begin{enumerate}
\item The non-escaping set $\mathcal{J} \neq \emptyset$.
\item If $a\in \mathcal{V}_0$ is a periodic point of period $p$, then $R^p|_{\hat\C_a}$ is post-critically finite and hyperbolic.
\end{enumerate}
\end{prop}
\begin{proof}
For the first statement, we note that by conditions iv), v) and vi), there exist periodic exposed critical points. 
Hence some point in $\mathcal{V}_0$ is periodic, so $\mathcal{J}\neq \emptyset$.

For the second statement, let $z\in \hat\C_a$ be a critical point of $R^p|_{\hat\C_a}$.
If $z$ is exposed, then it is eventually mapped to a periodic critical point by iv).
Otherwise, $z \in \Xi$. 

There are two cases: either it is eventually mapped outside of $\Xi$ or it is eventually periodic as $\Xi$ is finite.
In the first case, $z$ must be first mapped to some point $w$ corresponding to an edge in $\Omega$, so it is eventually mapped to a periodic critical point by v).
In the second case, after passing to an iterate if necessary, let $w\in \hat\C_a$ be the fixed point in $\Xi$ that $z$ is eventually mapped to.
The point $w$ must be a critical point, as otherwise, the edge $e$ associated to $w$ would be fixed by $F^p$ by  i) and ii).
Then the non-escaping set is not totally disconnected, thus violates condition vi).
Hence, $R^p|_{\hat\C_a}$ is post-critically finite and hyperbolic.
\end{proof}

The totally disconnectedness condition can be checked by a finite procedure using the following proposition:
\begin{prop}\label{prop:nonE}
Let $(F,R)$ be a tree mapping scheme satisfying the first two conditions in Definition \ref{defn:hpcf}. Then the non-escaping set
$\mathcal{J}=\bigcap_{n=0}^\infty F^{-n}(\mathcal{T}_0)$
is totally disconnected if and only if
for any edge $e \in \mathcal{T}_0$, there exists $k\in \N$ such that $F^k(e)$ is not contained in $\mathcal{T}_0$.
\end{prop}
\begin{proof}
It is clear that if the non-escaping set is totally disconnected, then every edge is eventually mapped out of $\mathcal{T}_0$.

To see the converse, we will prove by contradiction.
Assume for contradiction that $\mathcal{J}$ is not totally disconnected, then there exists a non-trivial closed interval $I$ that is non-escaping, i.e. $I, F(I), F^2(I),...$ all lie in $\mathcal{T}_0$.
We assume that $I$ is maximal in the sense that there is no other non-escaping closed interval $\mathcal{J}$ strictly containing $I$.

We claim that $\mathcal{V}_0$ intersects $\cup_{k=0}^\infty F^k(I)$.
Indeed, otherwise, all intervals are strictly contained in edges of $\mathcal{T}_0$.
Since $F$ is piecewise linear compatible with the local degree and $F$ is injective on edges, the length of $F^k(I)$s are non-decreasing.
This implies that either $I$ is eventually periodic, or $F^i(I)$ intersects $F^j(I)$ for some $i< j$.
In the first case, we denote the pre-period by $l$. 
Then $F^l(I)$ is strictly contained in some edges and periodic, so we can enlarge $F^l(I)$ so that it remains periodic, and hence non-escaping. 
By taking preimages, we can also enlarge $I$, which contradicts the maximality of $I$.
Similarly, in the second case, let $i<j$ so that $F^i(I)$ intersects $F^j(I)$, then $F^i(I) \cup F^j(I)$ is still non-escaping and strictly contained in an edge. Take the $i$-th preimage of $F^i(I) \cup F^j(I)$ containing $I$, we get a larger non-escaping set, which contradicts the maximality of $I$.

Therefore, by taking iterates and restricting to a smaller interval, we may assume that $I = [a,b] \subseteq e\in E_0$ with $a\in \mathcal{V}_0$.
Since there are only finitely many vertices and edges, by taking further iterates, we may assume that there exists $p$ so that $I \subseteq F^p(I)$.
If $I$ is strictly contained in $F^p(I)$, then $F^p$ has derivative at least $2$ on $e$ so there exists $k$ so that $F^k(I)$ contains the edge $e$.
This contradicts the premise that the edge $e$ eventually maps out of $\mathcal{T}_0$.
Otherwise, $I = F^p(I)$, then $F^p$ has derivative $1$ on $e$, so $e$ is a periodic edge.
This is also a contradiction.
Therefore, the non-escaping set $\mathcal{J}$ is totally disconnected.
\end{proof}

\subsection*{Irreducible tree mapping schemes}
Given a post-critically finite hyperbolic tree mapping scheme, one can get different but dynamically equivalent tree mapping schemes by introducing more vertices or branches on the tree by taking preimages, or more `gaps' in $\Omega$. 
To get rid of these ambiguities, we introduce the notion of {\em irreducible} tree mapping schemes.

Let $\mathcal{U}$ be a component of $\Omega$ with center $p$.
We say $\mathcal{U}$ is {\em critical} if $n_c(\mathcal{U}) \geq 1$.
We say a component $\mathcal{U}$ of $\Omega$ is {\em primitive} if there is no component $\mathcal{W}$ of $\Omega$ with $F(\mathcal{W}) = \mathcal{U}$ under the extension.

Let $\mathcal{Q}\subseteq \mathcal{V}_0$ be the smallest forward invariant set containing vertices $v$ with $n_c(v) \geq 1$ and the boundaries of critical components of $\Omega$.
Let $\mathcal{P} \subseteq \mathcal{V}_0$ be the smallest forward invariant set containing $\partial \Omega$ and $\mathcal{Q}$. 

Let $\mathcal{B}_1$ be the set of branch points of $\mathcal{T}_1$.
For reasons that will become apparent later, we say a vertex $v\in \mathcal{B}_1$ is a Julia branched point if $v\in \mathcal{J}$, and is a Fatou branch point otherwise. We denote this partition as
$$
\mathcal{B}_1 := \mathcal{B}_1^J \cup \mathcal{B}_1^F.
$$
\begin{defn}\label{defn:irreducible}
Let $(F,R)$ be a post-critically finite hyperbolic tree mapping scheme. It is said to be {\em irreducible} if 
\begin{enumerate}
\item $\mathcal{T}_1$ is the convex hull of $\mathcal{Q}$;
\item $\mathcal{V}_0 = \mathcal{P} \cup  \mathcal{B}_1^J$;
\item Every primitive component $\mathcal{U}$ of $\Omega$ is either critical or contains a branch point in $\mathcal{B}_1^F$.
\end{enumerate}
\end{defn}

By construction, $\mathcal{V}_0$ is forward invariant.
By deleting the additional vertices (and their associated edges if the vertices are ends) and inductively using the extension to fill in those primitive components of $\Omega$ that neither contain critical points nor branch points, any post-critically finite hyperbolic tree mapping scheme $(F, R)$ can be reduced to a unique irreducible one.

\subsection*{Limits of rational maps as tree mapping schemes}
Tree mapping schemes appear as more general limits of rational maps (cf. \S 4  in \cite{Arfeux17}).

\begin{defn}\label{defn:ctm}
	Let $(F, R)$ be a tree mapping scheme from $(\mathcal{T}_0, \hat\C^{\mathcal{V}_0})$ to $(\mathcal{T}_1, \hat\C^{\mathcal{V}_1})$.
	A family of rescalings $M_{a,n} \in  \PSL_2(\C)$ for $a \in \mathcal{V}_1$ is said to {\em represent} the vertices if for all $a\neq b \in \mathcal{V}_1$,
	$$
	M_{b,n}^{-1} \circ M_{a,n}(z) \to \xi_a(\vec{v}_b)
	$$
	compactly away from a single point, where $\vec{v}_b\in T_a\mathcal{T}_1$ is the tangent vector in the direction of $b$.
	
	A sequence $f_n$ of degree $d$ rational maps is said to {\em converge} to $(F, R)$ if there exist {\em rescalings} $M_{a,n} \in \PSL_2(\C)$ representing $a \in \mathcal{V}_1$ such that
$$
M_{F(a),n}^{-1} \circ f_n \circ M_{a,n}(z) \to R_a(z)
$$ 
compactly on $\hat \C_a- \Xi_a$ for all $a\in \mathcal{V}_0$.
We shall call $R_a$ the {\em rescaling limits} of $f_n$ from $M_{a,n}$ to $M_{F(a),n}$.
\end{defn}
Since there is no convenient normal form for the conjugacy classes of rational maps, the rescalings $M_{a,n}$ are regarded as dynamically appropriate normalizations of $f_n$.
A family of rescalings represents the vertices means that the rescalings are compatible with the tree structure $\mathcal{T}_1$: when $a,b$ are adjacent, the Riemann sphere $\hat \C_b$ is `glued' to $\hat \C_a$ at the marked point $\xi_a(\vec{v}_b) \in \Xi_a \subseteq \hat\C_a$.

\section{Examples of tree mapping schemes and rescaling limits}\label{sec:eg}
Before diving into the proof of the theorems, we shall first see some examples of hyperbolic rational maps with finitely connected Fatou set, and how the tree mapping schemes arise for these examples.
To simplify the notations and graphs, we shall omit the centers of a component $\mathcal{U}$ of $\Omega$, i.e., we shall only label the vertices in $\mathcal{V}_0$.

\subsection*{Hyperbolic rational maps with a Cantor set of circles.}
\begin{figure}[ht]
    \centering
    \resizebox{\linewidth}{!}{
    \def\svgwidth{\columnwidth}
    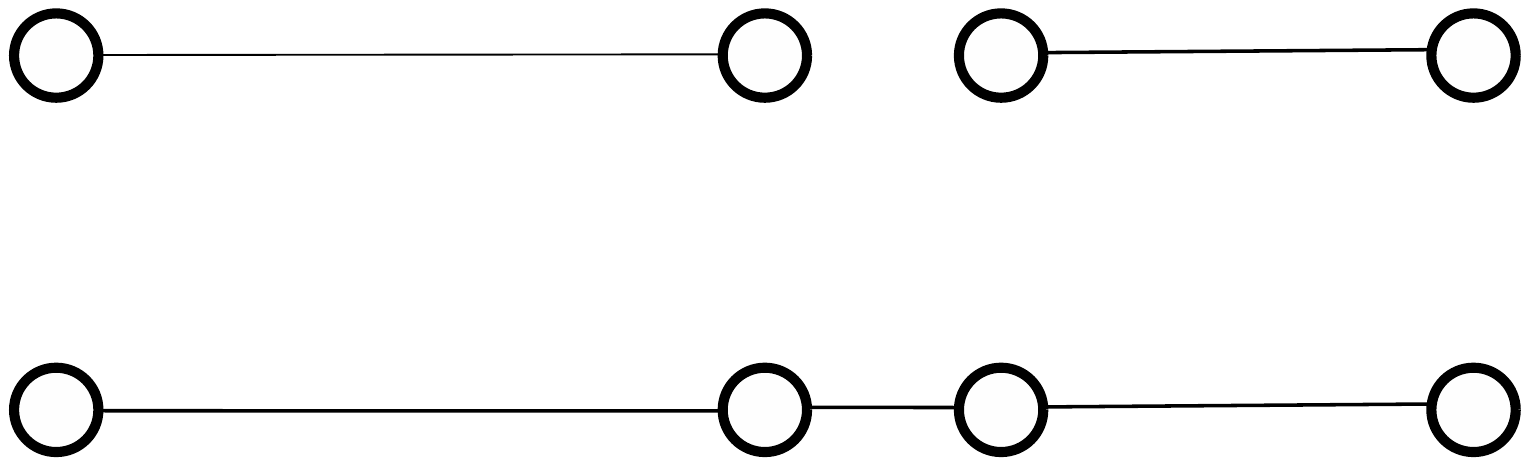

    \def\svgwidth{\columnwidth}
    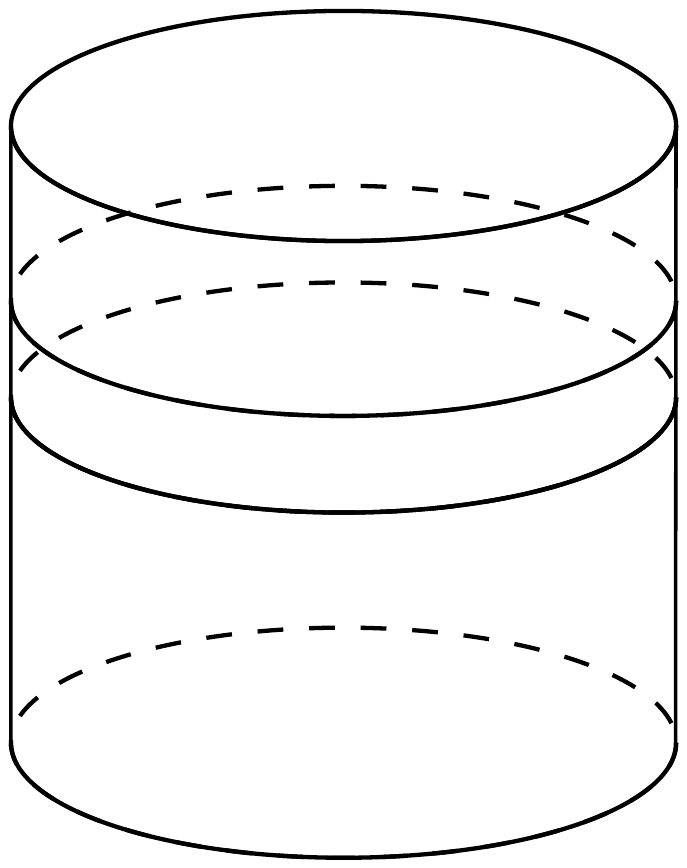

    }
    \caption{The tree map for the Cantor circle example on the left.}
    \label{fig:CantorCircleExample}
\end{figure}
The first example of rational map with non-degenerate buried Julia component was introduced by McMullen in \cite{McM88}.
In that paper, McMullen shows that for all sufficiently large $n$, the Julia set of the rational map $f_n(z) = z^2 + \frac{1}{nz^3}$ is homeomorphic to a Cantor set times circle.
The associated tree map for $f_n(z) = z^2 + \frac{1}{nz^3}$ is described in Figure \ref{fig:CantorCircleExample}.
$\mathcal{T}_1 = [A, A']$, $\mathcal{T}_0 = [A, B]\cup [B', A']$, $\mathcal{V}_1 = \{A, A', B, B'\}$.
The map $F$ sends $[A,B]$ $[A', B']$ homeomorphically onto $[A, A']$ by expanding with factor $2$ and $3$ respectively.

It is already observed in \cite{McM88} the connection with the tree map (see Figure \ref{fig:CantorCircleExample}): for sufficiently large $n$, there exists a cylinder $\mathcal{A}$ containing two cylinders $\mathcal{A}_1$ and $\mathcal{A}_2$, which are mapped to $\mathcal{A}$ via degrees $2$ and $3$ covering maps respectively. 
The boundaries of $\mathcal{A}_1$ and $\mathcal{A}_2$ in Figure \ref{fig:CantorCircleExample} are labeled so that $C_a$ is mapped to $C_b$ if $a$ is sent to $b$ by the tree map $F$.

\begin{figure}[ht]
    \begin{subfigure}{0.45\textwidth}
    \centering
    \includegraphics[width=.8\linewidth]{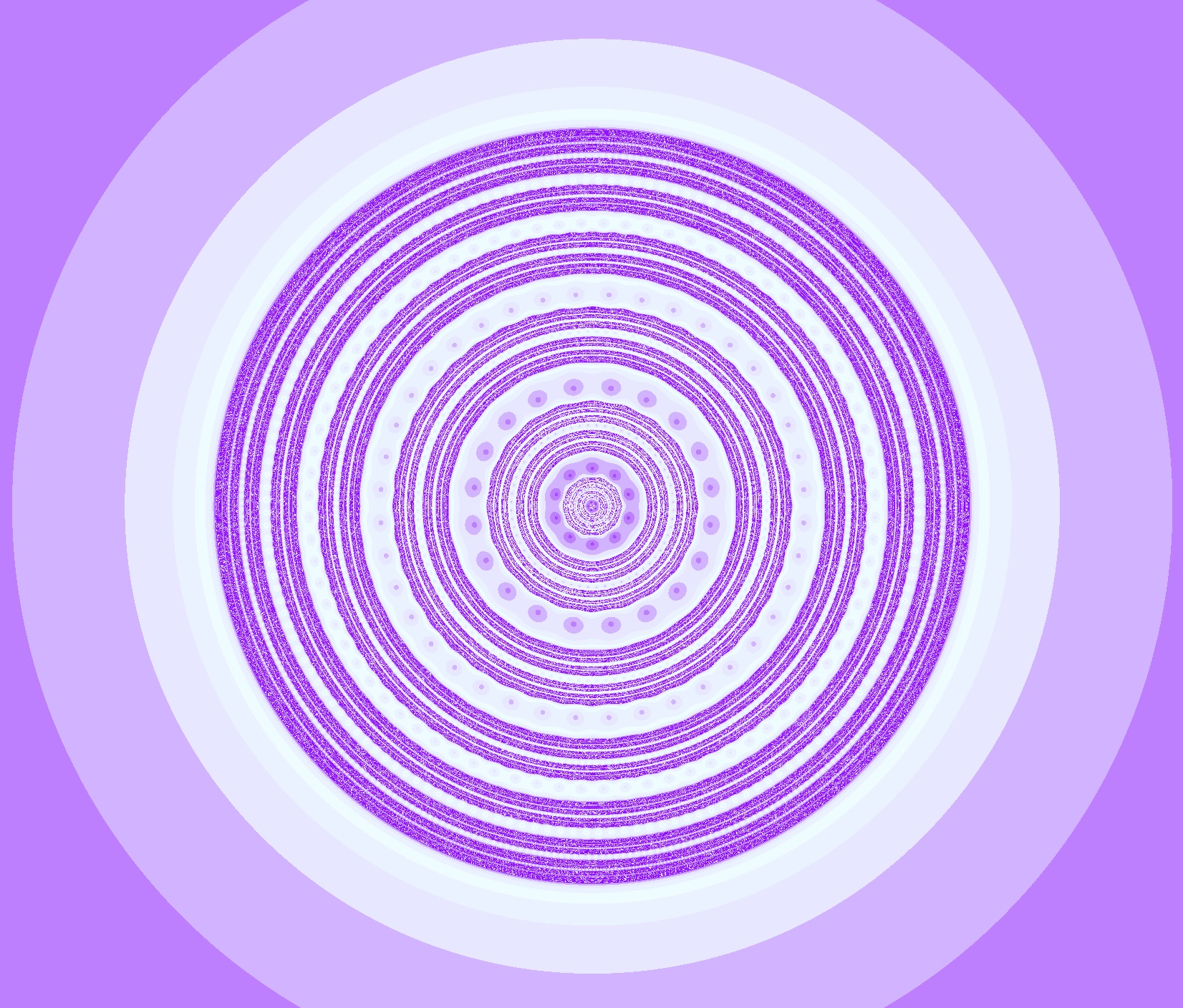}
    \caption{The Julia set of the usual Cantor circle example $f_n(z) = z^2+\frac{1}{nz^3}$. $R_{A\to A}(z) = z^2$}
    \end{subfigure}
    \begin{subfigure}{0.45\textwidth}
    \centering
    \includegraphics[width=.8\linewidth]{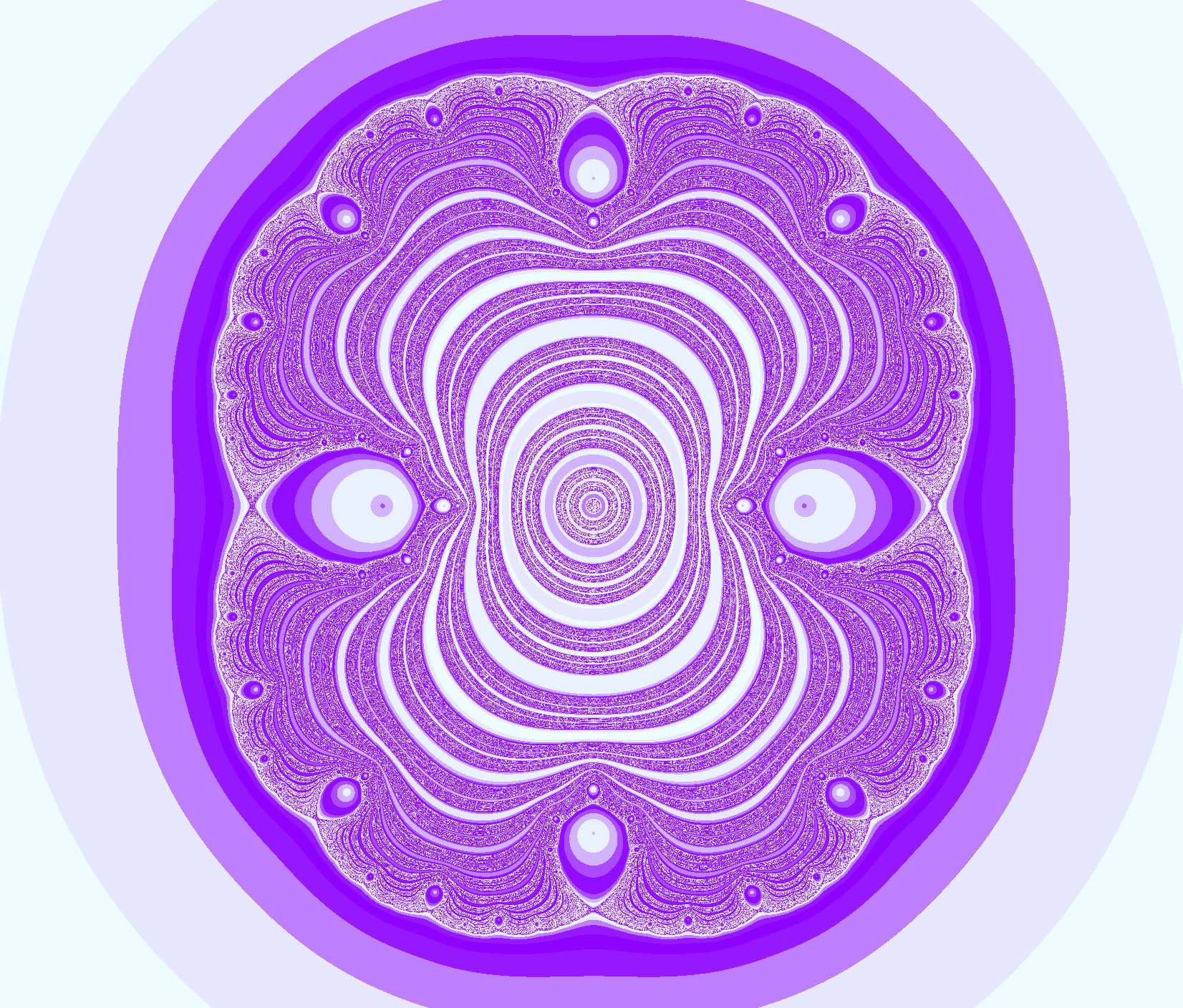}
    \caption{The Julia set of the nested mating of the Basilica polynomial $f_n(z) = \frac{z^2}{z^2-1}+\frac{1}{nz^3}$. $R_{A\to A}(z) = \frac{z^2}{z^2-1}$}
    \end{subfigure}
    \begin{subfigure}{0.45\textwidth}
    \centering
    \includegraphics[width=.8\linewidth]{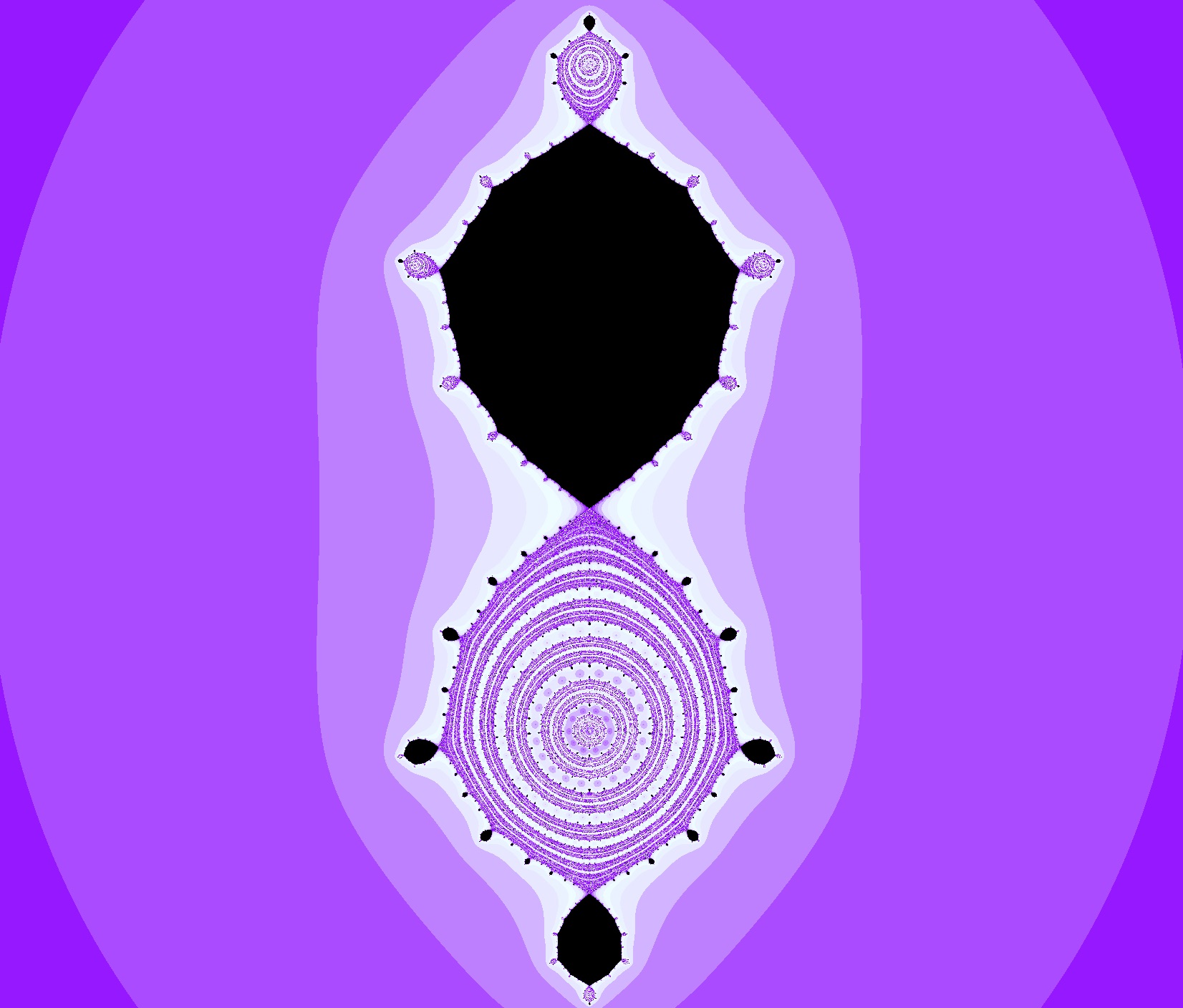}
    \caption{The Julia set of $f_n(z)=z^3-\frac{3i}{\sqrt{2}}z^2+\frac{1}{nz^3}$. $R_{A\to A}(z) = z^3-\frac{3i}{\sqrt{2}}z^2$}
    \end{subfigure}
    \begin{subfigure}{0.45\textwidth}
    \centering
    \includegraphics[width=.8\linewidth]{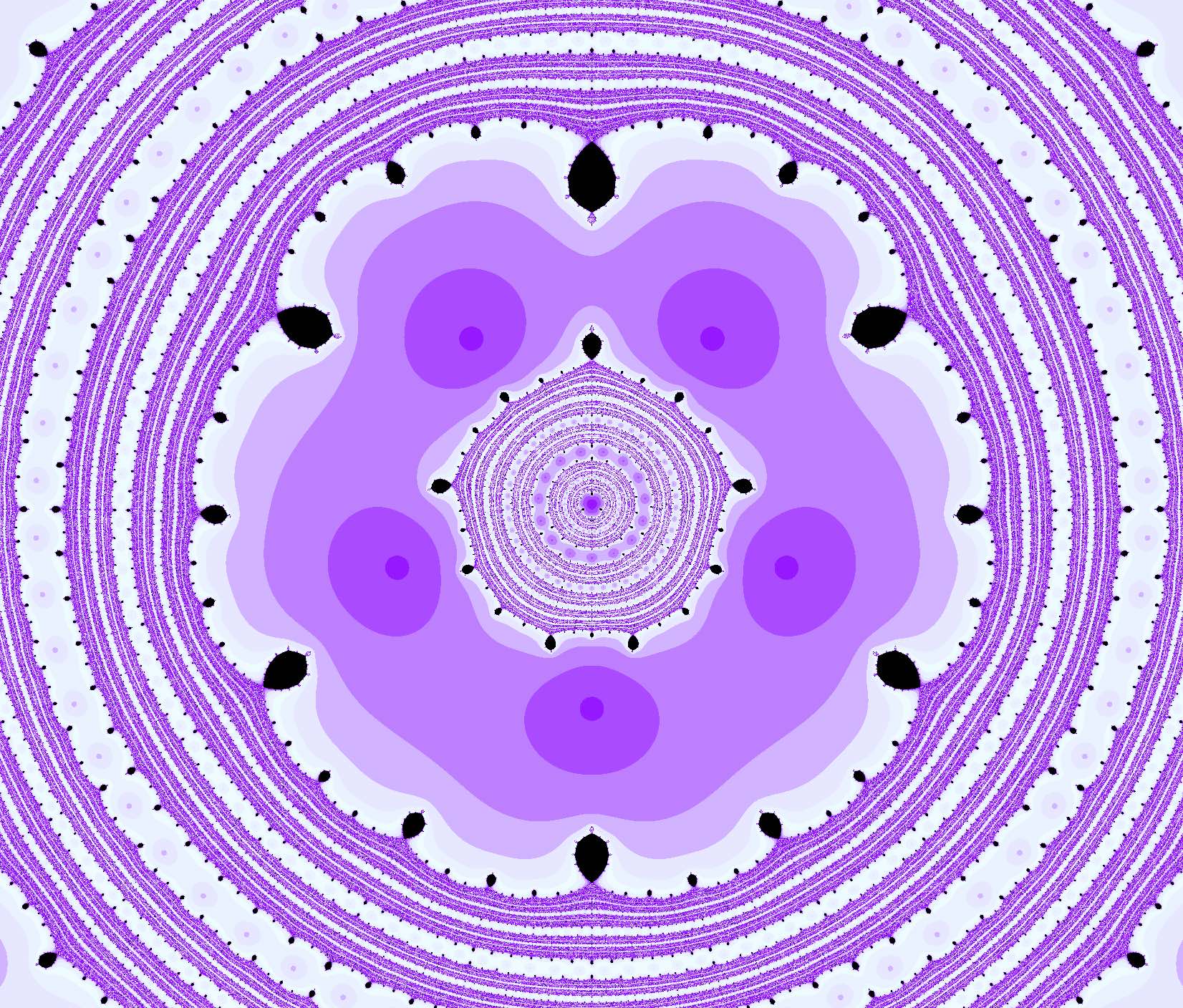} 
    \caption{A zoom of the Julia set on the left.}
    \end{subfigure}
    \caption{All three hyperbolic rational maps have the same tree map as in Figure \ref{fig:CantorCircleExample}. $R_{A\to A}$ is the only difference in these examples.
    The first two examples are of degree $5$, and have exactly $1$ exposed critical periodic cycle (fixed point $\infty_A$ in the first example, and period 2 cycle $\infty_A\to 1_A$ in the second example); the last example is of degree $6$, and has $2$ exposed critical periodic cycle (fixed points $\infty_A$ and ${\sqrt{2}i}_A$).}
    \label{fig:JuliaCantorCircle}
\end{figure}

We mark the Riemann spheres so that the tangent vector pointing towards (respectively, away from) $A$ corresponds to $\infty$ (respectively, $0$). 
Under this marking, the rational mapping scheme associated to $f_n$ is
\begin{align*}
R_{A\to A}(z) = R_{B\to A'}(z) &= z^2,\\
R_{A'\to A}(z) = R_{B'\to A'}(z) &= \frac{1}{z^3}.
\end{align*}

The rescaling limits appear explicitly in this setting.
If we choose the rescalings as $M_{A,n}(z) = z$, $M_{B,n}(z) = n^{-1/6}z$, $M_{B',n}(z) = n^{-2/9}z$ and $M_{A',n}(z) = n^{-1/3}z$, then it is easy to check that the corresponding rescaling limits give the rational mapping scheme.

There are many different post-critically finite hyperbolic tree mapping schemes supported on this tree map.
If we fix $R_{B\to A'}(z) = z^2$ and $R_{A'\to A}(z) = R_{B'\to A'}(z) = \frac{1}{z^3}$ and only change $R_{A\to A}$,
then the conditions on hyperbolicity and post-critical finiteness are equivalent to 
\begin{enumerate}
\item $R_{A\to A}$ is a post-critically finite hyperbolic map.
\item $0_A$ is a degree $2$ critical fixed point of $R_{A\to A}$.
\item The orbits of all other critical points of $R_{A\to A}$ are disjoint from $0_A$.
\item $\infty_{A}$ is preperiodic to some super attracting periodic cycle other than $0_A$.
\end{enumerate}
The Julia sets with different $R_{A\to A}$ can have drastically different shapes (see Figure \ref{fig:JuliaCantorCircle}).
We mention that the annulus $\mathcal{A}$ is visible in all three examples of Figure \ref{fig:JuliaCantorCircle}.
In the first two examples, the closure of $\mathcal{A}$ contains the Julia set.
In the third example, there are infinitely many Julia components that are outside of $\mathcal{A}$, in fact, there are infinitely many periodic Julia components that are outside of $\mathcal{A}$.
However, $\mathcal{A}$ contains all non-degenerate periodic Julia components except for the one that intersects $\partial \mathcal{A}$.

From the algebraic formulas provided in the captions of Figure \ref{fig:JuliaCantorCircle}, it is easy to construct the explicit rescaling limits.
We remark that this is how we find the algebraic formulas for all the figures in this paper.

A direct generalization of the McMullen family is the sequence $f_n(z) = z^{d_1} + \frac{1}{nz^{d_2}}$.
In order for $f_n$ to lie in the same hyperbolic component for all sufficiently large $n$, it is important that $\frac{1}{d_1} + \frac{1}{d_2} < 1$ \cite{DM07}.
This condition follows from the Gr\"otzsch's inequality for the corresponding annuli as in Figure \ref{fig:CantorCircleExample}, and is equivalent to the hyperbolicity condition of the tree map.
The left hand side $\frac{1}{d_1} + \frac{1}{d_2}$ can also be constructed as the spectral radius $\lambda (\Gamma)$ of the Thurston's matrix for some curve system $\Gamma$, and the inequality can be interpreted as the classification of $\lambda(\Gamma) = 1$ (see Theorem B.4 in \cite{McM94}).

\begin{figure}[ht]
    \centering
    \begin{subfigure}{0.45\textwidth}
        \centering
    \def\svgwidth{\columnwidth}
    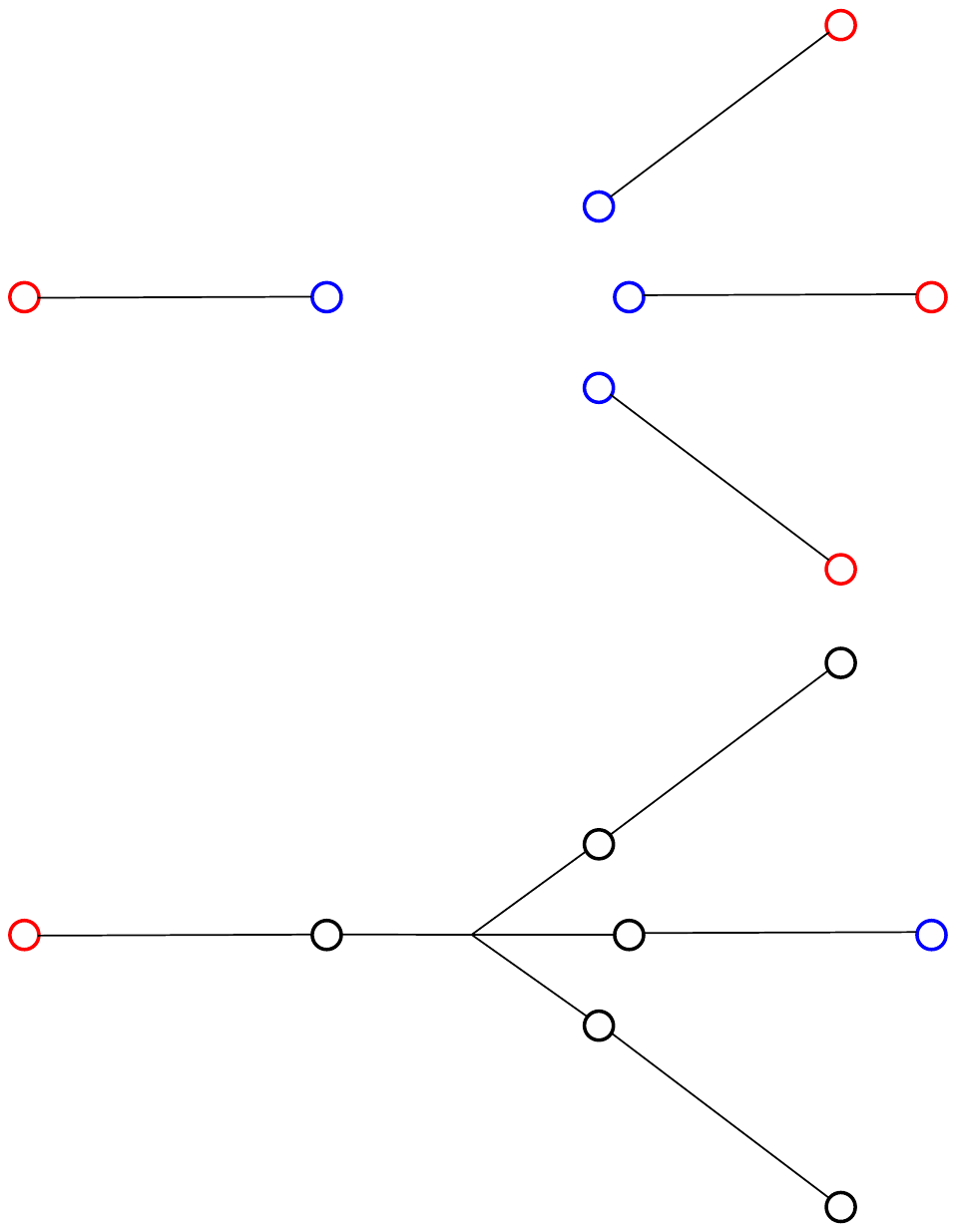

        \caption{The tree map with a quadruply connected critical Fatou component. Each edge is of unit length and $F$ sends red (blue) points to red (blue, respectively) points.}
    \end{subfigure}
    \begin{subfigure}{0.45\textwidth}
        \centering
        \includegraphics[width=\textwidth]{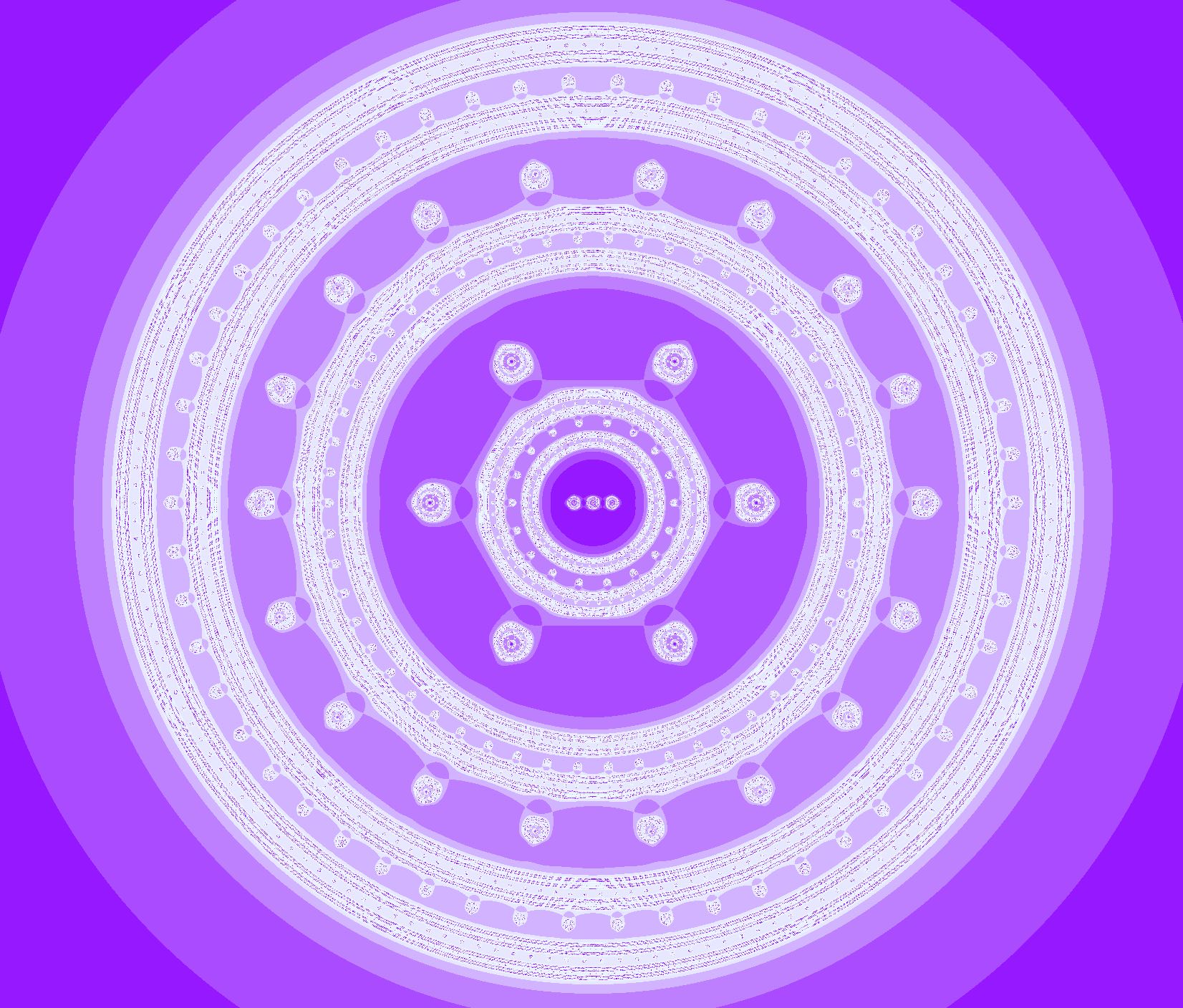}
        \caption{The Julia set of hyperbolic rational map realizes $F$ on the right. The algebraic formula is $f_n(z) = z^3+\frac{1}{(n^2z)^3} + \frac{1}{(n^2z-n)^3} +\frac{1}{(n^2z+n)^3}$ for sufficiently large $n$.}
    \end{subfigure}
    \caption{}
    \label{fig:kthlyConnectedExample}
\end{figure}

\subsection*{Examples of $k$-thly connected Fatou component.}
Examples with $k$-thly connected Fatou component can be constructed similar as in the previous examples.
Consider $\mathcal{T}_1$, and $\mathcal{T}_0$ as in Figure \ref{fig:kthlyConnectedExample}.
For any $j=0,1,2,3$, $F$ sends $[A_j,B_j]$ homeomorphically onto $[A_0, A_3]$ via expanding expanding with factor $3$.
A particular example of hyperbolic rational map with tree map $F$ is also given in Figure \ref{fig:kthlyConnectedExample}.

\subsection*{Other examples}
We refer the readers to Figures \ref{fig:BuriedSierpinski}, \ref{fig:BuriedSierpinskiTreeMap}, \ref{fig:SurgeryOnFixedPoint}, \ref{fig:JuliaSurgeryOnFixedPoint}, \ref{fig:JuliaSetOfDegree3Example} and \ref{fig:Degree3Example} for more examples.


\section{From rational map to tree mapping scheme}\label{sec:tms}
In this section, we will explain how post-critically finite hyperbolic tree mapping schemes arise as limit of stretch of hyperbolic rational map.

\subsection{Quasiconformally stretching}
We shall first introduce the construction of {\em quasiconformal stretch} for a hyperbolic rational map.
To start the discussion, we review some definitions.

\subsection*{Quasiconformal maps}
Let $U \subseteq \hat \C$. Let $K \geq 1$, and set $k = \frac{K-1}{K+1}$. 
A map $\phi: U \longrightarrow \phi(U)$ is called {\em $K$-quasiconformal} if 
\begin{enumerate}
\item $\phi$ is a homeomorphism;
\item the partial derivatives $\partial_z \phi$ and $\partial_{\bar z} \phi$ exist in the sense of distributions and belong to $L^2_{loc}$ (i.e. are locally square integrable);
\item and satisfy $|\partial_{\bar z} \phi| < k |\partial_z \phi|$ in $L^2_{loc}$.
\end{enumerate}
A map $f: U \longrightarrow f(U)$ is called {\em $K$-quasiregular} if $f = g\circ \phi$, where $\phi$ is quasiconformal and $g$ is holomorphic.

The following well-known result known as Shishikura's principle (see Lemma 1 in \cite{Shishikura87} or Proposition 5.2 and Corollary 5.4 in \cite{BrannerFagella14}):
\begin{prop}\label{prop:ShishikuraPrinciple}
Let $f:\hat \C \longrightarrow \hat \C$ be a proper $K$-quasiregular map.
Let $U \subseteq \hat \C$ be an open set.
Assume that
\begin{enumerate}
\item $f$ is holomorphic in $\hat \C - U$;
\item there exists $N$ such that $f^j(U) \cap U = \emptyset$ for all $j \geq N$.
\end{enumerate}
Then $f$ is quasiconformally conjugate to a rational map.
\end{prop}

\subsection*{Linear stretch of an annulus and a disk}
Let $s< r\leq 1$, and let $A_r(s)$ denote the round annulus $B(0,r) - \overline{B(0,s)}$.
Given two round annuli $A_r(s)$ and $A_r(s')$, we define a {\em linear stretch} between them as
\begin{align*}
\psi_{r, s\to s'} : A_r(s) &\longrightarrow A_r(s')\\
(\rho,\theta) &\mapsto (\frac{r-s'}{r-s}\rho+\frac{s'-s}{r-s}r, \theta)
\end{align*}
where $(\rho,\theta)$ is the polar coordinate.
We can extend the map continuously to unit disks $\psi_{r, s\to s'}: \Delta \longrightarrow \Delta$ by setting
$$
\psi_{r, s\to s'}(z) = \begin{cases} \frac{s'}{s}z &\mbox{if } z\in \overline{B(0,s)} \\ 
z & \mbox{if } z\in \Delta - B(0,r) \end{cases}.
$$
Note that $\psi_{r, s\to s'}$ is $K$-quasiregular, where $K = \max\{\frac{\log s}{\log s'}, \frac{\log s'}{\log s}\}$.

\subsection*{Stretch of a hyperbolic rational map}
Let $f$ be a hyperbolic rational map with finitely connected Fatou set.
By the Riemann-Hurwitz formula and the classification of Fatou components, this is equivalent to saying every periodic Fatou component is a disk.

Let $D$ be a disk Fatou component.
Let $\phi_D: \Delta \longrightarrow D$ be a Riemann mapping appropriately normalized (see \cite{Milnor12}).
Since $f$ is hyperbolic, there exist $s < r <1$ so that for any disk Fatou component $D$,
\begin{itemize}
	\item no post-critical points are in $\phi_D(\Delta - \overline{B(0, s)})$; and
	\item $\phi_D(B(0, r) - \overline{B(0, s)})$ is not recurrent.
\end{itemize}

We construct $\tilde{f}_n:\hat\C\longrightarrow \hat\C$ by replacing $f$ on any component of $U \subseteq f^{-1}(D)$ of a preimage of a post-critical disk Fatou component $D$ by
$$
\tilde{f}_n|_{U}:=\phi_{D}\circ\psi_{r,s\to e^{-n}}\circ \phi_{D}^{-1}\circ f|_U.
$$
For all sufficiently large $n$, $\tilde{f}_n$ is $n/\log s$-quasiregular.
Since $\phi_{D}(B(0, r) - \overline{B(0, s)})$ are not recurrent, by Shishikura's principal (Proposition \ref{prop:ShishikuraPrinciple}), $\tilde{f}_n$ is quasiconformally conjugate to a rational map $f_n$. We will call such construction a {\em quasiconformal stretch} or a {\em stretch} for short.

By construction, the separations of the post-critical set and the boundaries of Fatou components are going to infinity.
Thus if the Julia set $J(f)$ is connected, the stretch $f_n$ converges to the corresponding post-critically finite hyperbolic rational map $f_0$.

In the next few subsections, we shall prove that in general, $f_n$ converges to a post-critically finite hyperbolic tree mapping scheme $(F,R)$.
We will construct the tree map $F$ in \S \ref{subsec:tm}, and the rational mapping schemes are constructed \S \ref{subsec:ms}.
Finally, the proof of Theorem \ref{thm:cvg} is given in \S \ref{subsec:hp}.

\subsection{Tree map from the Shishikura tree}\label{subsec:tm}
In this subsection we construct the tree map for $f$.
Let $U_1,..., U_k$ be a list of multiply connected Fatou components that are mapped to disk Fatou components $D_1, ...., D_k$ under $f$.
Here $D_i$ may be the same as $D_j$.

Let $A(s) = \Delta - \overline{B(0,s)}$.
Since $\phi_{D_i}(A(s))$ contains no post-critical points of $f$, the preimage $f^{-1}(\phi_{D_i}(A(s))) \cap U_i$ is a finite union of annuli $A_{i, 1},..., A_{i, b_i}$.
Note that $b_i$ equals to the number of boundary components of $U_i$, and $f$ is a degree $d_{i,j}$ covering between $A_{i,j}$ and $\phi_{D_i}(A(s))$ where the degree $d_{i,j}$ equals to the degree of covering $f$ on the corresponding boundary component of $U_i$.
We first define the collection of all preimages of these annuli as
$$
\tilde{\mathscr{A}} := \{f^{-l}(A_{i,j}): i=1,..., k, j=1,..., b_i, l \geq 0\}.
$$
Note that by construction, each annulus $A \in \tilde{\mathscr{A}}$ is contained in some multiply connected Fatou component.

To get a finite combinatorial model, we first set up some notations.
Let $\mathscr{U}$ be the list of critical and post-critical multiply connected Fatou components and $\mathscr{D}$ be the list of critical and post-critical disk Fatou components.
We denote 
$$
\mathscr{F} := \mathscr{U} \cup \mathscr{D}
$$
as the collection of all critical and post-critical Fatou components.
We say a domain $E$ is {\em separating} for $\mathscr{F}$ if there are a pair $V, W\in \mathscr{F}$ and a pair of points $x \in \overline{V}$, $y\in \overline{W}$ that are in different components of $\hat\C-E$.
Note that by the above convention, any annulus $A \in \tilde{\mathscr{A}}$ that is contained in $U \in \mathscr{U}$ is separating.
We consider the sub-collection
$$
\mathscr{A} := \{A\in \tilde{\mathscr{A}}: A \text{ is separating for } \mathscr{F}\}.
$$

\begin{figure}[ht]
    \centering
    \resizebox{0.6\linewidth}{!}{
    \def\svgwidth{\columnwidth}
\begingroup%
  \makeatletter%
  \providecommand\color[2][]{%
    \errmessage{(Inkscape) Color is used for the text in Inkscape, but the package 'color.sty' is not loaded}%
    \renewcommand\color[2][]{}%
  }%
  \providecommand\transparent[1]{%
    \errmessage{(Inkscape) Transparency is used (non-zero) for the text in Inkscape, but the package 'transparent.sty' is not loaded}%
    \renewcommand\transparent[1]{}%
  }%
  \providecommand\rotatebox[2]{#2}%
  \newcommand*\fsize{\dimexpr\f@size pt\relax}%
  \newcommand*\lineheight[1]{\fontsize{\fsize}{#1\fsize}\selectfont}%
  \ifx\svgwidth\undefined%
    \setlength{\unitlength}{841.88976378bp}%
    \ifx\svgscale\undefined%
      \relax%
    \else%
      \setlength{\unitlength}{\unitlength * \real{\svgscale}}%
    \fi%
  \else%
    \setlength{\unitlength}{\svgwidth}%
  \fi%
  \global\let\svgwidth\undefined%
  \global\let\svgscale\undefined%
  \makeatother%
  \begin{picture}(1,0.70707071)%
    \lineheight{1}%
    \setlength\tabcolsep{0pt}%
    \put(0,0){\includegraphics[width=\unitlength,page=1]{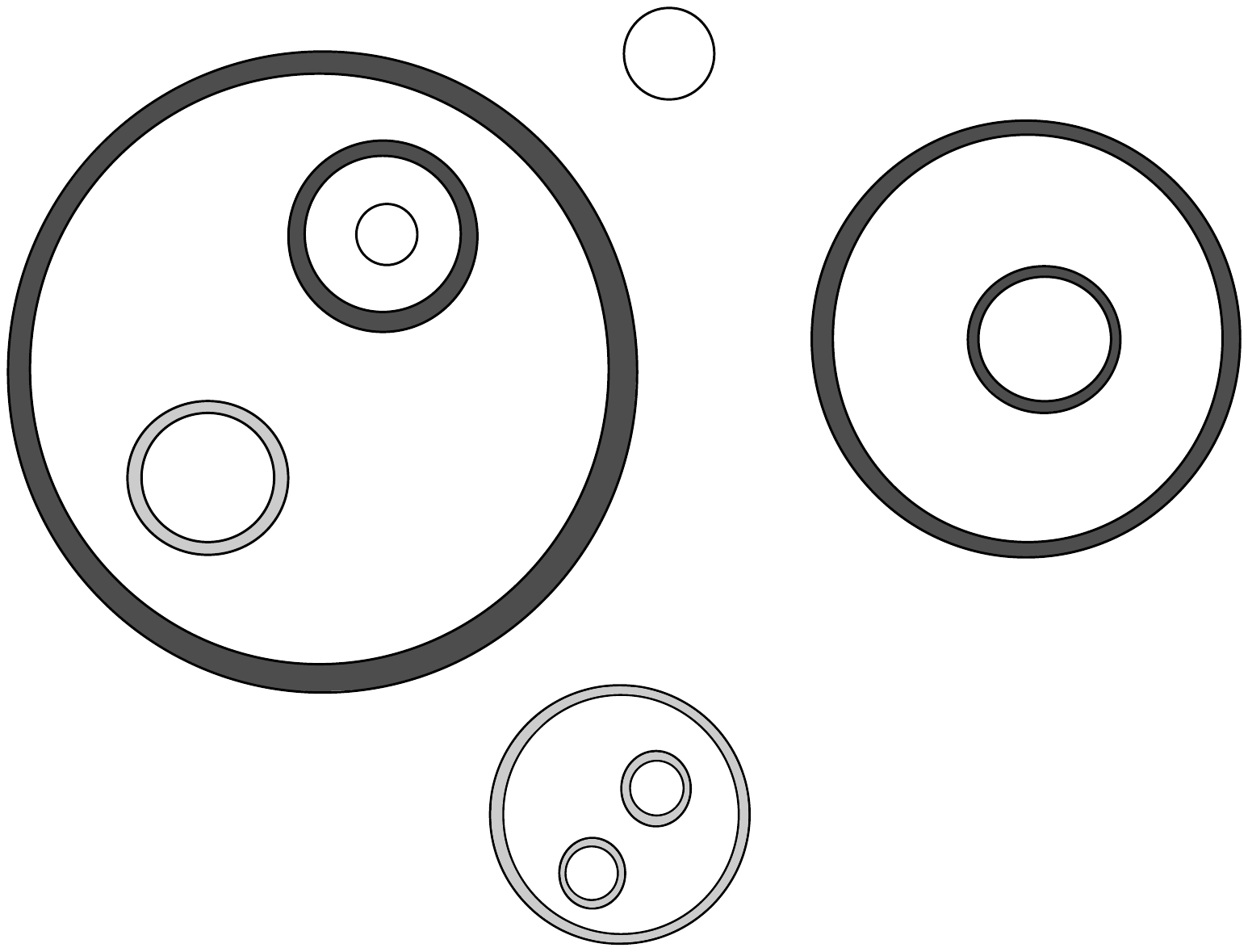}}%
    \put(0.77676464,0.5213767){\color[rgb]{0,0,0}\makebox(0,0)[lt]{\lineheight{1.25}\smash{\begin{tabular}[t]{l}$U$\end{tabular}}}}%
    \put(0.36475609,0.34065668){\color[rgb]{0,0,0}\makebox(0,0)[lt]{\lineheight{1.25}\smash{\begin{tabular}[t]{l}$U'$\end{tabular}}}}%
    \put(0.45511609,0.10853916){\color[rgb]{0,0,0}\makebox(0,0)[lt]{\lineheight{1.25}\smash{\begin{tabular}[t]{l}$U''$\end{tabular}}}}%
    \put(0.3456893,0.49945176){\color[rgb]{0,0,0}\makebox(0,0)[lt]{\lineheight{1.25}\smash{\begin{tabular}[t]{l}$D$\end{tabular}}}}%
    \put(0.54216014,0.6280691){\color[rgb]{0,0,0}\makebox(0,0)[lt]{\lineheight{1.25}\smash{\begin{tabular}[t]{l}$D'$\end{tabular}}}}%
  \end{picture}%
\endgroup%

    }
    \caption{An illustration of positions of $\mathscr{A}, \mathscr{U}, \mathscr{D}$. Here $D, D' \in  \mathscr{D}$, $U \in \mathscr{U}$ while $U', U'' \notin \mathscr{U}$. The dark grey annuli in the figures are elements of $\mathscr{A}$, while the light grey annuli are elements of $\tilde{\mathscr{A}} - \mathscr{A}$.}
    \label{fig:Ann}
\end{figure}

From this collection of annuli $\mathscr{A}$, one can construct a metric tree, known as the {\em Shishikura tree} first introduced in \cite{Shishikura89}.
We briefly summarize the construction here and refer the readers to \cite{Shishikura89} for details.

Let $A$ be an annulus, we define $A[z] = \phi_A(\{\zeta: |\zeta| = |\phi_A^{-1}(z)|\})$ for $z\in A$, where $\phi_A$ is any uniformizing map of a round annulus centered at $0$ to $A$.
Note that each $A[z]$ is a Jordan curve.
We denote the annulus $A(x,y) = \{z\in A: A[z] \text{ separates } x \text{ and } y\}$.
Now for any two points $x, y\in \hat\C$, we define a psuedo-metric
$$
\tilde{d}(x,y) = \sum_{A\in \mathscr{A}} m(A(x,y)),
$$
where $m(A(x,y))$ is the modulus of the annulus $A(x,y)$.
This pseudo-metric defines an equivalence relation: $x\sim_{\mathscr{A}} y$ if $\tilde{d}(x,y) = 0$.
The quotient of $\hat\C$ by this equivalence relation $\sim_{\mathscr{A}}$ gives a finite metric tree $\mathcal{T}_1$.
Let us denote the continuous projection by $\pi:\hat\C\longrightarrow \mathcal{T}_1$.
From our construction,
\begin{itemize}
\item If $K$ is a Julia component, then $\pi(K)$ is a single point;
\item If $D$ is a disk Fatou component, then $\pi(K)$ is a single point;
\item If $U$ is a separating multiply connected component for $\mathscr{F}$, then $\pi(U)$ is star-shaped;
\item If $U$ is a non-separating multiply connected component for $\mathscr{F}$, then $\pi(U)$ is a single point.
\end{itemize}
We remark that in the third case, if $\pi(U)$ is a $k$-star, then $U$ has at least $k$ boundary components, and $\pi(U - \bigcup\mathscr{A})$ gives the center of the $k$-star.

The rational map $f:\hat\C\longrightarrow \hat\C$ naturally induces a dynamics on $\mathcal{T}_1$ by
$$
f_*(x) = \pi \circ f(\partial \pi^{-1}(x))
$$
where $\partial \pi^{-1}(x)$ is the boundary of $\pi^{-1}(x)$ in $\hat\C$ (see Theorem 3.6 in \cite{Shishikura87}).
The induced map $f_*$ is piecewise linear with integral derivative, but may collapse some interval of $\mathcal{T}_1$ to a single point.
We are going to choose a subspace $\mathcal{T}_0\subseteq \mathcal{T}_1$ as the domain, so that $f_*$ restricts to an expanding map.

Recall $U_1,..., U_k$ is a list of multiply connected Fatou components that are mapped to disk Fatou components under $f$. 
We first show that the map $f_*$ is locally injective on $\pi(A)$ for  $A \in \mathscr{A}$ if $\pi(A) \subseteq \mathcal{T}_1 - \bigcup_{j=1}^k \pi(U_j)$.
\begin{lem}\label{lem:inj}
Let $A\in \mathscr{A}$ be an annulus, then either 
\begin{itemize}
\item $A \subseteq U_j$; or
\item $f(A) \in \mathscr{A}$.
\end{itemize}

In particular, if $A\in \mathscr{A}$ and $A \not\subseteq U_j$ for all $j=1,..., k$, then $f_*$ is injective on $\pi(A)$.
\end{lem}
\begin{proof}
If $A \not\subseteq U_i$ for all $i=1,..., k$, then there exists $l\geq 1$ so that $f^l(A) \in U_j$ by construction.
Suppose for contradiction that $f(A) \notin \mathscr{A}$, then $f_*$ sends $\pi(A)$ to a single point $\pi(f(A))$.
Thus, $f_*^{l-1}$ sends the point $\pi(f(A))$ to an interval $\pi(f^l(A))$, which is a contradiction.
Therefore, $f(A) \in \mathscr{A}$.
\end{proof}

\subsection*{Construction of $\mathcal{T}_0 \subseteq \mathcal{T}_1$}
Let $\mathcal{B}_1$ denote the set of branch points of $\mathcal{T}_1$.
We call a branch point $b$ a {\em Fatou branch point} if $\pi^{-1}(b)$ is contained in a Fatou component, and a {\em Julia branch point} otherwise.
We denote this partition by
$$
\mathcal{B}_1 : = \mathcal{B}_1^F \cup \mathcal{B}_1^J.
$$

Let $U_b$ be the Fatou component associated to a Fatou branch point.
We denote $k_b$ as the smallest integer so that $f^{k_b}(U_b) \in \mathscr{U}$.
Then in the same way as in the proof of Lemma \ref{lem:inj}, $f^i(U_b)$ is separating for $\mathscr{F}$ for any $0\leq i\leq k_b$.
Let 
$$
\Omega := \{\pi(U): U\in \mathscr{U}\} \cup \{\pi(f^i(U_b)): b \in \mathcal{B}_1^F, 0\leq i\leq k_b\},
$$
then $\Omega$ is a finite union of star-shaped open subtrees of $\mathcal{T}_1$.
Each component $\mathcal{U} = \pi(U)$ of $\Omega$ contains a unique center $c=\pi(U-\bigcup\mathscr{A})$.
We denote the collection of all centers of $\Omega$ by $\mathcal{C}_\Omega$.
We define 
$$
\mathcal{T}_0 = \mathcal{T}_1-\Omega.
$$

\begin{remark}
Let $\mathcal{U}$ be a component of $\Omega$ with corresponding Fatou component $U$.
If $U \in \mathscr{U}$, then $\mathcal{U}$ is a $k$-star if and only if $U$ has $k$ boundary components, as there exists an annulus in $\mathscr{A}$ separating any pair of boundary components of $U$.
By adding finitely many annuli $A\in \tilde{\mathscr{A}}$ into the collection $\mathscr{A}$ if necessary, we assume this also holds if $U \notin \mathscr{U}$.
Such a modification is needed for $(F, R)$ to extend to a branched covering on $U$.
\end{remark}

\subsection*{Simplicial structures of $\mathcal{T}_1$}
Let 
$$
\mathcal{Q}:=\{\pi(\partial U): U\in \mathscr{F}\},
$$
and let
$$
\mathcal{P} = \mathcal{Q} \cup \partial \Omega.
$$
Note that the points in $\mathcal{P}-\mathcal{Q}$ correspond to the boundary components of $\pi(f^i(U_b))$ where $b \in \mathcal{B}_1^F$ such that $f^i(U_b)$ is {\em not} critical or post-critical.
We define the vertex set
$$
\mathcal{V}_1 := \mathcal{P} \cup \mathcal{B}_1 \cup \mathcal{C}_\Omega = (\mathcal{P} \cup \mathcal{B}_1^J) \sqcup \mathcal{C}_\Omega,
$$
and 
$$
\mathcal{V}_0 := \mathcal{P} \cup \mathcal{B}_1^J = \mathcal{V}_1 \cap \mathcal{T}_0.
$$
From the definition, we see $\mathcal{V}_1 = \mathcal{V}_0 \sqcup \mathcal{C}_\Omega$.

\subsection*{Fibers for the vertices}
For each component $\mathcal{U}$ of $\Omega$, there is a corresponding multiply connected Fatou component $U$.
Let $c$ be the center of $\mathcal{U}$, we shall call
$$
U_c := U - \bigcup \tilde{\mathscr{A}}\subseteq \pi^{-1}(c),
$$
the corresponding {\em core} of the Fatou component $U$.

We can assign a Julia component for every vertex $x\in \mathcal{V}_0$.

Indeed, if $x\in \mathcal{Q}$, then by definition, there is a critical or post-critical Fatou component $U$ and a component $X \subseteq \partial U$ such that $\pi(X) = x$.
$X$ is contained in some Julia component $K_x$, which we call the corresponding Julia component of $x$.
We note that the Julia component here is well-defined.
If another critical or post-critical Fatou component $U'$ with component $X' \subseteq\partial U'$ satisfying $\pi(X') = x$, then $X'$ is also contained in $K_x$.
As otherwise, $X$ and $X'$ are separated by some separating Fatou components for $\mathscr{F}$, which is a contradiction as $\pi(X) = \pi(X')$.
The same construction also works for $x\in \partial \Omega$. 

If $x\in \mathcal{B}_1^J$ is a Julia branch point.
Let $e$ be an edge adjacent to $x$, then either
\begin{itemize}
\item There exists an annulus $A \in \mathscr{A}$ with $A \subseteq \pi^{-1}(e)$ and a component $X\subseteq \partial A$ such that $X \subseteq \partial \pi^{-1}(x)$; or
\item There exists a sequence of nested annuli $A_i \in \mathscr{A}$ with $A_i \subseteq \pi^{-1}(e)$ accumulating to $X \subseteq \partial \pi^{-1}(x)$.
\end{itemize}
Since $x$ is a Julia branch point, in either case, $X$ is contained in some Julia component $K_x$.
The same argument as above shows that $K_x$ is well-defined.

Let $F: \mathcal{T}_0 \longrightarrow\mathcal{T}_1$ be the restriction of induced map $f_*$.
We prove the resctriction gives a tree map:
\begin{prop}\label{prop:tm}
The map $F: (\mathcal{T}_0, \mathcal{V}_0) \longrightarrow (\mathcal{T}_1, \mathcal{V}_1)$ is injective on edges.
\end{prop}
\begin{proof}
Since $f_*$ is continuous (see Lemma 3.5 in \cite{Shishikura87}), $F$ is also continuous.

Let $x\in \mathcal{V}_0$ be a vertex. If $x\in \mathcal{Q}$, then since the family $\mathscr{F}$ is forward invariant, we see $F(x) \in \mathcal{Q}$ as well. If $x\in \partial \Omega$, we see $F(x) \in \partial \Omega \cup \mathcal{Q}$ by our construction.
Thus it remains to consider the case $x\in \mathcal{B}_1^J-\mathcal{Q}$.
We claim that $F(x)$ is another branch point of $\mathcal{T}_1$, and hence $F(\mathcal{V}_0) \subseteq \mathcal{V}_1$.
\begin{proof}[Proof of the claim]
Let $K$ be the corresponding Julia component of $x$. If $f: K \longrightarrow f(K)$ has degree $1$, then $F$ locally sends edges adjacent to $b$ to different edges adjacent to $F(x)$, so $F(x)$ is another branch point.

Now consider the case $f: K \longrightarrow f(K)$ has degree $e \geq 2$.
We say a component $D$ of $\hat\C-K$ is a {\em local critical} component if $f|_{\partial D}$ has degree $\geq 2$.
The corresponding component $D'$ of $\hat\C-f(K)$ with boundary $f(\partial D)$ will be called a local critical value component.
We also say $D$ corresponds to an edge $e$ if $\pi^{-1}(e) \subseteq D$.
Since $x\notin \mathcal{Q}$, we conclude that every local critical component $D$ corresponds to an edge adjacent to $x$.
If there are three or more local critical value components, then $F(x)$ is a branch point.
Otherwise, there are exactly two local critical value components.
As we can extend $f$ to a branched covering of sphere of degree $e$, this means there are exactly two local critical components by degree counting and Riemann-Hurwitz formula.
Since $x$ is a branch point, there is another component $D$ corresponding to some edge adjacent to $x$.
By degree counting, the component bounded by $f(\partial D)$ is not a local critical value component, thus $F(x)$ is a branch point.
\end{proof}

To see $F$ is injective on edges of $\mathcal{T}_0$, we first note that $\mathcal{V}_0$ contains all branch points and $\pi(c)$ for all critical points $c$ of $f$.
Thus Theorem 3.6 in \cite{Shishikura87}  implies that $F$ is linear on each edge.
By Lemma \ref{lem:inj}, $F$ is locally injective on each edges of $\mathcal{T}_0$, so $F$ is injective on the edges of $\mathcal{T}_0$.
\end{proof}

For future reference, we also have the following
\begin{prop}
$\mathcal{V}_0$ is forward invariant.
\end{prop}
\begin{proof}
If $x\in \mathcal{P}$, then $F(x) \in \mathcal{P}$ by construction. If $x\in \mathcal{B}_1^J$, then $F(x)$ is again a branch point as in the proof of Proposition \ref{prop:tm}. Note that $F(x)$ is not a Fatou branch point, thus $F(x) \in \mathcal{B}_1^J$. So $\mathcal{V}_0$ is forward invariant.
\end{proof}

\subsection{Rational mapping scheme from rescaling limits}\label{subsec:ms}
In this subsection, we will construct a rational mapping scheme to the tree map $F: \mathcal{T}_0\longrightarrow \mathcal{T}_1$.
This can be done by quasi-conformal surgery. 
Here, we present a different construction by taking rescaling limit.
Such a construction gives a hint on how to recover the rational map from the tree mapping schemes.

\subsection*{Limits of rational maps}
Let $g_n$ be a sequence of rational map of degree $d$, we say $g_n$ converges to $g$ or $g_n\to g$ if written in homogeneous coordinate $g_n = (P_n: Q_n)$, we have $(P_n: Q_n) \to (P:Q)=(Hp:Hq)$ where $H = \gcd(P,Q)$ and $g = (p:q)$.
Note that we allow the degree of $g_n$ to drop in the limit.
A zero of $H$ is called a {\em hole} of $g$, and the set of zeros of $H$ is denoted by $\mathcal{H}(g)$.

It is well-known that if $g_n\to g$, then $g_n$ converges compactly to $g$ on $\hat\C-\mathcal{H}(g)$ (see Lemma 4.2 in \cite{DeMarco05}), and conversely, if $g_n$ converges compactly to $g$ away from a finite set, then $g_n \to g$.

The strategy of constructing the rational mapping scheme is as follows: as we stretch the map, the Julia components corresponding to the vertices in $\mathcal{V}_0$ are separated by longer and longer annuli.
Thus, we can extract the local dynamics on the Julia components by choosing appropriate normalization and taking the limit.

To implement the strategy, we consider three cases: periodic points in $\mathcal{V}_0$, strictly pre-periodic  point in $\mathcal{V}_0$ and points in $\mathcal{C}_\Omega$.
The fibers corresponding to these three cases are periodic (non-degenerate) Julia components, strictly pre-periodic (non-degenerate) Julia components and the cores of Fatou components.

\subsection*{Periodic Julia components}
Let $x_0,..., x_{p-1} \in \mathcal{V}_0$ be periodic cycle of $F$.
Let $K:=K_{x_0}$ be the corresponding Julia component of $f$.
Since $K$ is non-degenerate, i.e., not a single point, there is a post-critically finite hyperbolic rational map $g_K$ with conjugating dynamics on $J(g_K)$ (see \S 5 in \cite{McM88}).
Let $f_n$ be a stretch for $f$.
We will show 
\begin{lem}\label{lem:rlpc}
	Let  $x_0,..., x_{p-1} \in \mathcal{V}_0$ be periodic cycle of $F$. There exist rescalings $M_{x_0,n}, ..., M_{x_{p-1},n} \in \PSL_2(\C)$ so that $M_{x_{i+1},n}^{-1} \circ f_n \circ M_{x_i,n} \to g_i$ of degree $\geq 1$, and $g_K = g_{p-1} \circ g_{p-2} \circ ... \circ g_0$.
\end{lem}
\begin{proof}
We first assume that $x$ is a fixed point, thus $K$ is invariant under $f$.
Let $D$ be a component of $\hat\C-K$, then either
\begin{itemize}
\item $f(D)$ is a disk with boundary $f(\partial D)$;
\item $f(D)=\hat\C$.
\end{itemize}
Let $D$ be a component with $f(D) = \hat\C$, and $D'$ be the component bounded by $f(\partial D)$.
Let $D_n, D'_n$ be the corresponding disks for the stretch $f_n$.
We are going to perform a particular quasiconformal surgery to get $S(f_n)$ which restrict to a quasiregular branched covering between $D_n$ and $D'_n$.

Let $A \subseteq D$ be an annulus with $\partial D$ as one of its boundary component such that $f:A \longrightarrow f(A)$ is a degree $e$ covering.
Let $A_n$ be the annulus associated to $f_n$.
Then by our construction, the modulus $m(A_n) \to \infty$.
Then there exists a sequence of $L_n \in \PSL_2(\C)$ so that $L_n(A_n)$ converges to $\C^*$.
Consider the annulus $\tilde{A} = B(0,2) - \overline{B(0, \frac{1}{2})}$, then $L_n^{-1}(\tilde{A})$ is {\em well-buried} in the annulus $A_n$, i.e., the modulus of each component $A_n- L_n^{-1}(\tilde{A})$ goes to infinity.
Therefore, there exists $L'_n \in \PSL_2(\C)$ such that $L'_n \circ f_n \circ L_n^{-1}$ converges uniformly to $z^e$ on $\tilde{A}$.
We now perform an interpolation of the values between $f_n$ and $z^e$ on the two boundary components of $\tilde{A}$ and get the desired quasiregular map $S(f_n)$.
Note that in this way, $S(f_n)$ is $(1+\epsilon_n)$-quasiregular with $\epsilon_n \to 0$.

Inductively, we perform such a surgery for all components $D$ with $f(D) = \hat\C$.
Since each annulus on which we perform the interpolation is mapped to an attracting periodic cycle of $S(f_n)$, so by Shishikura's principal (Proposition \ref{prop:ShishikuraPrinciple}), $S(f_n)$ is quasiconformally conjugate by $\phi_n$ to some rational map $g_n$.

By construction, $g_n$ is hyperbolic with conjugate dynamics on the the Julia set as $g_K$.
The dynamics on the Fatou components of $g_n$ converges to the post-critically fixed one, thus after possibly composing the quasiconformal conjugacy $\phi_n$ with a sequence in $\PSL_2(\C)$, we assume $g_n \to g_K$.

Since $\phi_n$ is $(1+\epsilon_n)$-quasiconformal with $\epsilon_n \to 0$, normalizing by $M_n \in \PSL_2(\C)$, we may assume $M_n^{-1} \circ \phi_n$ converges to the identity map uniformly.
Thus, $M_n^{-1} \circ S(f_n) \circ M_n \to g_K$ uniformly.

As each annulus on which we perform the interpolation is well-buried, it converges to a point under the normalization $M_n$. Therefore, there exists a finite set $S$ so that for any compact set $\Omega \subseteq \hat\C-S$, there exists an $N = N(\Omega)$ so that $M_n^{-1} \circ S(f_n)\circ M_n = M_n^{-1} \circ f_n\circ M_n$ for all $n \geq N$.
Thus $M_n^{-1} \circ f_n\circ M_n$ converges compactly to $g_K$ away from this finite set $S$, so $M_n^{-1} \circ f_n\circ M_n \to g_K$.

More generally, if $x$ has period $p$, similar arguments as above give us $p$ sequences $M_{0,n}, ..., M_{p-1,n} \in \PSL_2(\C)$ so that $M_{i+1,n}^{-1} \circ f_n \circ M_{i,n} \to g_i$ and the associated post-critically finite rational map $g_K$ is decomposed into $g_K = g_{p-1}\circ g_{p-2}\circ...\circ g_0$.
\end{proof}

Let $x_0,..., x_{p-1} \in \mathcal{V}_0$ be periodic cycle of $F$.
Fix a choice of rescalings $M_{x_i,n}$ for $x_i$ in Lemma \ref{lem:rlpc}.
The rescalings $M_{x_i,n}$ give the coordinates for the Riemann sphere $\hat\C_{x_i}$.
We will now assign the markings $\xi_{x_i}: T_{x_i} \mathcal{T}_1 \xhookrightarrow{} \hat\C_{x_i}$.
If $e$ is an edge adjacent to $x_i$, then $\pi^{-1}(e)$ is contained in some component $D$ of $\hat\C- K_i$.
Let $B$ be a disk compactly contained in $D$. 
Then the modulus of the associated annuli $D_n-\overline{B_n}$ for $f_n$ is going to infinity.
So $M_{x_i,n}^{-1}(B_n)$ converges to a point $z_e$.
We assign the marking 
$\xi_{x_i}(\vec{v}_e) = z_e$, and the rational mapping scheme 
$$
R_{x_i} = \lim M_{x_{i+1},n}^{-1} \circ f_n \circ M_{x_i,n}: \hat\C_{x_i} \longrightarrow \hat\C_{x_{i+1}}.
$$
It is easy to check the markings are compatible with the dynamics, and the holes are contained in $\Xi_{x_i} = \xi_{x_i}(T_{x_i}\mathcal{T}_1)$.

\subsection*{Strictly pre-periodic Julia components.}
Let $x\in \mathcal{V}_0$ be a strictly pre-periodic point of $F$. Let $K:=K_x$ be the corresponding Julia component.
By induction, we may assume $F(x)$ is a periodic point.
Let $K':=K_{F(x)}$ be the associated Julia component.
Assume that $f:K \longrightarrow K'$ has degree $e$.

Let $M_{F(x), n}$ be the rescaling for $F(x)$ constructed in Lemma \ref{lem:rlpc}.
Then $M_{F(x), n}^{-1}(K'_n)$ converges to $J(g_{K'})$.
Similarly to the proof of Lemma \ref{lem:rlpc}, we perform a surgery on each component $D$ of $\hat\C-K$ with $f(D) = \hat\C$ to get $S(f_n)$ where $S(f_n)$ is a uncritical quasiregular covering between $D_n$ and $D'_n$ where $D'_n \subset \hat\C-K_n'$ is bounded by $f(\partial D_n)$.
Unlike the previous case, we are not in the dynamical setting, i.e., we treat the domain and codomain differently.
By pulling back the complex structure under $S(f_n)$, we get a sequence of degree $e$ rational maps 
$$
g_n = M_{F(x), n}^{-1} \circ S(f_n) \circ \phi_n
$$ 
where $\phi_n$ is $(1+\epsilon_n)$-quasiconformal with $\epsilon_n \to 0$.
In the codomain, since the critical values $B_n$ of $g_n$ converges to the dynamical centers of each Fatou component of $g_K'$, so $(\hat\C, B_n)$ stay in the bounded part of the moduli space of punctured spheres. 
Thus, by normalizing $\phi_n$ with a sequence in $\PSL_2(\C)$ if necessary, $g_n$ converges to a rational map $g$ of the same degree $e$.

The same argument as in Lemma \ref{lem:rlpc} gives $M_{x,n} \in \PSL_2(\C)$ so that $M_{x,n}^{-1} \circ \phi_n$ converges to the identity map.
Then $M_{F(x), n}^{-1} \circ f_n \circ M_{x,n} \to g$.

We assign the rational mapping scheme 
$$
R_x:= \lim M_{F(x), n}^{-1} \circ f_n \circ M_{x,n}: \hat\C_x \longrightarrow \hat\C_{F(x)},
$$
and define the marking on $\xi_x: T_x \mathcal{T}_1 \xhookrightarrow{} \hat\C_x$ by pulling back the markings of $\hat\C_{F(x)}$ using the rational map $g$.
By construction, the marking is compatible with the dynamics, and the holes are contained in $\Xi_{x}$.

\subsection*{Cores of Fatou components}
Let $x\in \mathcal{C}_\Omega$ represent a center of a component $\mathcal{U}$ in $\Omega$.
Let $U_x$ represent the core of corresponding Fatou component, i.e., $U_x = \pi^{-1}(x) - \bigcup\tilde{\mathscr{A}}$.
Let $U_{x,n}$ be the corresponding core for $f_n$.
Then by construction, each $U_{x,n}$ is conformally equivalent to $U_x$.
Thus, there exists a rescaling $M_{x,n} \in \PSL_2(\C)$ so that $M_{x,n}^{-1}(U_{x,n})$ converges to $U_{x, \infty}$ with the same number of boundary components.

To define the markings, if $e$ is an edge adjacent to $x$, then it corresponds to a component $D$ of $\hat\C-U_x$.
In the same way, let $B$ be a disk compactly contained in $D$. 
Then the modulus of the associated annuli $D_n-\overline{B_n}$ for $f_n$ is going to infinity, so $M_{x,n}^{-1}(B_n)$ converges to a point $z_e$.

We now prove that 
\begin{lem}\label{lem:UU'}
If $U':=f(U)$ is a multiply connected Fatou component, then $(F, R)$ extends to a branched covering between $\mathcal{U}$ and $\mathcal{U}'$.
\end{lem}
\begin{proof}
Note that $f:U\longrightarrow U'$ is a degree $e$ branched covering, thus $U'$ is separating by Lemma \ref{lem:inj} and corresponds to a component $\mathcal{U}'$ of $\Omega$.
We can extend $F:\mathcal{U} \longrightarrow \mathcal{U}'$ using the Shishikura's tree map $f_*$.
Let $M_{x,n}$ and $M_{F(x),n}$ be the two rescalings for $x$ and $F(x)$ given as above. 
Then $M_{F(x),n}^{-1} \circ f_n \circ M_{x,n}$ converges to a rational map $g$ of degree $e$ compatible with the markings.
\end{proof}

\subsection*{Convergence to tree mapping schemes}
Let $(F, R)$ be the tree mapping scheme constructed as above.
Note that different choices in the above construction give conjugate tree mapping schemes.
By construction, we have
\begin{lem}\label{lem:ctms}
	The stretch $f_n$ converges to the tree mapping scheme $(F,R)$.
\end{lem}
\begin{proof}
	Let $a\in \mathcal{V}_1$. Let $M_{a, n}\in \PSL_2(\C)$ be the rescalings constructed in \S \ref{subsec:ms}.
	Since each annulus that we perform the interpolation converges to a marked point in $\Xi_a$, the holes of $R_a = \lim M_{F(a),n}^{-1} \circ f_n \circ M_{a, n}$ are contained in $\Xi_a$. 
	Thus $M_{F(a),n}^{-1} \circ f_n \circ M_{a, n}$ converges compactly to $R_a$ on $\hat\C_a - \Xi_a$.
	It also follows from our definition that $M_{b,n}^{-1} \circ M_{a,n}$ converges to the constant map $\xi_a(\vec v_b)$ where $\vec v_b \in T_a\mathcal{T}_1$ is the tangent vector in the direction of $b$.
	Thus, $f_n$ converges to the tree mapping scheme $(F,R)$.
\end{proof}

\subsection{Hyperbolic and post-critical finiteness}\label{subsec:hp}
We shall now prove
\begin{prop}\label{prop:hpcf}
The tree mapping scheme $(F, R)$ is irreducible, post-critically finite and hyperbolic.

Moreover, if we start with a different map $g$ in the same hyperbolic component of $f$, we get topologically conjugate tree mapping schemes.
\end{prop}
\begin{proof}
Let $e=[x,y]$ be an edge in $\mathcal{T}_0$. Then there is a component $D$ of $\hat\C-K_x$ that corresponds to $e$.
The local degree at $z_e \in \hat\C_x$ for the rational mapping scheme $R$ equals to the degree of $f|_{\partial D}$.
Similarly, $e$ also corresponds to a component $D'$ of $\hat\C-K_y$, and the local degree at $z_e \in \hat\C_y$ is the degree of $f|_{\partial D'}$.
Let $A = D\cap D'$ be the annulus.
By our construction, the projection $\pi(c)$ of any critical point $c$ is contained in $\mathcal{V}_0$, so $A$ contains no critical points. 
Thus $f$ is a covering between $A$ and $f(A)$.
In particular, $(F, R)$ satisfies the condition i).
The condition ii) also follows from the above observation (see Theorem 3.6 in \cite{Shishikura87}).

To prove the condition iii) on extension, we note that any component $\mathcal{U}$ corresponds to some multiply connected Fatou component $U$. Then either $f(U)$ is multiply connected or is a disk. In the first case, $(F, R)$ extends to a branched covering by Lemma \ref{lem:UU'}, and thus corresponds to case (b) in the condition iii).
It's easy to verify that the second case gives case (a) in the condition iii).
The moreover part comes from the fact that any multiply connected Fatou component is eventually mapped to some disk Fatou component.

A critical point for the rational mapping scheme is exposed if and only if it corresponds to a critical point in the disk Fatou component.
On the other hand, given an edge $e$ adjacent to $x$, the corresponding component $D$ of $\hat\C-K_x$ contains other Julia components, thus $D$ is not a disk Fatou component for $f$.
Therefore, conditions iv) and v) follow from dynamics on the associated Fatou components for $f$.
Since each annulus $A$ in $\mathscr{A}$ is eventually mapped inside of a disk Fatou component, we conclude any edge of $\mathcal{T}_0$ is eventually mapped outside of $\mathcal{T}_0$, so the condition vi) is satisfied by Proposition \ref{prop:nonE}.
By our construction, $(F, R)$ is irreducible.

If $g$ is another map in the same hyperbolic component of $f$, then $f$ and $g$ are quasiconformally conjugate near the Julia set. which gives the topological conjugacy between the tree map schemes.
\end{proof}
We remark that the extension to a component of $\mathcal{U}$ of $\Omega$ depends on the choice of the rational map in the hyperbolic component we start with.
Indeed, such a component  $\mathcal{U}$ corresponds to a Fatou component, and the dynamics on the Fatou set are in general {\em not} topologically conjugate for two rational maps in the same hyperbolic component.
This is the main reason why in the definition of tree mapping schemes $(F, R)$, we only record the tree map $F$ and rational mapping schemes $R$ on a sub-forest $\mathcal{T}_0$.

\begin{proof}[Proof of Theorem \ref{thm:cvg}]
	Combining Lemma \ref{lem:ctms} and Proposition \ref{prop:hpcf}, we have the result.
\end{proof}

\section{From tree mapping scheme to rational map}\label{sec:tmsr}
In this section, we show any irreducible post-critically finite hyperbolic tree mapping scheme $(F, R)$ arises as such a limit by constructing a hyperbolic rational map whose stretch converges to $(F, R)$.

The general idea for the construction is simple.
Given a tree mapping scheme, we can thicken the edges to long annuli and get a Riemann sphere $\hat\C$ (\S \ref{sec:te}).
We construct a model map $\tilde{f}$ on $\hat\C$ from the tree mapping scheme (\S \ref{sec:qrm}).
Since the tree mapping scheme is post-critically finite hyperbolic, we show that the model map can be chosen to be holomorphic away from some non-recurring open set $U$, and quasiregular on $U$.
By Shishikura's principle, the model map is quasiconformally conjugate to a rational map $f$ (\S \ref{sec:qrmq}).
By construction, it is easy to verify the stretch of $f$ converges to $(F,R)$, and different choices of the construction give J-conjugate rational maps (\S \ref{sec:JConj}).

\subsection{Riemann sphere from thickening the edges.}\label{sec:te}
Let $(F, R)$ be an irreducible post-critically finite hyperbolic tree mapping scheme.
Let $a\in \mathcal{V}_0$. 
Then $a$ is either periodic or strictly pre-periodic.
If $a$ is periodic of period $p$, we define the {\em Julia set} $J_a$ as the Julia set of $R^p|_{\hat\C_a}$.
If $a$ is strictly pre-periodic with $F^l(a)$ periodic, we define the Julia set $J_a$ as the pull back of $J_{F^l(a)}$ by $R^l|_{\hat\C_a}$.
The complement $\hat\C_a - J_a$ will be called the {\em Fatou set} in $\hat\C_a$.

Since each marked point is eventually mapped to a periodic critical cycle, there exists $r_0>0$ so that for any $a\in \mathcal{V}_0$, and for any marked point $z\in \Xi_a \subset\hat\C_a$ or any critical point $z \in \hat\C_a$,
$$
R^k|_{\hat\C_a} (B_{S^2}(z,r_0))
$$
is contained in the union of Fatou components of $\hat\C_{F^k(a)}$ for all $k$.
Here $B_{S^2}(z,r_0)$ is the ball of radius $r_0$ centered at $z$ with the standard spherical metric on $S^2 \cong \hat\C_a$.

Let $a\in \mathcal{V}_1$ and let
$$
\Pi_a := \hat\C_a - \bigcup_{z\in \Xi_a} B_{S^2}(z,r_0).
$$
Shrink $r_0$ if necessary, we may assume any exposed critical point or exposed post-critical point is contained in $\Pi_a$.

Given an edge $e=[a,b]$ of $\mathcal{T}_1$ of length $L$, we define
$$
\Pi_{e,n}:= B(0,1) - \overline{B(0, \exp(-2\pi nL))}
$$
as the round annulus of modulus $nL$.
The components of the boundary of $\Pi_{e,n}$ are marked by $a$ and $b$.
A point $x\in e$ can be identified as a round circle centered at $0$, denoted by  $C_{x,n}\subseteq\Pi_{e,n}$.
Indeed, if $d(a,x) = L_1$ and $d(x,b) = L_2$, then the circle $C_{x,n}$ is the unique one that has modulus $nL_1$ and $nL_2$ to the boundary of $\Pi_{e,n}$ associated to $a$ and $b$ respectively.

Let $e=[a,b]$ be an edge. We can glue $\Pi_{e,n}$ with $\Pi_a$ along their boundaries (see Figure \ref{fig:ConstructionRational}).
Assume without loss of generality that the boundary component $\partial B(0,1)$ of $\Pi_{e,n}$ corresponds to $a$.
We can identify $B_{S^2}(z,r_0)$ with $B(0,1)$ by a M\"obius transformation sending $z$ to $0$. 
We fix such an identification, which is unique up to post composing with a rotation.

By gluing all $\Pi_a$ and $\Pi_{e,n}$ in this way, we get a complex manifold 
$$
	\hat\C_{\Pi,n} = \bigcup_{a\in \mathcal{V}_1} \Pi_a \cup \bigcup_{e\in \mathcal{E}_1} \Pi_{e,n}.
$$
By the uniformization theorem, $\hat\C_{\Pi,n}$ is a Riemann sphere.

We remark that for $a\in \mathcal{V}_1$, there are $\nu(a)$ different circles associated to $a$, corresponding to $\partial \Pi_a$.
Since the circle domain $\Pi_a$ stays the same for all $n$, the modulus of the annulus between any pair of components of $\partial \Pi_a$ stays the same, in particular, is uniformly bounded.
We make a choice and let $C_{a,n} = C_a$ denote a particular boundary component of $\partial \Pi_a$.

Since the gluing maps are M\"obius, round circles in $\Pi_a$ or $\Pi_{e,n}$ are round circles in $\hat\C_{\Pi,n}$.
The following lemma follows from our construction:
\begin{lem}\label{lem:approxM}
There exists $M>0$, so that for any $n$, and for any pairs $x, y \in \mathcal{T}_1$ with $d(x,y) = L$, the modulus of the annulus $A_n$ between $C_{x,n}$ and $C_{y,n}$ in $\hat\C_{\Pi,n}$ satisfies
$$
nL \leq m(A_n) \leq nL+M.
$$
\end{lem}
\begin{proof}
If both $x, y$ are in the same edge, then $m(A_n) = nL$.
In general, the largest round annulus bounded by $C_{x,n}$ and $C_{y,n}$ has modulus bounded by $nL+M_1$ for some constant $M_1$.
Since every annulus with modulus not too small can be uniformly approximated by the round annulus (see Theorem 2.1 in \cite{McM94}), we conclude the result.
\end{proof}

\begin{figure}[ht]
    \centering
    \resizebox{0.7\linewidth}{!}{
    \def\svgwidth{\columnwidth}
    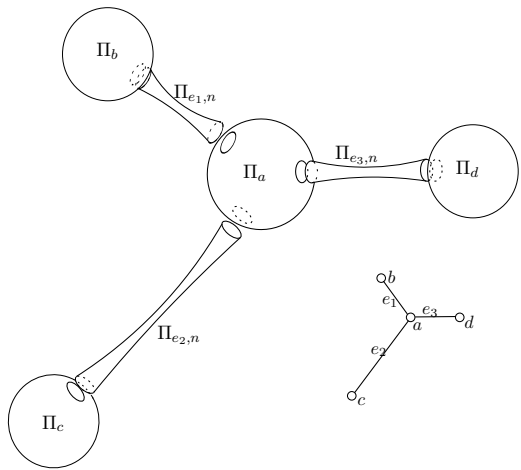

    }
    \caption{The construction of Riemann sphere by `thickening' the edges of the tree. The annulus $\Pi_{e,n}$ for $e\in E_1$ becomes longer as $n\to \infty$, but $\Pi_x$ for $x\in \mathcal{V}_1$ stays the same.}
    \label{fig:ConstructionRational}
\end{figure}

Let $e=[a,b]$ be an edge in $\mathcal{T}_0$.
Since the non-escaping set $\mathcal{J}$ of the tree map $F: \mathcal{T}_0 \longrightarrow \mathcal{T}_1$ is totally disconnected,
we can further partition the edge $e = [a,b]$ into alternating pieces 
$$
a\leq \alpha_1(e) < \beta_1(e) \leq... \leq \alpha_{k_e}(e) < \beta_{k_e}(e) \leq b
$$ 
such that for each $(\alpha_i(e), \beta_i(e))$, there exists a $q\geq 0$ and a component $\mathcal{U}$ of $\Omega = \mathcal{T}_1 - \mathcal{T}_0$ so that
$F^q((\alpha_i(e),\beta_i(e)))\subseteq \mathcal{U}$ and $F^q(\alpha_i(e)), F^q(\beta_i(e)) \in \mathcal{V}_1$.
Moreover, we assume that image $F([a, \alpha_1(e)+\delta_e])$ (and $F([\beta_{k_e}(e)-\delta_e, b])$) is contained in an edge of $\mathcal{T}_1$ for some $\delta_e > 0$.

By adding the (finite) orbits of such small intervals $(\alpha_i(e), \beta_i(e))$ in the partition, we may assume that such intervals are forward invariant (whenever it is defined).
More precisely, for any edge $e$ and $i$, either $F((\alpha_i(e), \beta_i(e)))$ is contained in $\Omega$ or there exist $e'$ and $i'$ so that 
$$
F((\alpha_i(e), \beta_i(e))) = (\alpha_{i'}(e'), \beta_{i'}(e')).
$$

\subsection{The quasiregular model}\label{sec:qrm}
We can now construct a quasiregular model map $\tilde{f}_n$ on $\hat\C_{\Pi, n}$.
The construction is done in $4$ steps.
The idea is to construct the map on $\Pi_a$ using the rational map $R_a$, and then interpolate along edges and the gaps $\Omega$ (see Figure \ref{fig:ConstructionF}).

Let us fix a small number $\epsilon>0$ so that $[\alpha_i(e), \beta_i(e)]$ has length at least $4 \epsilon$ for all $e$ and $i$.
We also assume $4\epsilon < \delta_e$ for all $e$ and the length of each edge in $\mathcal{T}_1 - \mathcal{T}_0$ is at least $4\epsilon$.

\begin{figure}[ht]
    \centering
    \resizebox{0.7\linewidth}{!}{
    \def\svgwidth{\columnwidth}
    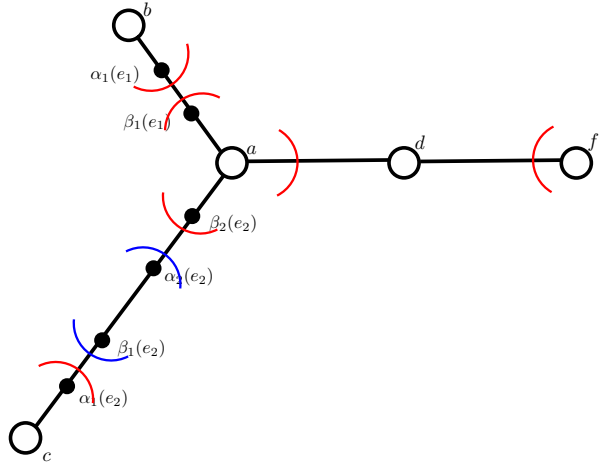

    }
    \caption{An illustration of the construction of $\tilde f_n$. Here $\Omega = (a,f)$. In step 1, we define the map using rational maps on the region bounded by the red arcs. In step 2, we define the map as holomorphic coverings on the region bounded by the blue arcs. The interpolations in step 3 fill in the gaps on the edges of $\mathcal{T}_0$, and step 4 fill in the gaps in $\Omega$.}
    \label{fig:ConstructionF}
\end{figure}

{\bf Step 1:}
Let $a\in \mathcal{V}_0$, we first define the map on a neighborhood of $\Pi_a$.
Let $\Phi_a: \hat\C_a \longrightarrow \hat\C_{\Pi, n}$ and $\Phi_{F(a)}: \hat\C_{F(a)} \longrightarrow \hat\C_{\Pi, n}$ be the conformal maps extending the inclusion maps of $\Pi_a$ and $\Pi_{F(a)}$ to $\hat\C_{\Pi,n}$.
Let $e_1,..., e_k$ be a list of edges of $\mathcal{T}_1$ adjacent to $a$.
Assume that the local degree at the tangent vector associated to $e_j$ is $d_j$.
By switching the orientation of the edges, we assume $\alpha_1(e_1),..., \alpha_1(e_k)$ are the closest points in the partition defined above to $a$. Let $x_j = \alpha_1(e_j) + \epsilon/d_j \in (\alpha_1(e_j), \beta_1(e_j))$.

If $x_j\in \mathcal{T}_0$, we let $y_j = F(x_j)$ and $C_{y_j, n}$ be the corresponding circle in $\hat\C_{\Pi, n}$. 
Otherwise, $x_j\in \mathcal{U}$ for some component $\mathcal{U}$ of $\Omega$. 
If $F$ extends geometrically to $\mathcal{U}$, then we let $y_j = F(x_j)$ and $C_{y_j, n}$ be the corresponding circle in $\hat\C_{\Pi, n}$. 
Otherwise, $R_a(z_j) = z_v \in \Pi_{F(a)} \subseteq \hat\C_{F(a)}$ where $z_j \in \Xi_a$ corresponds to $e_j$. 
We define $C_{y_j, n}$ as the unique round circle in $B_{S^2}(z_v, r_0)$ centered at $z_v\in\Pi_{F(a)} \subseteq \hat\C_{\Pi, n}$ so that the modulus of the annulus between $\partial B_{S^2}(z_v, r_0)$ and $C_{y_j, n}$ is $\epsilon n$.
For sufficiently large $n$, there exists a component $\tilde{C}_{x_j, n}$ of $(\Phi_{F(a)} \circ f_{a\to F(a)} \circ \Phi_{a}^{-1})^{-1}(C_{y_j, n})$ that approximates $C_{x_j, n}$.
More precisely,  $\tilde{C}_{x_j, n}$ and $C_{x_j,n}$ are contained in some annulus with bounded modulus as essential closed curves.
Let $\Pi_{x_1,..., x_k,n} \subseteq \hat\C_{\Pi, n}$ be the domain bounded by $\tilde{C}_{x_j, n}$.
We define 
$$
\tilde {f}_n = \Phi_{F(a)} \circ R_a \circ \Phi_{a}^{-1}: \Pi_{x_1,..., x_k,n} \longrightarrow \hat\C_{\Pi, n}.
$$

{\bf Step 2:}
We now define the map on a collection of annuli on $\Pi_{e,n}$ by appropriate holomorphic coverings.
More precisely, let $\alpha< \beta\leq\alpha' < \beta'$ be adjacent partitions in the edge $e$ of $\mathcal{T}_0$.
Let $x=\beta - \epsilon/d_e$ and $y = \alpha' + \epsilon/d_e$, where $d_e = \deg(e)$.
Let $A_n$ be the annulus bounded by $C_{x, n}$ and $C_{y, n}$.
By Lemma \ref{lem:approxM}, we can find $x'_n$ and $y'_n$ so that 
\begin{enumerate}
\item The annulus $A'_n$ bounded by $C_{x'_n, n}$ and $C_{y'_n, n}$ has modulus $d_e m(A_n)$;
\item $x'_n \to F(x)$ and $y'_n \to F(y)$ as $n\to \infty$.
\end{enumerate}
We define $\tilde {f}_n$ as a degree $d_e$ holomorphic covering between $A_n$ and $A'_n$.

{\bf Step 3:}
We now interpolate using quasiregular maps and define the map on the gaps in $\Pi_{e,n}$.
Let $(\alpha, \beta)$ be a partition in an edge $e$ of $\mathcal{T}_0$.
Let $x= \alpha+\epsilon/d_e$ and $y = \beta - \epsilon/d_e$.
Note that $x, y$ correspond to boundary points in the previous two steps, and let $A_n$ be the annulus bounded by the corresponding curves in the previous two steps.
Then $\tilde{f}_n$ is defined on $\partial A_n$ by continuity.
Let $A'_n$ be the annulus bounded by $\tilde{f}_n(\partial A_n)$ and we define $\tilde {f}_n$ as a degree $d_e$ quasiregular covering map between $A_n$ and $A_n'$ interpolating the boundary values.

{\bf Step 4:}
Finally, we interpolate using quasiregular maps and define the map on the gaps associated to $\Omega$.

Let $\mathcal{U}$ be a component of $\Omega$.
Let $\{a_1,..., a_k\}$ be its boundary.
By condition iii), either $(F,R)$ extends to a branched covering, or $F(\partial U)$ is a single point.
We will assume we are in the first case, and the second case can be treated in a similar way.

Let $\mathcal{W}$ be a component of $\Omega$ so that $F$ extends to a map between $\mathcal{U}$ and $\mathcal{W}$.
Let $x_j = a_j+\epsilon/d_e$ and $\tilde{C}_{x_j,n}$ be the Jordan curve that is used in the first step.
Let $b_i\in \partial \mathcal{W}$ with adjacent edge $e \subseteq \mathcal{W}$.
Let $y_i\in e$ with distance $\epsilon$ away from from $b_i$.
By the condition iii) and the construction in the first step, there exists a circle $C_{y_i,n} \subseteq \Pi_{e, n}$ so that
$C_{y_i,n} = \tilde{f}_n(\tilde{C}_{x_j,n})$ for all $j$ with $F(a_j) = b_i$ (and hence $F(x_j) = y_i$).
Let $U_n$ and $W_n$ be the domain bounded by $\tilde{C}_{x_j,n}$ and $C_{y_i, n}$ respectively.
We define $\tilde {f}_n$ as a proper quasiregular map between $U_n$ and $W_n$ interpolating the boundary values.
Note that the existence of such a quasiregular interpolation is guaranteed by the condition iii).
More concretely, let $R_c$ be the extension of the rational mapping scheme to the center $c\in \mathcal{U}$. The quasiregular map can be defined as a perturbation of $R_c$.

\subsection{The quasiregular model is quasirational.}\label{sec:qrmq}
After the above four steps, we get a sequence of quasiregular map $\tilde{f}_n : \hat\C_{\Pi, n} \longrightarrow \hat\C_{\Pi,n}$. We shall now prove that for sufficiently large $n$, $\tilde{f}_n$ is quasiconformally conjugate to a rational map.

Let $z \in \hat\C_a$ be an exposed periodic critical point of period $p$.
By the first step of the construction and our choice of $r_0$, there exists a disk $D \subseteq \Pi_a$ containing $B_{S^2}(z,r_0) \subseteq \Pi_a \subseteq \hat\C_{\Pi, n}$ on which $\tilde{f}_n^p: D \longrightarrow D$ is a uni-critical branched covering.
(Here the spherical metric comes from the inclusion of $\Pi_a \xhookrightarrow{} \hat\C_{\Pi, n}$.)
We call such a disk $D$ a {\em periodic disk Fatou component} of $\tilde{f}_n$.

\begin{prop}\label{prop:nonR}
Let $U_n$ be the annulus or multiply connected domain in the third step or the fourth step.
There exists an $N$ large enough so that for all $n\geq N$, each $U_n$ is eventually mapped compactly into the periodic disk Fatou component of $\tilde{f}_n$.
In particular, $U_n$ is non-recurring under $\tilde{f}_n$.
\end{prop}
\begin{proof}
The proof is essentially by induction.
Let $\mathcal{U}$ be a component of $\Omega$.
We first consider the case where $F(\partial \mathcal{U}) = \{v\}$.
Denote $\partial \mathcal{U} = \{a_1,..., a_k\}$ and let $e_j$ be the edge in $\mathcal{U}$ adjacent to $a_j$ with local degree $d_j$.
We choose $y_j$ in the interior of the edge $e_j$ so that $d(a_j, y_j) < \epsilon/d_{e_j}$, and let $\mathcal{U}_{y_1,..., y_k} \subset \mathcal{U}$ bounded by $y_1,..., y_k$.
Let $U_{y_1,..., y_k, n}$ be the domain in $\hat\C_{\Pi,n}$ bounded by $C_{y_j, n}$.
Note that the condition $d(a_j, y_j) < \epsilon/d_{e_j}$ guarantees that the domain $U_n$ in the fourth step of the construction associated to $\mathcal{U}$ is compactly contained in $U_{y_1,..., y_k, n}$.

Since $F(\partial \mathcal{U}) = \{v\}$, the marked points $z_{a_j} \in \hat\C_{a_j}$ associated to $e_j$ are all mapped to $z_v \in \hat\C_v$.
Moreover, $z_v$ has an exposed orbit and is pre-periodic to a periodic exposed critical cycle.
Hence the orbit of $z_v$ under $\tilde{f}_n$ are contained in $\bigcup_{a\in \mathcal{V}_0} \Pi_a$.
By the construction of $\tilde{f}_n$, as $n\to \infty$, the image $\tilde{f}_n(U_{y_1,..., y_k, n})$ converges to $z_v$, where the convergence is taken in $\Pi_v \subseteq \hat\C_v$.
Hence, for sufficiently large $n$, $U_{y_1,..., y_k, n}$ is eventually mapped into periodic disk Fatou component.
In particular, $U_n$ is eventually mapped into periodic disk Fatou component.

Now if $F$ extends to a branched covering between $\mathcal{W}$ and $\mathcal{U}$, denote $\partial \mathcal{W} = \{b_1,..., b_l\}$.
Then we can choose $z_1,..., z_l$ in the interior of edges associated to $\partial \mathcal{W}$ with $d(z_j, b_j) < \epsilon / d_j$ so that
$F(\mathcal{W}_{z_1,..., z_l})$ is compactly contained in $\mathcal{U}_{y_1,..., y_k}$, and $W_{z_1,...,z_l,n}$ be the corresponding domain.
Then a similar argument shows that for sufficiently large $n$, 
$$
\tilde{f}_n (W_{z_1,...,z_l,n}) \subseteq U_{y_1,..., y_k,n}.
$$

Thus inductively, the domain $W_n$ in the fourth step associated to $\mathcal{W}$ is eventually mapped into periodic disk Fatou component for sufficiently large $n$.

The argument for annuli $A_n$ in the third step is similar:
there exists $n$ large enough so that 
$$
\tilde{f}_n^m (A_n) \subseteq U_{y_1,..., y_k, n}
$$
for some $m$ and $U_{y_1,..., y_k, n}$ as in the previous induction step.

Since there are only finitely many such domains to consider, by taking the largest among all, the proposition is proved.
\end{proof}

We can now prove
\begin{theorem}\label{thm:inverse}
There exists an $N$ such that $\tilde{f}_n$ is quasiconformally conjugate to a hyperbolic rational map $f_n$ with finitely connected Fatou set for all $n\geq N$.
\end{theorem}
\begin{proof}
Let $N$ be as in Proposition \ref{prop:nonR}.
Since the quasiregular map $\tilde{f}_n$ is holomorphic away from union of domains in Proposition \ref{prop:nonR}, by Shishikura Principle (see Proposition \ref{prop:ShishikuraPrinciple}), $\tilde{f}_n$ is quasiconformally conjugate to a rational map $f$.

The rational map is hyperbolic as the critical points are either exposed (i.e., contained in $\Pi_a$ for some $a\in \mathcal{V}_0$), in which case, they are pre-periodic to periodic exposed critical cycles; or contained in domain $U_N$ in the fourth step of the construction, in which case, they are eventually mapped to periodic Fatou set.
The rational map has finitely connected Fatou set as any periodic Fatou component is simply connected.
\end{proof}

It is easy to verify by construction that $f_n$ converges to $(F,R)$.
Moreover, if we fix $f = f_N$, then the stretch of $f$ also converges to $(F,R)$.

\subsection{J-conjugacy}\label{sec:JConj}
There are many choices made in the above construction.
We shall now prove that different choices give J-conjugate rational maps.
The strategy is as follows.
The tree mapping scheme gives a conjugacy on finitely many Julia components corresponding to the vertices $\mathcal{V}_0$. 
This conjugacy can be extended to a global $K$-quasiconformal map.
Using an adapted version of the standard pull back argument, we shall construct a sequence of $K$-quasiconformal maps which restricts to conjugacies on preimages of those Julia components.
The compactness theorem for quasiconformal maps allows us to pass to a limit, giving a $K$-quasiconformal map conjugating the dynamics on the whole Julia sets.

Let $(F, R)$ be an irreducible hyperbolic and post-critically finite tree mapping scheme, and let $f$ and $g$ be two hyperbolic rational maps constructed as above.
Note that each vertex $a\in \mathcal{V}_0$ corresponds to Julia components $J_{a, f}$ and $J_{a, g}$, and we have a quasiconformal conjugacy near $J_{a, f}$ and $J_{a, g}$ for all $a\in \mathcal{V}_0$.
These conjugacies can be pasted together and we have a global quasiconformal map $\psi:\hat\C\longrightarrow\hat\C$ which restricts to conjugacies between $\bigcup_{a\in \mathcal{V}_0} J_{a, f}$ and $\bigcup_{a\in \mathcal{V}_0}J_{a,g}$.
Since both $f$ and $g$ are post-critically finite restricted to disk Fatou components, we may assume that $\psi$ is a conjugacy on disk Fatou components.

We shall now use a pull back argument to construct a sequence of $K$-quasiconformal maps $\psi_n$ on $\hat\C$ which restricts to a conjugacy between 
$$
\bigcup_{k=0}^n f^{-k}(\bigcup_{a\in \mathcal{V}_0} J_{a,f}) \text{ and } \bigcup_{k=0}^n g^{-k}(\bigcup_{a\in \mathcal{V}_0} J_{a,g}).
$$ 

Let $e =[a,b]$ be an edge of $\mathcal{T}_0$.
We define $A_{e,f}$ as the annulus between $J_{a, f}$ and $J_{b,f}$, and similarly for $A_{e,g}$.
For $a\in \mathcal{V}_0$ and an edge $e$ adjacent to $a$, we can associate $S = \partial A_{e,f} \cap J_{a,f}$, and similarly for $g$.
Let $S'$ be a component of $f^{-1}(S)$ that is contained in $J_{v,f}$ for some $v\in \mathcal{V}_0$.
If $S'$ bounds a disk $D$ that is disjoint from $J_{v,f}$ for all $v\in \mathcal{V}_0$, we will call $D$ an {\em exposed disk}.
We enumerate all exposed disks as $D^1_{1,f},..., D^1_{n_1,f}$ for all $a\in \mathcal{V}_0$ and edges adjacent to it.
Similarly, we have $D^1_{1,g},..., D^1_{n_1,g}$.
Since $f$ and $g$ are conjugate on $J_v$ for all $v\in \mathcal{V}_0$, we can assume that the subindices are compatible, i.e., $\psi(\partial D^1_{j,f}) = \partial D^1_{j,g}$.
Note that there are no critical points in $D^1_{j,f}$, thus $f$ is univalent on $D^1_{j,f}$ (and similarly for $g$).
Inductively, we define  $D^{k+1}_{j,f}$ as a disk component of the preimage of $D^k_{i,f}$ with boundary $\partial D^{k+1}_{j,f} \subseteq J_v$ for some $v\in \mathcal{V}_0$, and enumerate all such disks as
$D^{k+1}_{1,f},..., D^{k+1}_{n_{k+1},f}$ (and similarly as $D^{k+1}_{1,g},..., D^{k+1}_{n_{k+1},g}$ with compatible subindices).
We remark that these are exposed disks and that $f$ is univalent on $D^{k+1}_{j,f}$ (and similarly for $g$).

If $e\in E_0$, then $f(A_{e,f})$ is the annulus between $J_{f(a)}$ and $J_{f(b)}$ (and similarly for $g$).
For $e\in E_0$, we choose a core sub-annulus $\tilde{A}_e \subseteq A_{e,g}$ which is contained in a Fatou component.
We define a Dehn twist $T_e$ supported on $\tilde{A}_{e'}$ for some $e'\in E_0$ in the image $F(e)$ so that
$$
T_e \circ \psi \circ f \sim g \circ \psi
$$
on $A_{e, f}$ relative to the boundary $\partial A_{e,f}$.
Thus, we can lift $T_e \circ \psi : f(A_{e,f}) \longrightarrow g(A_{e,g})$ to a map from $A_{e,f}$ to $A_{e,g}$ homotopic to $\psi$ relative to the boundary.

Let $\psi_0 = \psi$. 
We define
$$
\psi_1 = 
  \begin{cases}
   \text{lift of $T_e \circ \psi_0$}, & \text{for } z\in A_{e,f}, e\in E_0 \\
    \text{lift of $\psi_0$}, & \text{for } z\in D^1_{j,f}, j=1,..., n_1\\
    \psi_0, & \text{otherwise} 
  \end{cases}.
$$

Note that the lift $\psi_1 \sim \psi$ on $A_{e,f}$ relative to the boundary for all $e\in E_0$.
Moreover, $\psi_1 = \psi$ on those Fatou components associated to components of $\Omega \subseteq \mathcal{T}_1$.
So we define 
$$
\psi_n = 
  \begin{cases}
   \text{lift of $T_e \circ \psi_{n-1}$}, & \text{for } z\in A_{e,f}, e\in E_0 \\
    \text{lift of $\psi_{n-1}$}, & \text{for } z\in D^i_{j,f}, i=1,...,n, j=1,..., n_i\\
    \psi_{n-1}, & \text{otherwise} 
  \end{cases}.
$$

Inductively, we have $\psi_n \sim \psi$ on $A_{e, f}$ relative to the boundary for all $e\in E_0$.
Moreover, $\psi_m = \psi_n$ on $\bigcup_{k=0}^n f^{-k}(\bigcup_{a\in \mathcal{V}_0} J_{a,f})$ for all $m\geq n$, so $\psi_n$ restricts to conjugacy between $\bigcup_{k=0}^n f^{-k}(\bigcup_{a\in \mathcal{V}_0} J_{a,f})$ and $\bigcup_{k=0}^n g^{-k}(\bigcup_{a\in \mathcal{V}_0} J_{a,g})$.

We claim that there exists a $K$ so that $\psi_n$ is $K$-quasiconformal.
Indeed, away from the backward orbits of the core sub-annuli $\tilde{A}_e$ for $e\in E_0$, $\psi_n$ has the same dilation as $\psi_{n-1}$ as $\psi_n$ is either $\psi_{n-1}$ or a conformal lift of $\psi_{n-1}$.
Thus away from the backward orbits of $\tilde{A}_e$, the dilation of $\psi_{n}$ is bounded by the dilation of $\psi_0 = \psi$.
But $\tilde{A}_e$ is not recurrent, thus the dilation on its backward orbit it is bounded by the dilation of Dehn twist $T_e$ composing with $\psi$.
Thus, $\psi_n$ is $K$-quasiconformal for some constant $K$ independent of $n$.

Since $K$-quasiconformal maps form a compact set, up to passing to a subsequence, we get a limiting $K$-quasiconformal map $\Psi: \hat\C \longrightarrow \hat\C$.
Note that $\Psi$ restricts to a conjugacy between 
$$
\bigcup_{k=0}^\infty f^{-k}(\bigcup_{a\in \mathcal{V}_0} J_{a,f}) \text{ and } \bigcup_{k=0}^\infty g^{-k}(\bigcup_{a\in \mathcal{V}_0} J_{a,g}).
$$ 
Note that the union is dense in $J(f)$ and $J(g)$, thus $\Psi$ restricts to a conjugacy between $J(f)$ to $J(g)$.

Conversely, if $g$ is a hyperbolic rational map which is J-conjugate to $f$, then the stretches of $g$ and $f$ are J-conjugate.
It is not hard to check that the associated tree mapping schemes are topologically conjugate (where the conjugacies on $\hat\C_a$ for $a\in \mathcal{V}_0$ comes from the conjugacies on the associated Julia components).

We remark that there exist J-conjugate rational maps that are in different hyperbolic components (see Figure \ref{fig:JConjugate}).

\begin{figure}[ht]
    \includegraphics[width=0.4\textwidth]{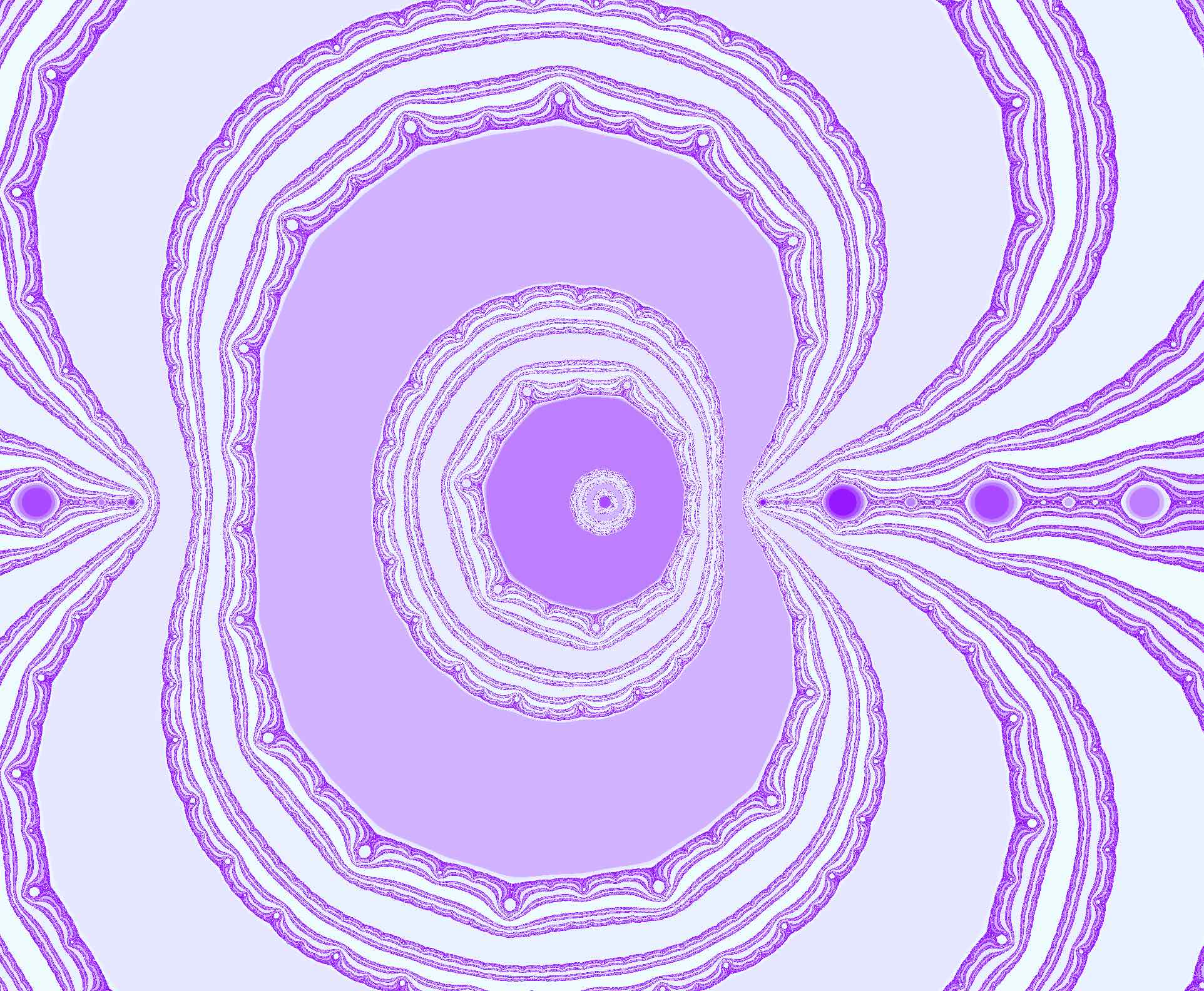}
    \includegraphics[width=0.4\textwidth]{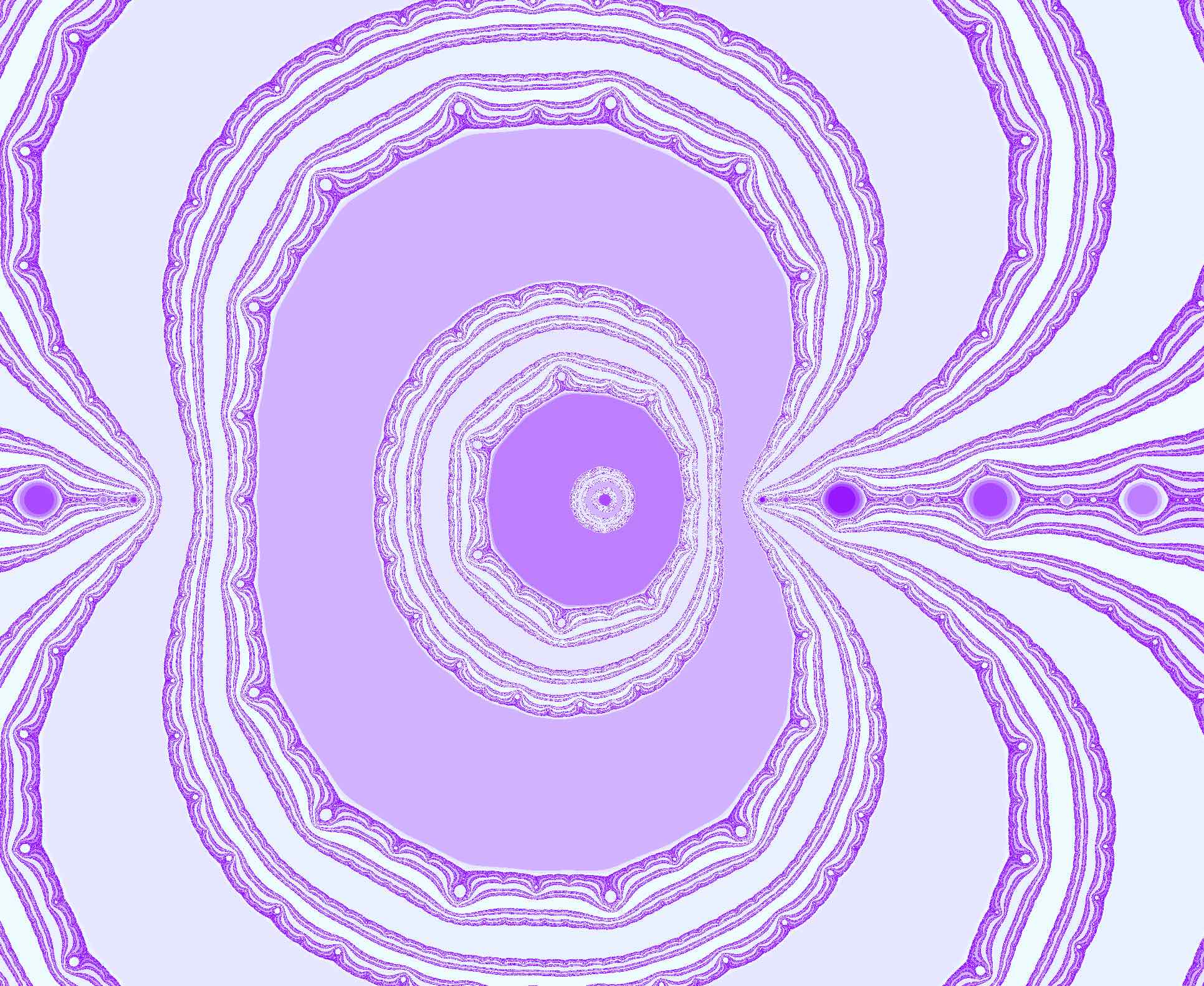}
    \caption{The Julia sets of two J-conjugate hyperbolic rational maps, but are in different hyperbolic components. The rational map is of the form $f(z) = \frac{z^3-3(\frac{3}{4})^{\frac{1}{3}}t^{\frac{1}{3}}z^2}{\frac{9}{2}z-3}+\frac{t^3-3(2^{-\frac{1}{3}})t^{\frac{7}{3}}z}{-2z^3+3tz^2}$ for sufficiently small $t$. There are two choices of cubic roots which gives the identification of the formula with $\Z_3 \times \Z_3$. $(0,0)$ is in the same hyperbolic component as $(1,1)$ by moving $t$ around $0$ once. This turns out to be the only additional identification, thus the formula gives 3 J-conjugate hyperbolic components. One can see the Julia components are twisted in a different (and incompatible) rate. This phenomenon will be studied in detail in the subsequent paper \cite{L20}.}
    \label{fig:JConjugate}
\end{figure}

Theorem \ref{thm:classification} now follows:
\begin{proof}[Proof of Theorem \ref{thm:classification}]
By Theorem \ref{thm:inverse}, we can construct a hyperbolic rational map $f$ whose stretch converges to $(F,R)$.
By the above discussion, any two such constructions give J-conjugate hyperbolic rational maps.
\end{proof}

\section{Julia components}\label{sec:JC}
In this section, we discuss the connections between the Julia set and the tree mapping scheme.
Let $f$ be a hyperbolic rational map with finitely connected Fatou set, and $(F, R)$ be the corresponding tree mapping scheme.
There are three sets of properties we can associate to a Julia component $K$ of $f$.
\begin{itemize}
\item Inherently, the Julia component $K$ is said to be {\em degenerate} if it is a single point, and {\em non-degenerate} otherwise.
If $K$ is neither degenerate nor a Jordan curve, then we say $K$ is of {\em complex type}.
\item We say a Julia component $K$ is {\em buried} if $K$ does not meet the boundary of any Fatou component, and {\em unburied} otherwise.
\item If there exists a number $p$ such that $f^p(K) = K$, then $K$ is said to be {\em periodic}, and the smallest such $p$ is called its {\em period}. It is said to be {\em preperiodic} if there exists $l\geq 0$ so that $f^l(K)$ is periodic. Otherwise, it is said to be {\em wandering}.
\end{itemize}

Since $f$ is a hyperbolic rational map with finitely connected Fatou set, we have
\begin{itemize}
\item Any degenerate Julia component is buried (this follows from the definition and the classification of Fatou components);
\item Any Julia component of complex type is preperiodic (see \cite{PT00});
\item Any unburied Julia component is preperiodic (this follows from no wandering domain theorem \cite{Sullivan85}).
\end{itemize}

Recall that $\mathscr{F}$ consists of critical and post-critical Fatou components.
We say a Julia component $K$ is separating for $\mathscr{F}$ if there exist $V, W \in \mathscr{F}$ that are contained in two different components of $\hat\C-K$.
Let $\mathcal{J}$ be the non-escaping set for the tree map $F:\mathcal{T}_0\longrightarrow \mathcal{T}_1$.
Given $x\in \mathcal{J}$, then the same construction for $x\in \mathcal{V}_0$ gives a Julia component $K_x$ associated to it.
It is easy to see that $K_x$ is separating for $\mathscr{F}$ and any separating Julia component for $\mathscr{F}$ arise in this way.
We have the correspondence:
\begin{lem}\label{lem:JC}
Let $J^{s}$ be the union of separating Julia components for $\mathscr{F}$, there is a semiconjugacy
$$
\phi: J^s \longrightarrow \mathcal{J}
$$
conjugating the dynamics of $f$ and $F$ such the fiber $\phi^{-1}(x)$ is precisely the Julia component $K_x$.
\end{lem}
\begin{proof}
We define $\phi (z) = x$ if $z\in K_x$.
$\phi$ is continuous as $\pi: \hat\C\longrightarrow \mathcal{T}_1$ is continuous.
The fact that it conjugates the dynamics follows from the fact $\pi\circ f (\partial \pi^{-1}(x)) = f_*(x)$.
\end{proof}

Note that the Julia set and Fatou set can be naturally generalized for the rational mapping scheme $R: \hat\C^{\mathcal{V}_0} \longrightarrow \hat\C^{\mathcal{V}_1}$.
For periodic Julia components of $f$, we have
\begin{theorem}\label{lem:PS}
Let $K$ be a non-degenerate periodic Julia component, then it is separating for $\mathscr{F}$, so $\phi(J) \in \mathcal{J}$.

Moreover, if $K$ is of complex type, then $\phi(K) \in \mathcal{V}_0$, and $K$ is homeomorphic to the Julia set of $R$ in $\hat\C_{\phi(K)}$.
\end{theorem}
\begin{proof}
If $K$ is a non-degenerate periodic Julia component of period $p$, then there exist at least two component $D_1, D_2$ of $\hat\C-K$ so that $f^p|_{\partial D_i}$ has degree $\geq 2$. This means that there exists $V_i\in \mathscr{F}$ contained in $D_i$, so $K$ is separating for $\mathscr{F}$.

For the moreover part, we note that if $K$ is of complex type, then it must be periodic. 
Thus there are three components $D_1, D_2, D_3$ of $\hat\C-K$ with $f^p|_{\partial D_i}$ has degree $\geq 2$. So $\phi(K)$ must either be  a branch point or has an exposed critical point. Therefore $\phi(K) \in \mathcal{V}_0$.
It follows from our construction that $K$ is homeomorphic to the Julia set of $R$ in $\hat\C_{\phi(K)}$.
\end{proof}

\begin{cor}\label{thm:ConjugacyPeriodicComponents}
Let $f$ be a hyperbolic rational map with finitely connected Fatou set, and $(F, R)$ be the associated tree mapping scheme. Then non-degenerate periodic Julia components for $f$ are in bijective correspondence with periodic points of $F: \mathcal{T}_0 \longrightarrow \mathcal{T}_1$.
\end{cor}

We say a point $x\in \mathcal{J} \subseteq \mathcal{T}_0$ is {\em buried} in $\mathcal{J}$ if $x$ is not on the boundary of any component of $\mathcal{T}_1 - \mathcal{J}$.
As a corollary, we have
\begin{cor}\label{cor:bjc}
	Let $f$ be a hyperbolic rational map with finitely connected Fatou set.
	A non-degenerate periodic Julia component $K$ is buried if and only if the corresponding point $x \in \mathcal{V}_0$ is buried in the non-escaping set $\mathcal{J} = \bigcap_{n=0}^\infty F^{-n}(\mathcal{T}_0)$ and no exposed critical orbits are attached at $x$.
\end{cor}
\begin{proof}
	Note that the component $K$ does not meet the boundary of a multiply connected Fatou component if and only if $x$ is buried in $\mathcal{J}$, and $K$ does not meet the boundary of a disk Fatou component if and only if no exposed critical orbits are attached at $x$.
	Combining the two, we get the result.
\end{proof}

\section{Cycles of rescaling limits}\label{sec:rl}
In this section, we shall use tree mapping schemes to construct a sequence of rational maps with infinitely many non-monomial rescaling limits.

Let $f_n$ be a sequence of rational maps. 
A {\em period $p$ rescaling} for $f_n$ is a sequence $M_n \in \PSL_2(\C)$ so that $M_n^{-1} \circ f_n^p \circ M_n$ converges compactly away from a finite set to a rational map $g$ of degree $\geq 1$.
The limiting map $g$ is called a {\em rescaling limit of period $p$}.

Two rescalings $M_n, L_n$ are said to be {\em dynamically dependent} if after passing to a subsequence, there exists $1\leq m \leq q$ so that $L_n^{-1} \circ f_n^m \circ M_n \to g_1$ and $M_n^{-1} \circ f_n^{q-m} \circ L_n \to g_2$ for some rational maps $g_1, g_2$ of degree $\geq 1$.
They are called {\em dynamically independent} otherwise (see \cite{Kiwi15}).

In \cite{CP19}, it is shown that a rational map of degree $d$ may have an arbitrarily large number of periodic cycles of Julia components of complex type.
Such a map arises naturally as we perform surgery on periodic points on the tree mapping scheme, which is in the same spirit as the self-grafting construction in \cite{CP19}.
We are not aiming to give the most general construction as the construction is very flexible and there are many different variations. 

We first present the idea of the construction with the following example.
Let $g_n(z) = z^3+\frac{1}{nz^3}$.
Then for all sufficiently large $n$, the Julia set is a Cantor set of circles (see \cite{McM88}).
It corresponds to the following tree mapping scheme (see Figure~\ref{fig:CantorCircleExample1}):
$$
\mathcal{T}_{1,F_0} = [A, A'] \cong [0,1],
$$
and 
$$
\mathcal{T}_{0, F_0} = [A, B]\cup [B', A'] \cong [0, \frac{1}{3}] \cup [\frac{2}{3},1].
$$
\begin{figure}[ht]
    \centering
    \resizebox{0.6\linewidth}{!}{
    \def\svgwidth{\columnwidth}
    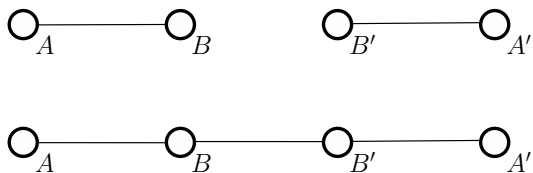

    }
    \caption{The tree map for $f_n(z) = z^3+\frac{1}{nz^3}$. }
    \label{fig:CantorCircleExample1}
\end{figure}

The tree map is a tent map:
$$
F_0(x)=
\begin{cases}
3x, \text{ for } x\in [0, \frac{1}{2}]\\
3(1-x), \text{ for } x\in [\frac{2}{3},1]
\end{cases}.
$$
We mark the Riemann spheres so that the tangent vector pointing towards (respectively, away from) $A$ corresponds to $\infty$ (respectively, $0$). 
Under this marking, the rational mapping schemes associated to $g_n(z) = z^3+\frac{1}{nz^3}$ is
\begin{align*}
R_{0,A\to A}(z) = R_{0,B\to A'}(z) &= z^3,\\
R_{0,A'\to A}(z) = R_{0,B'\to A'}(z) &= \frac{1}{z^3}.
\end{align*}

Note that the point $A_0 = \frac{3}{4}$ is fixed by $F$ and $F$ sends $[B', A_0]$ onto $[A', A_0]$, and $[A_0, A']$ onto $[A_0, A]$.
We now construct a new tree mapping scheme $(F_1, R_1)$ by performing surgery on $(F_0, R_0)$.
Intuitively, the surgery construction will replace a Jordan curve Julia component by a more complicated Julia component.

First, we construct the tree $\mathcal{T}_{1, F_1}$ by attaching a new arc $[A_0, D]$ of the same length of $[C, A']$ at $A_0$ to $\mathcal{T}_{1, F_0}$ (see Figure \ref{fig:SurgeryOnFixedPoint}).
The new tree map $F_1$ sends $[A, B]$ and $[A', B']$ homeomorphically onto $[A,D]$ by expanding with a factor of $3$, and sends $[A_0, D]$ isometrically onto $[A_0, A']$.

\begin{figure}[ht]
    \centering
    \resizebox{0.6\linewidth}{!}{
    \def\svgwidth{\columnwidth}
    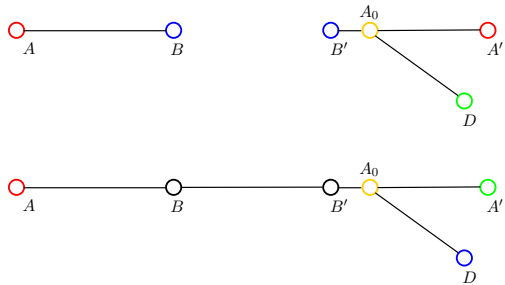

    }
    \caption{The tree map after the surgery on the fixed point $C$. The dynamics of $F_1$ on vertices are labeled by colors.}
    \label{fig:SurgeryOnFixedPoint}
\end{figure}

Let $\zeta_8$ be the primitive $8$-th root of unity. 
We mark the Riemann sphere $\hat\C_{A_0}$ so that $0_{A_0}, \infty_{A_0}, (\zeta_8)_{A_0}$ correspond to the direction of $A', A, D$ respectively.
In the same way, the other Riemann spheres are marked so that the tangent vector pointing towards (respectively, away from) $A$ corresponds to $\infty$ (respectively, $0$). 
The surgery introduces a post-critically finite hyperbolic rational map
$$
R_{1,A_0\to A_0}(z) = \frac{1}{z^3}+\zeta_8,
$$
with critical orbit
$$
0_{A_0}\to \infty_{A_0} \to (\zeta_8)_{A_0} \to 0_{A_0}.
$$
We modify the rational mapping scheme $R$ accordingly:
\begin{align*}
R_{1,A\to A}(z) = R_{1,B\to D}(z) &= z^3,\\
R_{1,A'\to A}(z) = R_{1,B'\to D}(z) &= \frac{1}{z^3},\\
R_{1,D\to A'}(z) &= z.
\end{align*}

\begin{figure}[ht]
    \centering
    \begin{subfigure}{0.7\textwidth}
        \centering
        \includegraphics[width=\textwidth]{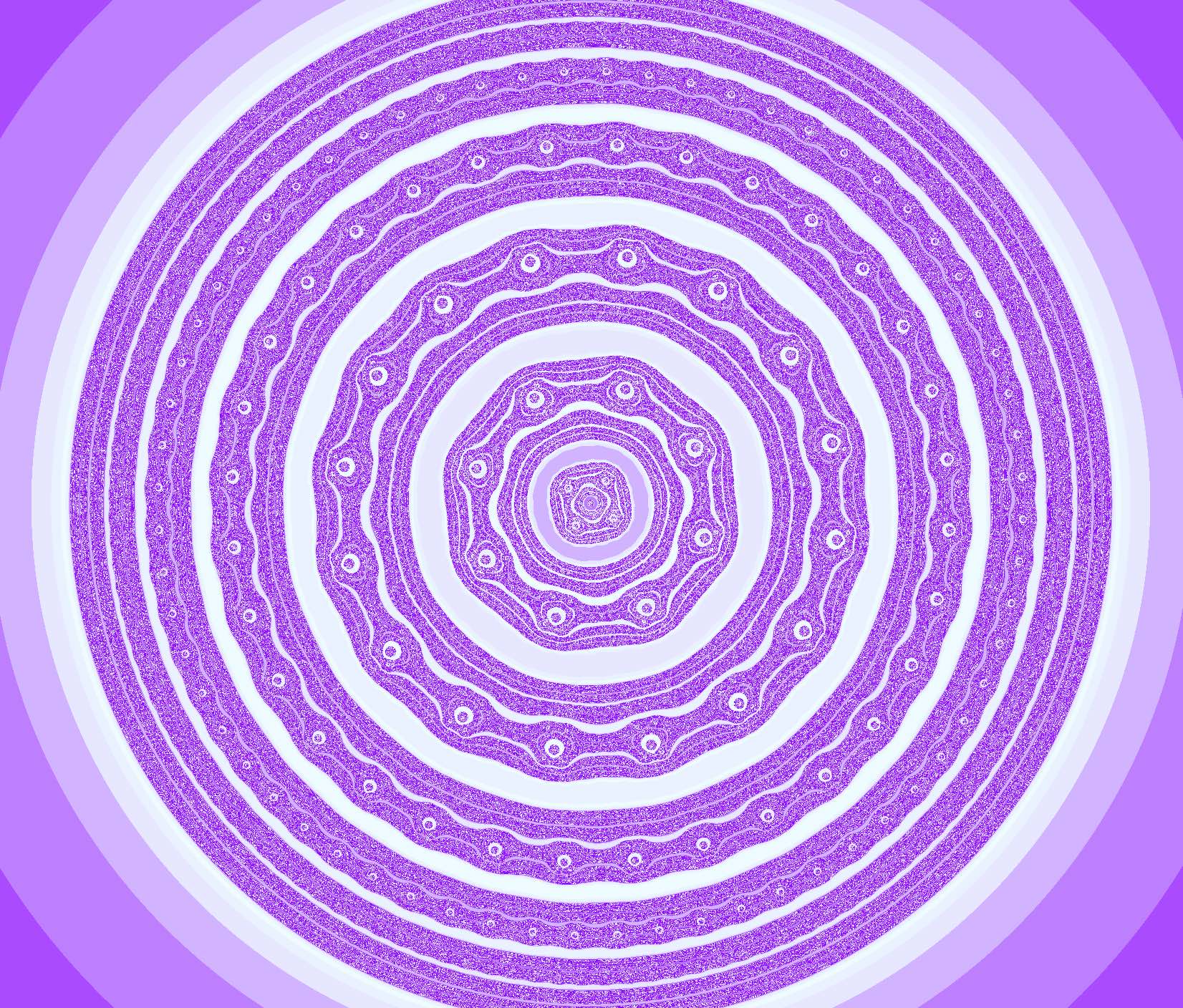}
        \caption{The Julia set of the hyperbolic rational map after performing the surgery. The algebraic formula is given by $g_n(z) = z^3+\frac{1}{nz^3}+\frac{\zeta_8}{n^{1/4}}$ for sufficiently large $n$.}
    \end{subfigure}
    \begin{subfigure}{0.9\textwidth}
        \centering
        \includegraphics[width=0.4\textwidth]{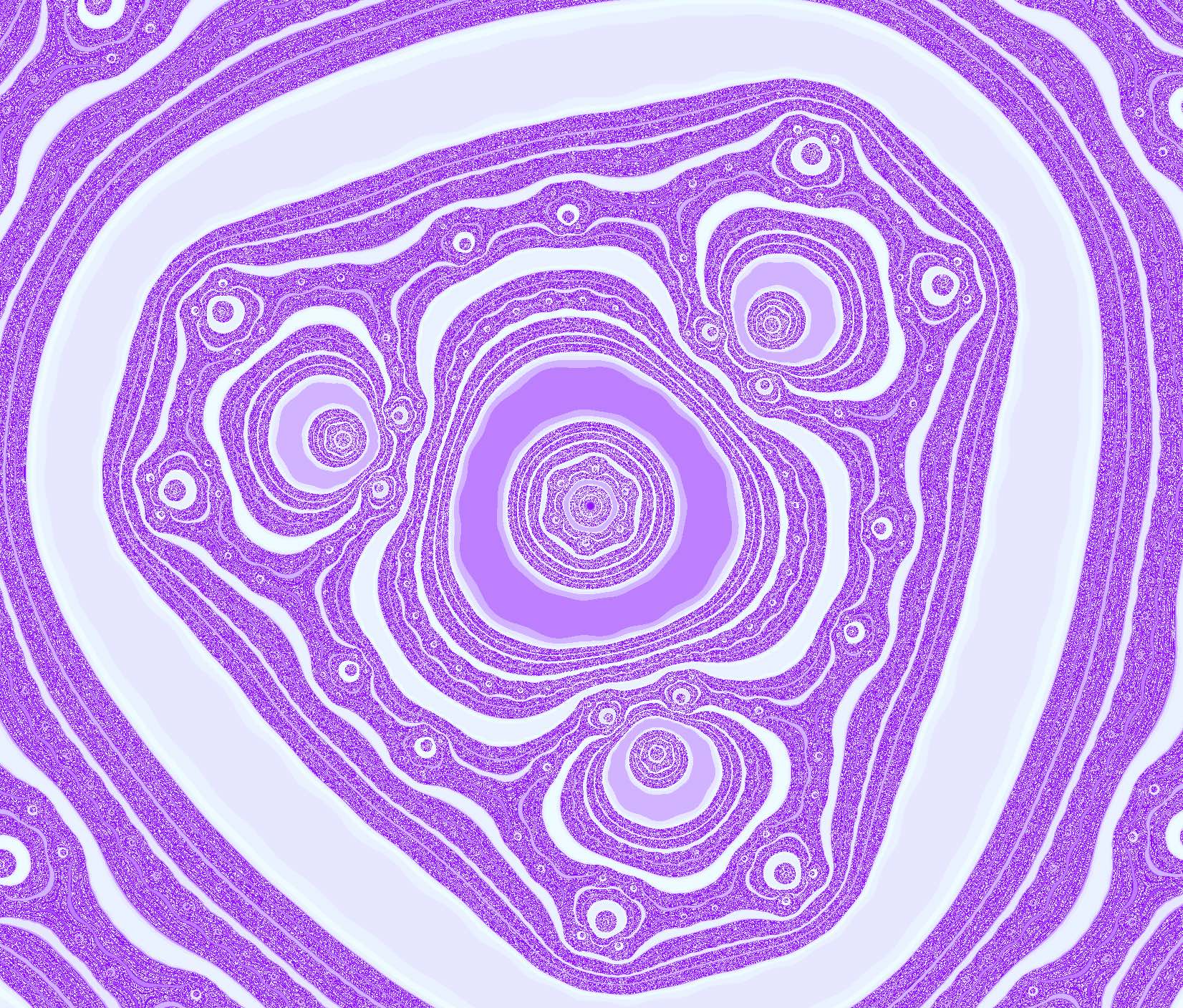}
        \includegraphics[width=0.4\textwidth]{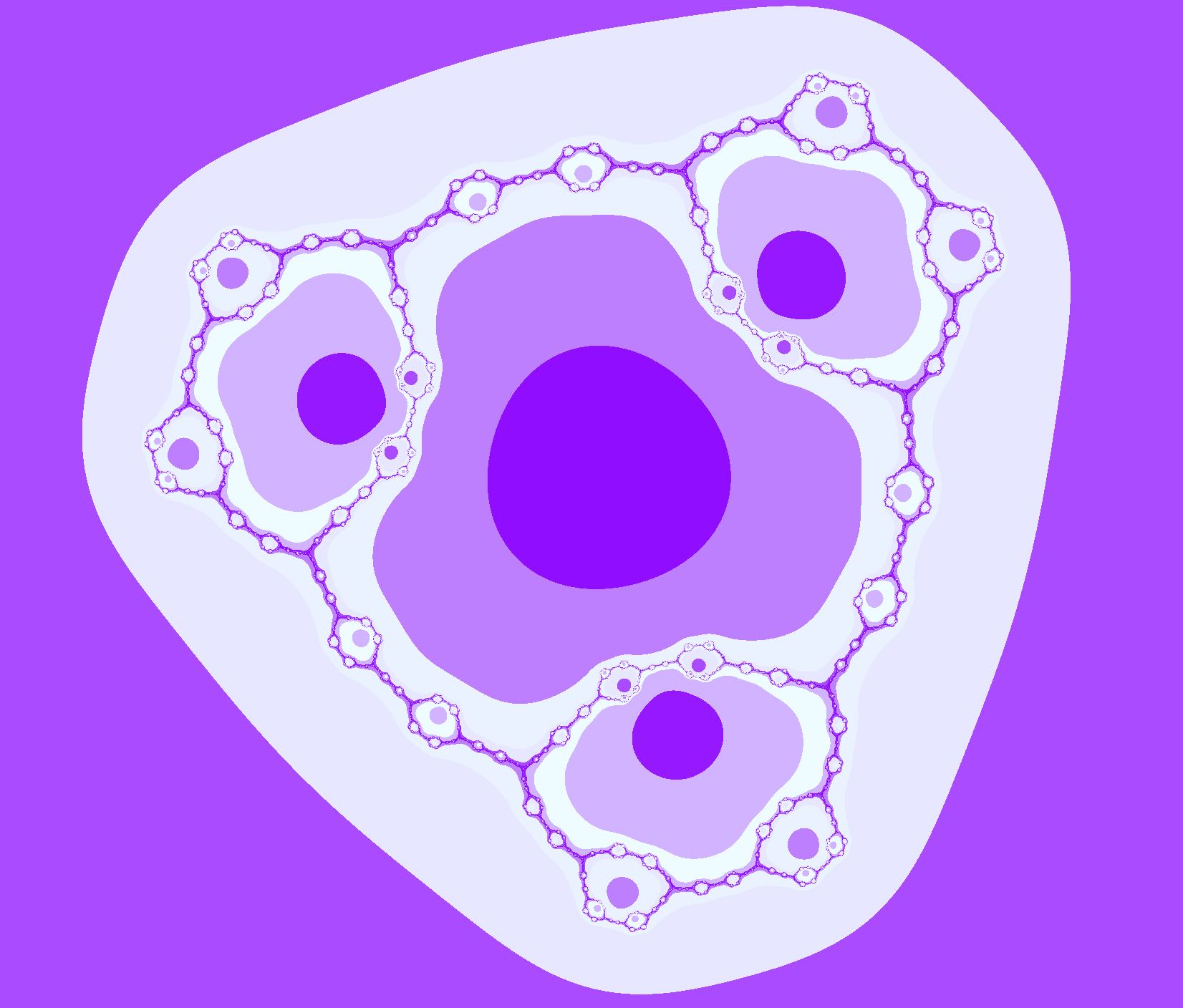}
        \caption{A zoom of the above Julia set on the left. It replaces the fixed buried Julia component which is a Jordan curve in the original Cantor circle example to the Julia set of $\frac{1}{z^3}+\zeta_8$ (on the right).}
    \end{subfigure}
    \caption{}
    \label{fig:JuliaSurgeryOnFixedPoint}
\end{figure}

We will now iterate this construction to get a sequence of tree mapping schemes $(F_k, R_k)$.
We shall identify the tree $\mathcal{T}_{1,F_i}$ as a subtree in $\mathcal{T}_{1, F_j}$ for all $j\geq i$ as the trees $\mathcal{T}_{1, F_j}$ will be constructed by adding more branches at periodic points to $\mathcal{T}_{1, F_{j-1}}$.
Through out the construction, the Riemann spheres for unbranched points are always marked so that the tangent vector pointing towards (respectively, away from) $A$ corresponds to $\infty$ (respectively, $0$). 
Let $A_0= \frac{3}{4}$ and $A_m \in \mathcal{T}_{1, F_0} = [0,1]$ be the smallest periodic point of period $2^m$ under $F_0$ that is larger than $A_{m-1}$.
Note that $R_0^{2^m}|_{\hat\C_{A_m}}(z) = \frac{1}{z^{3^{2^m}}}$.

Let $(F_k, R_k)$ be constructed. 
By induction, we have $A_k$ has period $q=3^k$ under $F_k$,
we construct $\mathcal{T}_{1, F_{k+1}}$ by attaching an edge $S_j$ of the same length as $[A_k, A']$ at $F^j_k(A_k)$ to $\mathcal{T}_{1, F_k}$ where $j=0,..., q-1$:
$$
\mathcal{T}_{1,F_{k+1}} = \mathcal{T}_{1,F_k} \cup \bigcup_{j=0}^{q-1} S_j.
$$
Denote $X_j = F^{j}_k(A_k)$.
By induction, $X_j$ is not a branch point for $\mathcal{T}_{1, F_k}$.
Thus $R_{k, X_j\to X_{j+1}} = z^l$ for some $l \in \mathbb{Z}$.
Consider $h_j(z) = R_{k, X_j \to X_{j+1}} (z)$ for $j<q-1$ and $h_{q-1}(z) = R_{k, X_{q-1} \to X_0} (z) +a_k$,
Then we have
$$
H(z) := h_{q-1}\circ h_{q-2} \circ... \circ h_0(z) = \frac{1}{z^{3^{2^k}}} + a_k.
$$ 
We choose $a_k$ so that $H(z)$ has the post-critical set $0\to \infty \to a_k \to 0$.
In the new tree $\mathcal{T}_{1, F_{k+1}}$, we mark the Riemann spheres at $X_j$ so that the edge $S_j$ corresponds to $h_{j-1}\circ ... \circ h_1(a_k) \in \hat\C_{X_j}$.
Denote $S = [A_k, A']$, and let $I_j: S \longrightarrow S_j$ be the isometry with $I_j(A_k) = X_j$.
We define the tree map $F_{k+1}$ as follows:
\begin{itemize}
\item If $x\in \mathcal{T}_{1, F_k}$ and $F_k(x) \notin S$, then $F_{k+1}(x) = F_k(x)$;
\item If $x\in \mathcal{T}_{1, F_k}$ and $F_k(x) \in S$, then $F_{k+1}(x) = I_0\circ F_k(x)$;
\item If $x\in S_j$ and $j\neq q-1$, then $F_{k+1}(x) = I_{j+1}\circ I_{j}^{-1} (x)$;
\item If $x\in S_{q-1}$, then $F_{k+1}(x) = I_{q-1}^{-1}(x)$.
\end{itemize}
If $x$ is a branch point of $\mathcal{T}_{1, F_k}$, the rational mapping scheme on $\hat\C_x$ stays the same.
At a new branch point $X_j$, we define
$$
R_{k+1, X_j \to X_{j+1}}(z) = h_j(z).
$$
At an end point $X$ of $\mathcal{T}_{0, F_{k+1}}$, the rational mapping scheme is defined as a monomial respecting the marking.

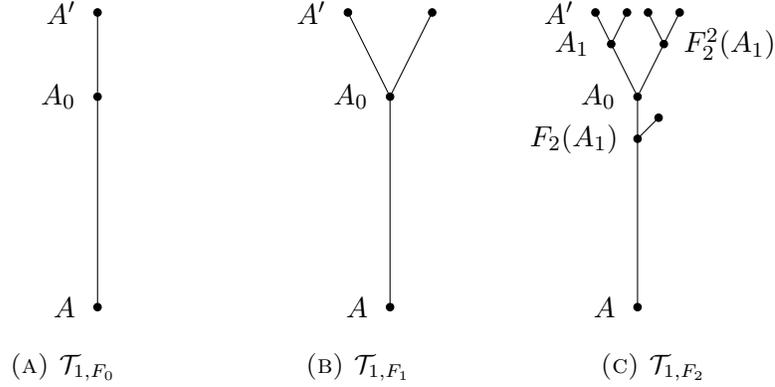
\begin{figure}[!htb]

\begin{subfigure}[t]{.3\textwidth}
  \centering

\begin{tikzpicture}[scale=1.4]

\tikzstyle{every node}=[draw,shape=circle]

\draw[] (0,0) node[circle,fill,inner sep=1pt,label=left:$A $](a){} -- (0,2) 
node[circle,fill,inner sep=1pt,label=left:$A_0$](b){};

\draw[] (b) -- (0,2.8) node[circle,fill, inner sep=1 pt, label=left:$A'$](c){};
\end{tikzpicture}
\caption{$\mathcal{T}_{1,F_0}$} \label{fig:T_-1}

\end{subfigure}
\begin{subfigure}[t]{.3\textwidth}
  \centering

\begin{tikzpicture}[scale=1.4]

\tikzstyle{every node}=[draw,shape=circle]

\draw[] (0,0) node[circle,fill,inner sep=1pt,label=left:$A$](a){} -- (0,2) 
node[circle,fill,inner sep=1pt,label=left:$A_0$](b){};

\draw[] (b) -- (-0.4,2.8) node[circle,fill, inner sep=1 pt, label=left:$A'$](c){};
\draw[] (b) -- (0.4,2.8) node[circle,fill, inner sep=1 pt](d){};
\end{tikzpicture}
\caption{$\mathcal{T}_{1,F_1}$} \label{fig:T_0}

\end{subfigure}
\begin{subfigure}[t]{.3\textwidth}
  \centering

\begin{tikzpicture}[scale=1.4]

\tikzstyle{every node}=[draw,shape=circle]

\draw[] (0,0) node[circle,fill,inner sep=1pt,label=left:$A$](a){} -- (0,2) 
node[circle,fill,inner sep=1pt,label=left:$A_0$](b){};

\draw[] (b) -- (-0.4,2.8) node[circle,fill, inner sep=1 pt, label=left:$A'$](c){};
\draw[] (b) -- (0.4,2.8) node[circle,fill, inner sep=1 pt](d){};

\draw[] (-0.25,2.5) node[circle,fill,inner sep=1pt,label=left:$A_1$](e){} -- (-0.1,2.8) 
node[circle,fill,inner sep=1pt](f){};

\draw[] (0.25,2.5) node[circle,fill,inner sep=1pt,label=right:$F_2^2(A_1)$](g){} -- (0.1,2.8) 
node[circle,fill,inner sep=1pt](h){};

\draw[] (0,1.6) node[circle,fill,inner sep=1pt,label=left:$F_2(A_1)$](j){} -- (0.2,1.8) 
node[circle,fill,inner sep=1pt](k){};

\node[left,,draw=none,scale=0.8] at (-0.22,2.65) {$$};
\node[left,,draw=none,scale=0.8] at (0.08,1) {$$};
\node[right,,draw=none,scale=0.8] at (0.05,2.22) {$$};
\node[left,,draw=none,scale=0.8] at (-0.05,2.22) {$$};
\node[left,,draw=none,scale=0.8] at (0.08,1.8) {$$};
\node[right,,draw=none,scale=0.8] at (0.22,2.65) {$$};
\node[right,,draw=none,scale=0.8] at (-0.3,2.65) {$$};
\node[right,,draw=none,scale=0.8] at (0,1.6) {$$};
\node[left,,draw=none,scale=0.8] at (0.3,2.65) {$$};

\end{tikzpicture}
\caption{$\mathcal{T}_{1,F_2}$} \label{fig:T_1}

\end{subfigure}

\caption{A schematic picture of the first three construction of $\mathcal{T}_{1,G_n}$.
At each step, we add one additional edge on the orbit of $a_k$ with period $3^k$.
The orbit of $a_k$ intersects each edge exactly once. The $k$-th step has exactly $3^{k+1}$ edges.} 
\label{fig:T_1}

\end{figure}

From the construction, it is easy to verify that $(F_{k+1}, R_{k+1})$ is hyperbolic, post-critically finite and irreducible.
By induction, we can verify that $A_m$ has period
$$ 
p_{k,m}:=
  \begin{cases}
                                   3^m& \text{if $m\leq k$} \\
                                  2^{m-k}\cdot 3^{k} & \text{if $m>k$}
  \end{cases}.
$$
Moreover, with the marking defined as above, the first return of the rational mapping scheme $R_k$ at $A_m$ is
$$
R_k^{p_{k,m}}|_{\hat\C_{A_m}}(z) = 
 \begin{cases}
                                   \frac{1}{z^{3^{2^m}}}+a_m & \text{if $m< k$} \\
                                   \frac{1}{z^{3^{2^m}}} & \text{if $m\geq k$}
  \end{cases},
$$
where the post-critical set is $0\to \infty\to a_m\to 0$ for $m<k$ or $0\to \infty \to 0$ for $m\geq k$.
Note that $\frac{1}{z^{3^{2^m}}}+a_m$ is conjugate to $1-\frac{1}{z^{3^{2^m}}}$ with post-critical set $0\to\infty\to 1 \to 0$.
Thus, to summarize, we have
\begin{prop}\label{prop:krl}
For each $k$, there exists an irreducible post-critically finite hyperbolic tree mapping scheme $(F_k, R_k)$ of degree $6$ with periodic points $x_m \in \mathcal{T}_0$ of periodic $p_m = 3^m$, $m=0,..., k-1$, so that $R_{x_m}^{p_m} (z) = 1-\frac{1}{z^{3^{2^m}}}$.
\end{prop}

\subsection*{Double infinite sequences and diagonal arguments.}
By Theorem \ref{thm:classification}, there exists a sequence of rational maps $g_{k,n}$ of degree $6$ converging to $(F_k, R_k)$ as in Proposition \ref{prop:krl}.
Thus given $m < k$, we have a sequence of rescalings $M_{k,m,n} \in \PSL_2(\C)$ so that
$$
\lim_n M_{k,m,n}^{-1} \circ g_{k,n}^{p_{k,m}} \circ M_{k,m,n}(z) = 1-\frac{1}{z^{3^{2^m}}}.
$$
By our construction, the holes are contained in the marked set $\{0, 1, \infty\}$.

Let $K_j$ be an exhaustion of compact sets for $\hat\C-\{0,1,\infty\}$.
Denote the metric generated by the sup norm (with respect to the standard spherical metric) on $K_j$ for a pair of continuous functions by
$$
d_j(f,g) := \sup_{x\in K_j}\{d_{S^2}(f(x), g(x))\}.
$$
Then for any fixed $j$ and $m<k$, 
$$
\lim_n d_j(M_{k,m,n}^{-1} \circ g_{k,n}^{p_{k,m}} \circ M_{k,m,n}, 1-\frac{1}{z^{3^{2^m}}}) = 0.
$$

We construct the sequence $h_k = g_{k, n(k)}$ where $n(k)$ is chosen so that
$$
d_j(M_{m,k,n(k)}^{-1} \circ g_{k,n(k)}^{p_{k,m}} \circ M_{m,k,n(k)}, 1-\frac{1}{z^{3^{2^m}}}) < 1/k
$$
for all $m< k$ and all $j< k$.
Let $M_{k,m}:= M_{k,m,n(k)}$ and $p_m := 3^m$. Note that $p_m = p_{k,m}$ for all sufficiently large $k$.
\begin{prop}
The sequence $h_k$ has infinitely many dynamically independent non-monomial rescaling limits.
\end{prop}
\begin{proof}
By construction, for any fixed $m$ and $j$,
$$
\lim_k d_j(M_{k,m}^{-1} \circ h_k^{p_m} \circ M_{k,m}, 1-\frac{1}{z^{3^{2^m}}}) = 0.
$$
Therefore,
$$
\lim_k M_{k,m}^{-1} \circ h_k^{p_m} \circ M_{k,m}(z) = 1-\frac{1}{z^{3^{2^m}}}.
$$
Therefore, for any $m$, $M_{k,m}$ is a periodic rescaling of for the sequence $h_k$.

Note for different $m$, $M_{k,m}$ represent dynamically independent periodic rescalings.
Indeed, if $M_{k,m}$ were dynamically dependent to $M_{k,l}$ with $l> m$, then after possibly passing to a subsequence, $\lim_k M_{k,m}^{-1} \circ h_k^{p_l} \circ M_{k,m}(z)$ is conjugate to $1-\frac{1}{z^{3^{2^l}}}$.
This means the $p_l/p_m$ iterate of $1-\frac{1}{z^{3^{2^m}}}$ is conjugate to $1-\frac{1}{z^{3^{2^l}}}$, which is not possible.
Thus, we have infinitely many dynamically independent cycles of non-monomial rescaling limits for $h_k$.
\end{proof}

\subsection*{General construction for arbitrary degrees}
We now consider the more general case with any degrees. 
The construction is essentially the same as in the above example.
\begin{proof}[Proof of Theorem \ref{thm:infiniterescalinglimit}]
To give the construction at any degree $d\geq 3$, we first consider a degree $3$ tree mapping scheme.
This construction, which is a modification of Godillon's original construction in \cite{Godillon15}.
\begin{figure}[ht]
    \includegraphics[width=0.4\textwidth]{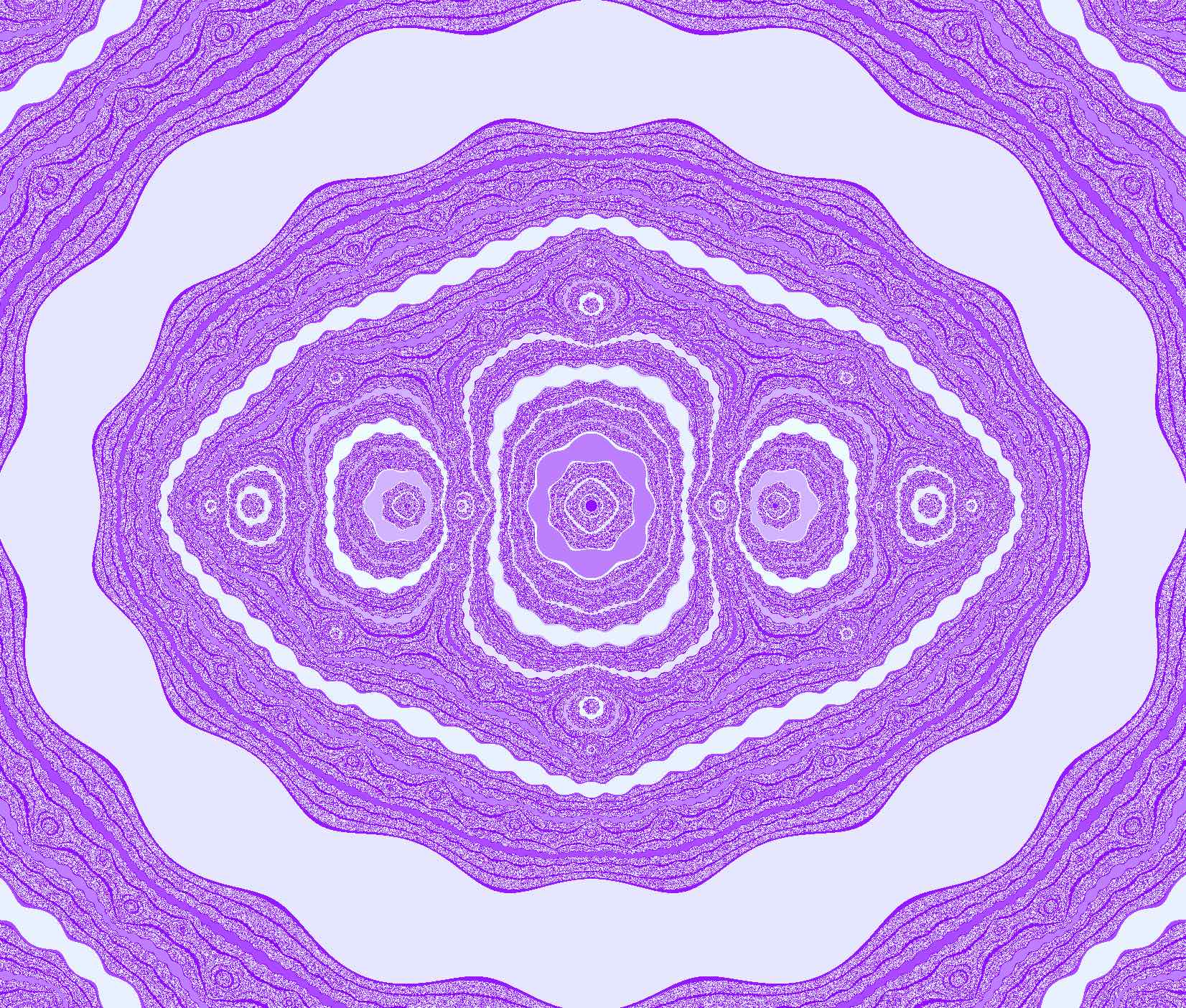}
    \includegraphics[width=0.4\textwidth]{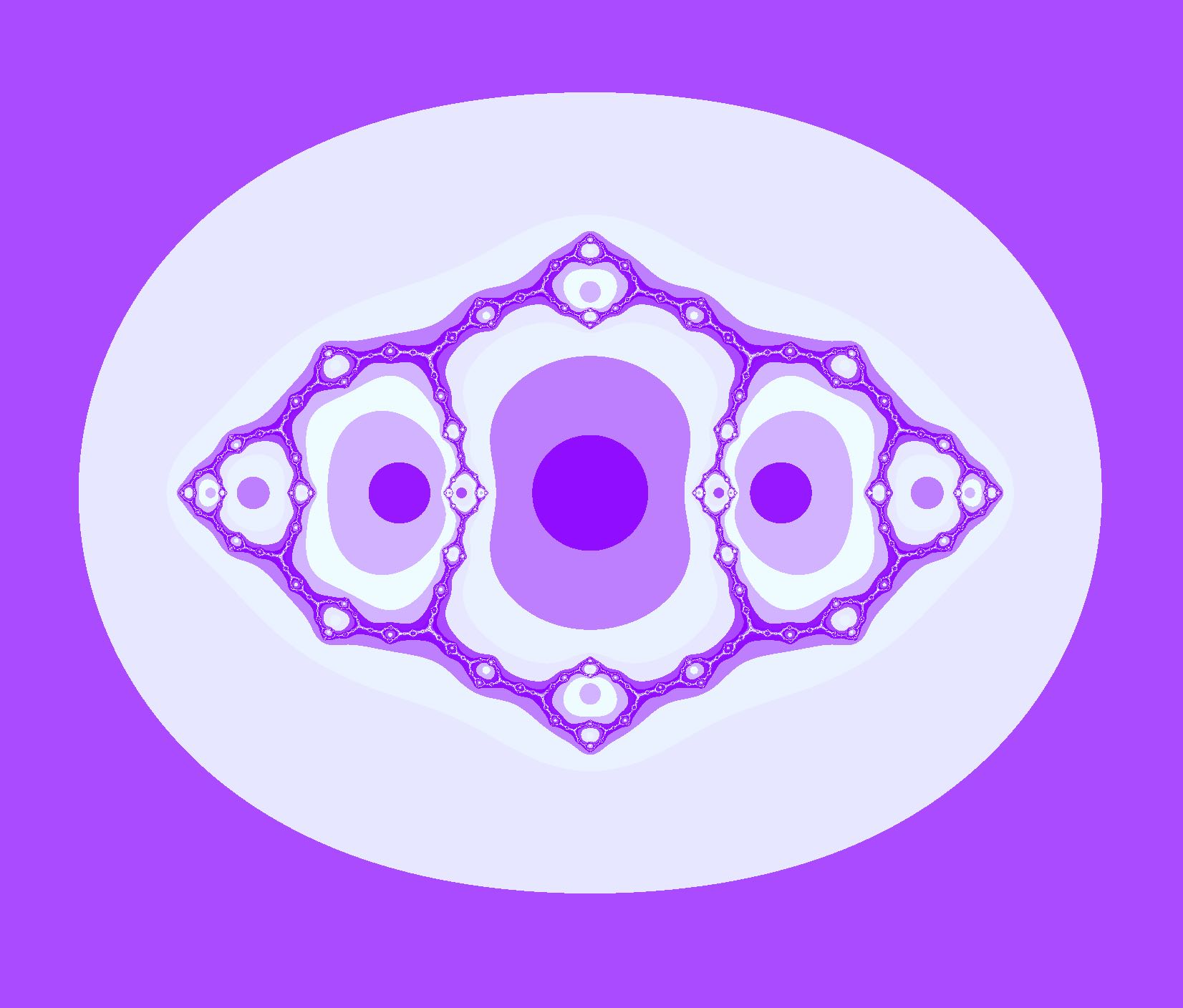}
    \caption{The Julia set of the degree 3 example on the left. The algebraic formula is given by $f_n(z) = \frac{1}{(z-1)^2}+\frac{1}{1-nz}$ for sufficiently large $n$. The Julia set contains a buried component that is the Julia set of $\frac{1}{(z-1)^2}$ (on the right).}
    \label{fig:JuliaSetOfDegree3Example}
\end{figure}

\begin{figure}[ht]
    \centering
    \resizebox{0.8\linewidth}{!}{
    \def\svgwidth{\columnwidth}
    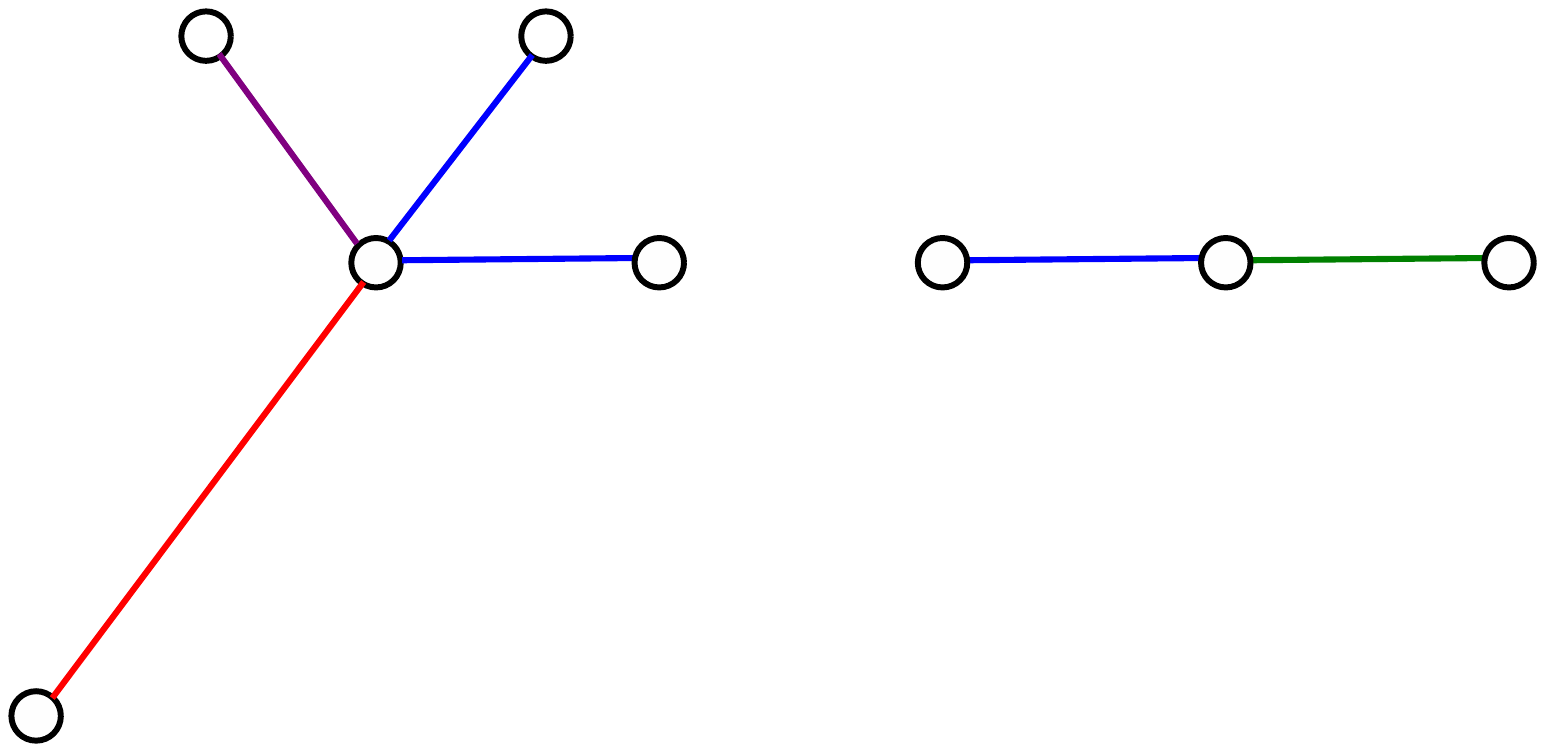

    \def\svgwidth{\columnwidth}
    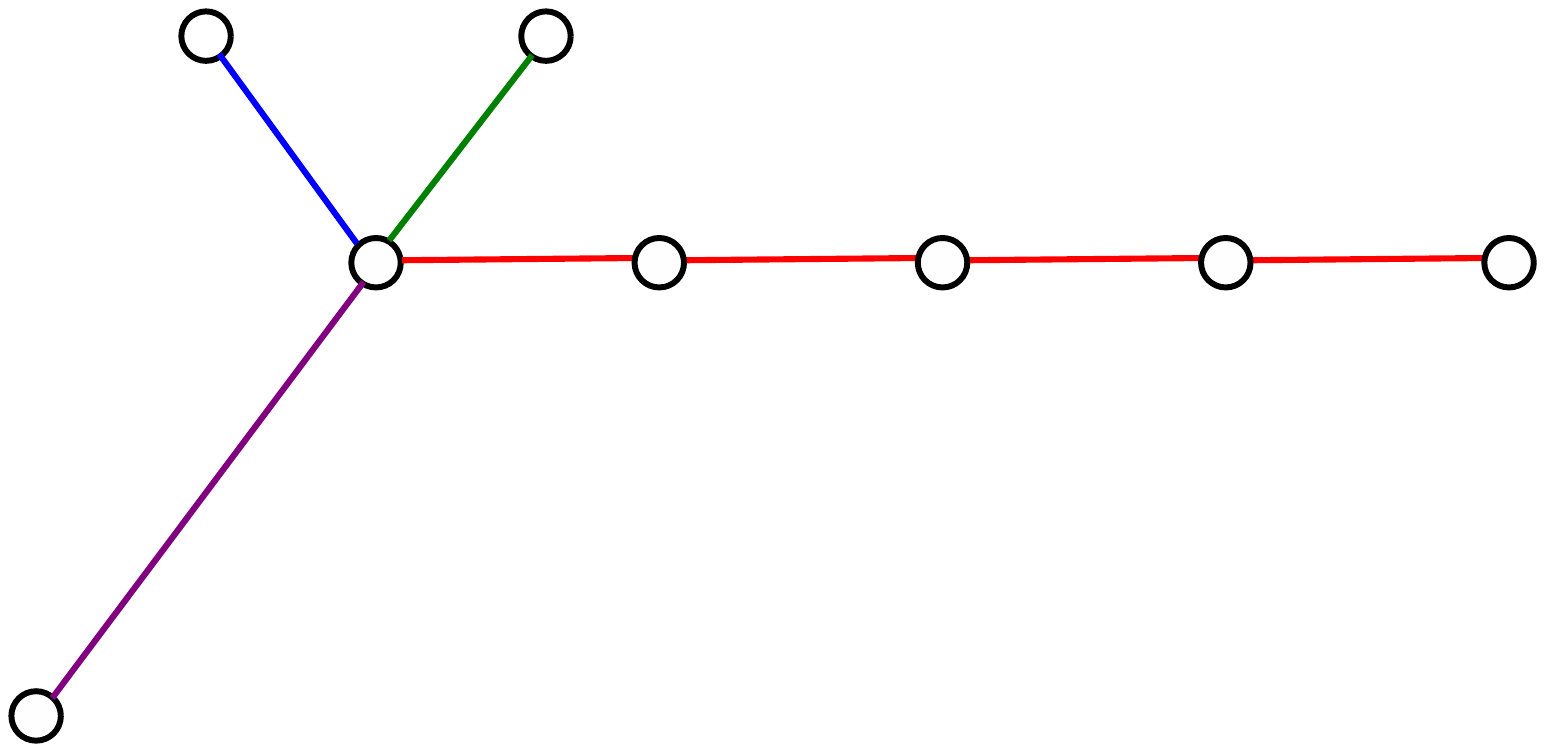

    }
    \caption{The tree map for the degree 3 example.}
    \label{fig:Degree3Example}
\end{figure}

The trees $\mathcal{T}_0 = \mathcal{T}_1-\mathcal{U}$ (on the left) and $\mathcal{T}_1$ (on the right) are shown in Figure \ref{fig:Degree3Example}, where $\mathcal{U}$ is the open interval $(B,B')$.
The vertices $\mathcal{V}_1=\{X_1, X_2, X_3, X_4, B, B', A, A'\}$.
Each edge has length $1$, except for the edge $AX_2$ of length $2$.
The dynamics $F$ is recorded by the color. 
More precisely, $F$ sends $A, A'$ to $A$, $B, B'$ to $X_1$, and $X_i$ to $X_{i+1}$.
$F$ is expanding by a factor of $2$ on $AX_1$ and $AX_2$, and it is local isometry on the rest of the edges.

The Riemann sphere $\hat\C_A$ is marked so that the tangent vectors 
$$
AB, AX_1, AX_2, AX_4
$$ 
correspond to $0, 1, \infty, 2 \in \hat\C_A$ (see Figure \ref{fig:Degree3Example}).
All the other Riemann spheres are marked so that the tangent vector pointing towards (or away from) $A$ corresponds to $\infty$ (or $0$ respectively).
the rational mapping scheme is given as:
\begin{align*}
R_{A\to A}(z) &= \frac{1}{(z-1)^2},\\
R_{A'\to A}(z) &= 1+\frac{1}{1-z},\\
R_{X_1\to X_2}(z) = R_{X_2\to X_3}(z) &= z^2,\\
R_{B\to X_1}(z) = R_{X_3\to X_4}(z) = R_{X_4\to X_1}(z) &= z,\\
R_{B'\to X_1}(z) &= \frac{1}{z}.
\end{align*}

It is easy to verify that the above gives a post-critically finite hyperbolic tree mapping scheme.
Since $F^3$ sends $AB$ and $A'B'$ homeomorphically to $AX_3$, there exist $C\in AB$ and $C'\in A'B'$ so that $F^3: AC \cup C'A' \longrightarrow AA'$ is a tent map with derivative $\pm 4$.
We can perform a surgery on periodic points on $AA'$ in a similar way as in the previous subsection, and a similar diagonal argument gives a sequence $h_n$ of degree $3$ with infinitely many dynamically independent cycles of non-monomial rescaling limits.

To construct the sequence for any degree $d>3$, we can modify the rational mapping scheme at $X_3$ with
$$
R_{X_3\to X_4}(z) = \frac{(\frac{z}{z-1})^k}{(\frac{z}{z-1})^k-1} = \frac{z^k}{z^k-(z-1)^k},
$$
while keeping the others the same.
In this way, we introduce two exposed critical points $0_{X_3}, 1_{X_3}$ with multiplicity $k-1$.
It is easy to verify that the modified tree mapping scheme is again post-critically finite hyperbolic, and the corresponding rational map has degree $2+k$.
Now a similar construction gives a sequence $h_n$ of any degree $d\geq 3$ with infinitely many dynamically independent cycles of non-monomial rescaling limits.
\end{proof}


\end{document}